\documentclass{amsart}

\pdfoutput=1

\usepackage[T1]{fontenc}
\usepackage{lmodern}
\usepackage[utf8]{inputenc}
\usepackage[mathscr]{eucal}%
\usepackage{amssymb}%
\usepackage{stmaryrd}
\usepackage[usenames,dvipsnames]{xcolor}
\usepackage{xspace}%
\usepackage{pifont}%
\usepackage{amsthm}%
\usepackage{bbold}%
\usepackage[normalem]{ulem}%
\usepackage{float}
\usepackage{adjustbox}
\usepackage{enumitem}%
\setlist{leftmargin=*, wide, labelindent=0pt}
\setlist[enumerate]{label*=(\alph*),ref=\alph*}

\usepackage{array}%
\usepackage{amsmath}%
\usepackage{wasysym}%
\usepackage{etoolbox}%
\usepackage{mathdots}%

\usepackage[colorlinks=true,linkcolor={Black},citecolor={Black},urlcolor={Black}]{hyperref}
\usepackage[nameinlink]{cleveref}%

\numberwithin{equation}{section}%
\crefname{Thm}{Theorem}{Theorems}
\crefname{Rem}{Remark}{Remarks}
\crefname{Prop}{Proposition}{Propositions}
\crefname{Cons}{Construction}{Constructions}
\crefname{Cor}{Corollary}{Corollaries}
\crefname{Exa}{Example}{Examples}
\crefname{Lem}{Lemma}{Lemmas}

\usepackage[all]{xy}
\SelectTips{cm}{}
\newdir{ >}{{}*!/-10pt/\dir{>}}

\usepackage{tikz}

\definecolor{azure(colorwheel)}{rgb}{0.0, 0.5, 1.0}
\definecolor{indigo}{rgb}{75, 0, 130}

\swapnumbers %
\newtheorem{Cor}[equation]{Corollary}
\newtheorem{Lem}[equation]{Lemma}

\newtheorem{Prop}[equation]{Proposition}
\newtheorem{Thm}[equation]{Theorem}

\theoremstyle{remark}
\newtheorem{Cons}[equation]{Construction}
\newtheorem{Def}[equation]{Definition}
\newtheorem{Not}[equation]{Notation}
\newtheorem{Exa}[equation]{Example}
\newtheorem{Exas}[equation]{Examples}

\newtheorem{Rem}[equation]{Remark}
\newtheorem{Rec}[equation]{Recollection}

\newtheorem*{Ack}{Acknowledgements}

\newtheorem{Ter}[equation]{Terminology}
\newtheorem{Conv}[equation]{Convention}

\newtheorem{War}[equation]{Warning}

\let\le=\leqslant%

\newcommand{\nc}{\newcommand}
\nc{\dmo}{\DeclareMathOperator}

\dmo{\Aut}{Aut}
\dmo{\add}{add}
\dmo{\Add}{Add}
\dmo{\Chain}{Ch}%
\nc{\CohMack}[1][]{\mathop{\mathrm{Mack^{coh}}_{#1}}}
\dmo{\coker}{coker}
\dmo{\cone}{cone}
\dmo{\Der}{D}%
\dmo{\DMack}{DMack}%
\dmo{\Hm}{H}%
\dmo{\Hom}{Hom}
\dmo{\Hty}{\Kb}%
\dmo{\Id}{Id}
\dmo{\incl}{incl}
\dmo{\Ind}{Ind}
\dmo{\Infl}{Infl}
\dmo{\Ker}{Ker}
\dmo{\modname}{mod}%
\dmo{\Mod}{Mod}%
\dmo{\opname}{op}
\dmo{\perm}{perm}%
\dmo{\Perm}{Perm}%
\dmo{\quot}{quot}%
\dmo{\SH}{SH}%
\dmo{\smallperf}{perf}%
\dmo{\Rel}{Rel}
\dmo{\Spc}{Spc}
\dmo{\Spec}{Spec}
\dmo{\Spech}{Spec^{{\rm h}}}
\dmo{\stab}{stab}%
\dmo{\Stab}{Stab}%
\dmo{\supp}{supp}
\dmo{\thick}{thick}
\dmo{\Tr}{Tr}
\dmo{\Locname}{Loc}%

\nc{\ababs}{{\sl ab absurdo}\xspace}
\nc{\ac}{\mathrm{ac}}%
\nc{\Inj}{\mathrm{Inj}}%
\nc{\calO}{\mathcal{O}}
\nc{\cat}[1]{\mathscr{#1}}%
\nc{\cc}{\mathsf{c}}%
\nc{\cA}{\cat{A}}
\nc{\cB}{\cat{B}}
\nc{\cC}{\cat{C}}
\nc{\cD}{\cat{D}}
\nc{\cF}{\cat{F}}
\nc{\cH}{\cat{H}}
\nc{\cI}{\cat{I}}
\nc{\cJ}{\cat{J}}
\nc{\cK}{\cat{K}}
\nc{\cL}{\cat{L}}
\nc{\cM}{\cat{M}}
\nc{\colim}{\mathop{\mathrm{colim}}}
\nc{\hocolim}{\mathop{\mathrm{hocolim}}}
\nc{\cN}{\cat{N}}
\nc{\cP}{\cat{P}}
\nc{\cQ}{\cat{Q}}
\nc{\cS}{\cat{S}}
\nc{\cT}{\cat{T}}
\nc{\Dperf}{\Der_{\smallperf}}%
\nc{\eg}{{\sl e.g.}\@\xspace}
\nc{\gp}{\mathfrak{p}}%
\nc{\gq}{\mathfrak{q}}%
\nc{\Homcat}[1]{\Hom_{\cat #1}}
\nc{\hook}{\hookrightarrow}
\nc{\ie}{{\sl i.e.}\@\xspace}
\nc{\into}{\mathop{\rightarrowtail}}
\nc{\inv}{^{-1}}
\nc{\kk}{k}%
\nc{\kkG}{\kk G}%
\nc{\Loc}[1]{\Locname(#1)}
\nc{\maxel}{\operatorname{Max}_{p\textup{-sec}}^{\textup{elab}}}
\nc{\Mid}{\,\big|\,}
\nc{\mmod}[1]{\modname(#1)}%
\nc{\MMod}[1]{\Mod(#1)}%
\nc{\onto}{\mathop{\twoheadrightarrow}}
\nc{\op}{^{\opname}}
\nc{\sminus}{\smallsetminus}
\nc{\potimes}[1]{^{\otimes #1}}%
\nc{\sbull}{{\scriptscriptstyle\bullet}}
\nc{\brk}[1]{\{#1\}}
\nc{\SET}[2]{\big\{\,#1\Mid#2\,\big\}}
\nc{\SHc}{\SH^{\mathrm{c}}}
\nc{\too}{\mathop{\longrightarrow}\limits}
\nc{\unit}{\mathbb{1}}%
\nc{\zerotwist}{_{0\textrm{-twist}}}

\dmo{\DPerm}{DPerm}
\nc{\Dperm}[2]{\Der_{\perm}(#1;#2)}%

\nc{\W}{\mathbb{W}}

\dmo{\End}{End}
\nc{\isoto}{\overset{\sim}{\,\to\,}}
\nc{\lto}{\leftarrow}
\nc{\xfrom}[1]{\xleftarrow{#1}}
\nc{\xto}[1]{\xrightarrow{#1}}
\nc{\xinto}[1]{\overset{#1}{\,\into\,}}
\nc{\xonto}[1]{\overset{#1}{\,\onto\,}}
\nc{\qquadtext}[1]{\qquad\textrm{#1}\qquad}
\nc{\quadtext}[1]{\quad\textrm{#1}\quad}

\nc{\normaleq}{\trianglelefteqslant}
\nc{\normal}{\lhd}
\nc{\lecl}{\Subset}
\dmo{\can}{can}
\dmo{\chara}{char}%
\dmo{\id}{id}
\dmo{\Img}{Im}
\dmo{\Komp}{K}%
\dmo{\comp}{comp}
\dmo{\open}{open}
\dmo{\proj}{proj}%
\dmo{\Proj}{Proj} %
\dmo{\rmH}{H}
\dmo{\Res}{Res}
\dmo{\smallb}{b}%
\dmo{\Rad}{Rad}
\dmo{\Supp}{Supp}
\dmo{\stmod}{stmod}
\dmo{\kosname}{kos}
\dmo{\subname}{Sub}
\dmo{\zulname}{zul}
\dmo{\red}{red}%
\dmo{\Max}{Max}

\nc{\Sub}[1]{\subname_{#1}}
\nc{\Weyl}[2]{{#1}/\!\!/{#2}}%
\nc{\WGH}{\Weyl{G}{H}}
\nc{\WGK}{\Weyl{G}{K}}
\nc{\WGL}{\Weyl{G}{L}}
\nc{\tInd}{{}^{\otimes\!}\Ind}%
\nc{\inn}{;}%
\nc{\Vee}[1]{V_{#1}}%
\nc{\tVee}[1]{\tilde{V}_{#1}}%
\nc{\VG}{\Vee{G}}
\nc{\Loctens}[1]{\Locname_{\otimes}(#1)}
\nc{\SpcKE}{\Spc(\cK(E))}
\nc{\SpcKG}{\Spc(\cK(G))}
\nc{\RN}{R^\sbull_{\textrm{max}}}
\nc{\twH}{\tilde{\rmH}}
\nc{\Rall}{\rmH^{\sbull\sbull}}%
\nc{\Rloc}[1]{\calO^\sbull_{#1}}%
\nc{\RE}{\Rloc{E}}%
\nc{\REH}{\Rloc{E}(H)}%
\nc{\RGH}{\Rloc{G}(H)}%
\nc{\Rlocp}[1]{\underline{\calO}^\sbull_{#1}}%
\nc{\Rallp}[1]{\underline{\rmH}^{\sbull\sbull}(#1)}%
\nc{\Aff}[1]{{\bbV}^{#1}}
\nc{\Pone}{{\bbP^1_{\!\!\displaystyle\cdot\!\cdot\!\cdot}}}
\nc{\tinyPone}{{\bbP^1_{\!\!\scriptstyle\cdot \cdot \cdot}}}
\nc\ouratop[2]{\genfrac{}{}{0pt}{1}{#1}{#2}}
\nc{\EA}[2]{\mathcal{E}_{#1}(#2)}
\nc{\EAt}[2]{\tilde{\mathcal{E}}_{#1}(#2)}
\nc{\EAp}[1]{\bar{\mathcal{E}}_{p}(#1)}
\nc{\EApp}[1]{\EA{p}{#1}}
\nc{\EApG}{\EAp{G}}
\nc{\EAppG}{\EApp{G}}
\nc{\EApppG}{\mathcal{O}''_{p\textrm{-elab}}(G)}
\dmo{\ttCat}{tt-Cat} %
\dmo{\GrRing}{GrRing} %
\nc{\goodopen}[2][]{U_{#1}(#2)}

\nc{\kos}[2][]{\kosname_{#1}(#2)}%
\nc{\sH}{\kos{H}}%
\nc{\sHG}{\kos[G]{H}}%
\nc{\sKG}{\kos[G]{K}}%

\nc{\zul}[2][]{\zulname_{#1}(#2)}%
\nc{\tHG}{\zul[G]{H}}%

\nc{\adh}[1]{\overline{#1}}%
\nc{\adhpt}[1]{\adh{\{#1\}}}%
\nc{\adj}{\dashv}
\nc{\aka}{{a.\,k.\,a.}\ }
\nc{\bs}{\backslash}
\nc{\Cb}{\Chain_{\smallb}}%
\nc{\Ch}{\Cb}%
\nc{\doublequot}[3]{#1\backslash #2/#3}%
\nc{\Db}{\Der_{\smallb}}%
\nc{\eps}{\epsilon}
\nc{\equalby}[1]{\overset{\textrm{#1}}=}
\nc{\etal}{{\sl et al.}}
\nc{\FFq}{\FF_{\!q}}
\nc{\FFp}{\FF_{\!p}}
\nc{\gm}{\mathfrak{m}}%
\mathchardef\mhyphen="2D
\nc{\Gasets}[1][]{\ifblank{#1}{\Gamma}{#1}\mathrm{\mhyphen Sets}}
\nc{\gasets}[1][]{\ifblank{#1}{\Gamma}{#1}\mathrm{\mhyphen sets}}
\nc{\HGH}{\doublequot HGH}%
\nc{\ideal}[1]{\langle #1\rangle}
\nc{\idealK}[1]{\langle #1\rangle_{\scriptscriptstyle\cK}}
\nc{\ihom}{{\mathsf{hom}}} %
\nc{\Kb}{\Komp_{\smallb}}%
\nc{\K}{\Komp}%
\nc{\Kbac}{\Komp_{\smallb,\mathrm{ac}}}%
\nc{\KKac}{\KK_{\mathrm{ac}}}%
\nc{\Kac}{\cK_{\mathrm{ac}}}%
\nc{\Kp}{\Komp_{+}}%
\nc{\Km}{\Komp_{-}}%
\nc{\Ksp}{\Komp_{+,\mathrm{sp}}}%
\nc{\Kbm}{\Komp_{\mathrm{b},\le0}}%
\nc{\Lotimes}{\otimes^{\rmL}}
\nc{\pproj}[1]{#1\text{-}\kern-0.1em\proj}%
\nc{\restr}[1]{_{\downharpoonright{\scriptstyle #1}}}%
\nc{\smat}[1]{\left(\begin{smallmatrix} #1 \end{smallmatrix}\right)}%
\nc{\Sn}[1]{\mathfrak{S}_{#1}}%
\nc{\Snm}{\Sn{m}}%
\nc{\Spccat}[1]{\Spc(\cat #1)}
\nc{\SpcH}{\Spc(\cH)}%
\nc{\SpcK}{\Spc(\cK)}%
\nc{\To}{\Rightarrow}
\nc{\Top}{\mathsf{Top}}
\nc{\EndHere}{\bibliographystyle{alpha}\bibliography{ref}\end{document}}
\nc{\tristar}{\smallbreak\begin{center}*\ *\ *\end{center}}

\usepackage{mathtools}
\makeatletter
\let\ea\expandafter
\newcount\foreachcount

\def\foreachLetter#1#2#3{\foreachcount=#1
  \ea\loop\ea\ea\ea#3\@Alph\foreachcount
  \advance\foreachcount by 1
  \ifnum\foreachcount<#2\repeat}

\def\definebb#1{\ea\gdef\csname #1#1\endcsname{\ensuremath{\mathbb{#1}}\xspace}}
\foreachLetter{1}{27}{\definebb}

\def\definebb#1{\ea\gdef\csname bb#1\endcsname{\ensuremath{\mathbb{#1}}\xspace}}
\foreachLetter{1}{27}{\definebb}

\makeatother

\nc{\martin}[1]{{\color{orange}{#1}}}
\nc{\paul}[1]{{\color{VioletRed}{#1}}}

\author{Paul Balmer}
\address{Paul Balmer, UCLA Mathematics Department, Los Angeles, CA 90095, USA}
\email{balmer@math.ucla.edu}
\urladdr{https://www.math.ucla.edu/~balmer}

\author{Martin Gallauer}
\address{Martin Gallauer, Warwick Mathematics Institute, Coventry CV4 7AL, UK}
\email{martin.gallauer@warwick.ac.uk}
\urladdr{https://warwick.ac.uk/mgallauer}

\title[The geometry of permutation modules]{The geometry of permutation modules}
\makeatletter \let\ourtitle\@title \makeatother
\hypersetup{pdfauthor={Paul Balmer and Martin Gallauer},pdftitle={\ourtitle}}

\begin{document}

\date{2023 October 23}

\subjclass[2020]{20C20; 18F99, 20J06, 13A02}
\keywords{Tensor-triangular geometry, permutation modules, modular fixed-points, stratification, twisted permutation cohomology}

\thanks{First-named author supported by NSF grant~DMS-2153758. Second-named author partially supported by the Max-Planck Institute for Mathematics in Bonn.
  The authors thank the Hausdorff Institute for Mathematics in Bonn for its hospitality during the 2022 Trimester ``Spectral Methods in Algebra, Geometry, and Topology''.
For the purpose of open access, the authors have applied a Creative Commons Attribution (CC-BY) licence to any Author Accepted Manuscript version arising from this submission.
}

\maketitle

\begin{abstract}
We consider the derived category of permutation modules for a finite group, in positive characteristic.
We stratify this tensor triangulated category using Brauer quotients.
We describe the spectrum of its compact objects, by reducing the problem to elementary abelian groups and then by using a twisted form of cohomology to express the spectrum locally in terms of the graded endomorphism ring of the unit.
Together, these results yield a classification of thick and of localizing ideals.
\end{abstract}
\vspace{.5cm}
\begin{figure}[H]
\centering
\begin{tikzpicture}
\node[inner sep=0pt] (dperm) at (-.5,0)
{\includegraphics[width=.5\textwidth]{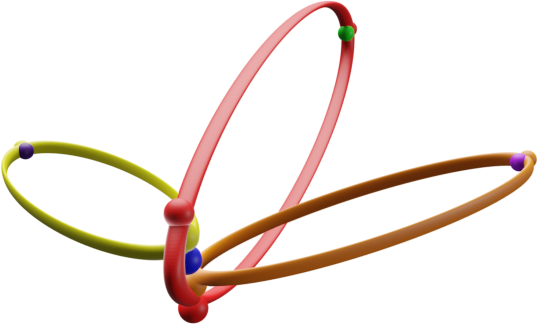}};
\node[inner sep=0pt] (mackey) at (-5,-4)
{\includegraphics[width=.18\textwidth]{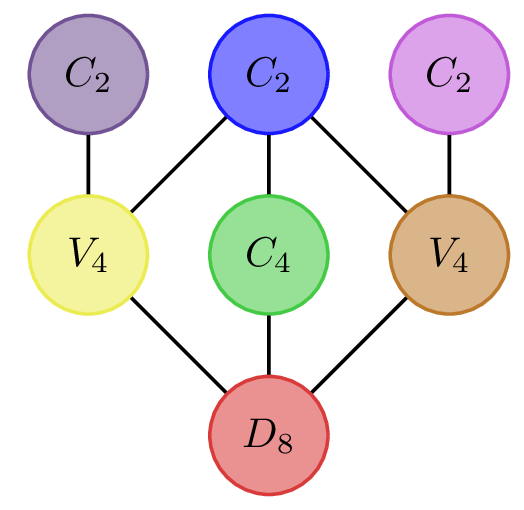}};
\node[inner sep=0pt] (stab) at (3,-4)
{\includegraphics[width=.4\textwidth]{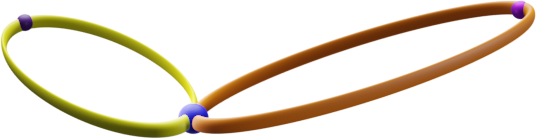}};
\draw[->,thick] (-3,-2) -- node[sloped,below] {\ } (-4,-2.5);
\draw[->,thick] (1.5,-2.5) -- node[below,sloped] {\ } (0.5,-2);
\end{tikzpicture}
\vspace{1cm}
\caption{The geometry for the dihedral group~$D_8$.}
\label{fig:mackey-perm-stab}
\end{figure}
\tableofcontents

\section{Preamble}
\label{sec:introduction}

Fix a field~$\kk$ of positive characteristic~\mbox{$p$}.
Let $G$ be a finite group.
We often write `tt' to abbreviate `tensor triangulated' or `tensor-triangular'.

\subsection*{Topic}
Among $\kk$-linear representations of~$G$, \emph{permutation modules} are perhaps the easiest to grasp: They are simply the $\kk$-linearizations $\kk(X)$ of $G$-sets~$X$.
They play an important role throughout equivariant mathematics, in subjects as varied as derived equivalences~\cite{rickard:splendid}, Mackey functors~\cite{yoshida:G-functors},
or equivariant homotopy theory~\cite{mathew-naumann-noel:nilpotence-descent}, to name a few.
The authors' original interest stems from yet another connection, with Voevodsky's theory of motives~\cite{Voevodsky00} and specifically Artin motives.
For elaboration on this theme, we refer the reader to~\cite{balmer-gallauer:Dperm}.

We consider a `small' tt-category, the homotopy category
\begin{equation}\label{eq:K(G)}%
\cK(G)=\cK(G;\kk):=\Kb\!\big(\perm(G;\kk)^\natural\big)
\end{equation}
of bounded complexes of finitely generated permutation $\kkG$-modules, idempotent-completed.
It sits as the compact part of the `big' tt-category
\begin{equation}\label{eq:T(G)}%
\cT(G)=\DPerm(G;\kk)
\end{equation}
obtained, for instance, by closing~$\cK(G)$ under coproducts and triangles in the homotopy category $\K(\MMod{\kkG})$ of all $\kkG$-modules.
We call $\DPerm(G;\kk)$ the \emph{derived category of permutation $\kkG$-modules}.
See~\Cref{Rec:perm} for details.

As we shall discuss in this preamble, these tt-categories of permutation modules are interesting and important for a variety of reasons.
To begin with, they stand at the crossroad of several subjects, as alluded to above.
Concretely, the tt-category~$\cT(G)$ is equivalent to:
\begin{enumerate}[left=0pt,label=(\roman*),itemsep=2pt]
\item\label{it:mack} the derived category of cohomological $\kk$-linear Mackey functors over~$G$,%
\item\label{it:equiv-hty} the homotopy category of modules over the constant Green functor $\rmH\!\underline{\kk}$ in genuine $G$-spectra,
\item\label{it:artin} the triangulated category of $\kk$-linear Artin motives generated by motives of intermediate fields in any Galois extension with Galois group~$G$.
\end{enumerate}
Consequently, while we adopt here the language of permutation modules, our results admit translations into each of the contexts~\ref{it:mack}, \ref{it:equiv-hty} and~\ref{it:artin}.

\subsection*{Main goals}
We want to understand the tensor-triangular geometry of these permutation tt-categories.
Tensor-triangular geometry~\cite{balmer:icm} is a way to bring organization to sometimes bewildering tt-categories, be it in topology, algebraic geometry or representation theory.
Its fundamental device is the tt-spectrum~$\SpcK$ of a small tt-category~$\cK$.
Computing~$\Spc(\cK(G))$ will provide a classification of all thick triangulated $\otimes$-ideals in~$\cK(G)$.
We also want to show that $\cK(G)$ strongly controls the big tt-category~$\cT(G)$, namely the Telescope Conjecture holds for~$\cT(G)$ and the localizing~$\otimes$-ideals of~$\cT(G)$ are classified by subsets of $\Spc(\cK(G))$.

\subsection*{Landscape}

Let us place~$\cK(G)$ among some standard $G$-equivariant tt-categories:
\begin{enumerate}[itemsep=2pt]
\item
\label{it:SH}%
The equivariant stable homotopy category~$\SH(G)^c$ of finite genuine $G$-spectra.
\item
\label{it:DMack}%
Kaledin's category of derived Mackey functors~$\DMack(G;\kk)^c$.
\item
\label{it:Db}%
The bounded derived category $\Db(\kkG)$ of finitely generated $\kkG$-modules.
\item
\label{it:stab}%
The stable module category~$\stab(\kkG)=\mmod{\kkG}/\proj(\kkG)$.
\end{enumerate}
These categories all fit in a natural sequence of tt-functors, from equivariant homotopy theory to modular representation theory, with our~$\cK(G)$ at center stage:
\begin{equation}
\label{eq:5-tt-cats}%
\xymatrix@C=1.5em{\SH(G)^c \ar[r] & \DMack(G;\kk)^c \ar[r]^-{} & \cK(G;\kk) \ar@{->>}[r]^-{} & \Db(\kkG) \ar@{->>}[r]^-{} & \stab(\kkG).}
\end{equation}
The initial one, $\SH(G)^c$, is topological in nature and its tt-geometry relies heavily on chromatic theory, \`a la Devinatz-Hopkins-Smith~\cite{devinatz-hopkins-smith:chromatic,hopkins-smith:chromatic}.
The first functor~$\SH(G)^c\to \DMack(G;\kk)^c$ moves to the $\kk$-linear world and thus the chromatic refinements disappear from~$\DMack(G;\kk)^c$ onwards.
A central feature of the tt-categories in~\eqref{eq:5-tt-cats} is their variance in the group~$G$.
Restriction, induction and conjugation turn them into so-called \emph{Mackey 2-functors}.
In the language of~\cite{balmer-dellambrogio:coh-two-mackey}, the three Mackey 2-functors~$\cK(G;\kk)$, $\Db(\kkG)$ and~$\stab(\kkG)$ are moreover \emph{cohomological}.
(This categorifies the fact that an ordinary Mackey functor is cohomological if $I^K_H\circ R_H^K$ is multiplication by~$[K\!:\!H]$.)
In other words, the second functor~$\DMack(G;\kk)^c\to \cK(G)$ in~\eqref{eq:5-tt-cats} moves us to the cohomological world.
The subsequent functors in~\eqref{eq:5-tt-cats} are simply localizations.
(For $\cK(G)\onto\Db(\kkG)$ this follows from~\cite{balmer-gallauer:resol-small}. For $\Db(\kkG)\onto \stab(\kkG)$ it is a well-known theorem due to Rickard~\cite{rickard:der-stab}, or Buchweitz~\cite{buchweitz:cohen-macauley}.)

\subsection*{Classical methods}

The four categories surrounding our~$\cK(G)$ in~\eqref{eq:5-tt-cats}
have a fairly well-understood tt-geometry, thanks to a series of powerful and widely used techniques that we shall now briefly review with application to~$\cK(G)$ in mind.

\smallbreak
The first obvious idea is to try some induction on the order of~$G$. For each of the tt-categories in~\eqref{eq:5-tt-cats} we can define a so-called
\begin{center}
\emph{geometric open}
\end{center}
inside their tt-spectrum. It is the open complement of all the images of the closed maps induced by restriction to proper subgroups.
This geometric open captures what is intrinsically new over~$G$, beyond what is detected by proper subgroups.
The name comes from stable homotopy theory~\eqref{it:SH}, as the localization of~$\SH(G)$ over the geometric open recovers the classical \emph{geometric $G$-fixed-points functor}.
In fact, a miracle occurs here: That localization of~$\SH(G)$ is simply the non-equivariant~$\SH$.
This fact has allowed~\cite{balmer-sanders:SH(G)-finite} to describe all points of the spectrum of~$\SH(G)^c$:
All points come from the non-equivariant chromatic spectrum~$\Spc(\SH^c)$ via geometric $H$-fixed-points, for all subgroups~$H\le G$ up to conjugation.
The same strategy has been applied by Patchkoria-Sanders-Wimmer~\cite{psw:derived-mackey} to derived Mackey functors over~$G$, where the same miracle occurs: The geometric open boils down to the non-equivariant version, independently of~$G$.
One deduces that the spectrum of~$\DMack(G;\kk)^c$ is the set of conjugacy classes of subgroups of~$G$ with a certain Alexandrov topology (if $G$ is a $p$-group, $K\in\adh{\{H\}}$ iff $K\leq_GH$).

This geometric fixed-points method has been formalized by Barthel-Castellana-Heard-Naumann-Pol~\cite{barthel-et-al:Quillen-stratification} for arbitrary equivariant tt-categories.
However, the induction process breaks down because the above `miracle' can evaporate in general: There is no simple description of the geometric open \textsl{a priori} and it can heavily depend on the group~$G$.
For example, for the stable module category of an elementary abelian $p$-group~$G=(C_p)^r$ the geometric open is dense in the spectrum, a projective space of dimension~$r-1$, and thus it grows with~$G$.
In that respect, $\cK(G)$ unfortunately behaves like~$\stab(\kkG)$, and~$\Db(\kkG)$; the miracle breaks down.
Beyond groups with very small $p$-Sylow the inductive approach of~\cite{balmer-sanders:SH(G)-finite} hits a wall because the geometric fixed-points of~$\cK(G)$ are too complicated.

\smallbreak
The second important method goes back to Serre's 1965 Theorem~\cite{serre:bockstein}.
In modern lingo, it says that the geometric open of $\Db(\kkG)$ is empty unless $G$ is an elementary abelian $p$-group.
As a consequence, and through further work of Quillen~\cite{quillen:spec-cohomology}, the tt-geometry of $\Db(\kkG)$ and~$\stab(\kkG)$ reduces to elementary abelian subgroups of~$G$.
Unfortunately again, Serre's result does not hold for~$\cK(G)$: The geometric open is non-empty for \emph{every} $p$-group~$G$.
Ipso facto, one cannot reduce the tt-geometry of~$\cK(G)$ to the elementary abelian subgroups of~$G$.

\smallbreak
Here is a third classical method.
Work of Benson-Carlson-Rickard~\cite{benson-carlson-rickard:tt-class-stab(kG)} determines the tt-geometry of the derived and stable categories by using the cohomology $\rmH^\sbull(G,\kk)$, that we can view as the graded endomorphism ring~$\End_{\Db(\kkG)}^\sbull(\unit)$ of the $\otimes$-unit object~$\unit$ in~$\Db(\kkG)$.
Reformulated in the language of~\cite{balmer:sss}, their result implies that the \emph{comparison map}, which exists for every tt-category~$\cK$,
\begin{equation}
\label{eq:comp-classical}%
\comp_{\cK}\colon\SpcK\to \Spech(\End^\sbull_{\cK}(\unit)),
\end{equation}
is a homeomorphism in the case of~$\cK=\Db(\kkG)$.
The case of $\stab(\kkG)$ only differs from the above by removing the closed point, \ie the `irrelevant' ideal~$\rmH^+(G;\kk)$.
Again, these ideas have been pushed and generalized, most famously in a corpus of work affectionately known as `BIK', after Benson-Iyengar-Krause~\cite{BIK:stratifying-stmod-kG}.
So we could hope that the BIK methods might apply to our tt-category of permutation modules~$\cK(G)$. Alas, the graded ring $\End^\sbull_{\cK(G)}(\unit)$ is just the field~$\kk$ and its spectrum, a meagre singleton, refuses to entertain any idea of geometry.

\subsection*{The challenge}
In summary, the classical methods that worked so well for~$\SH(G)^c$ and $\DMack(G;\kk)^c$ on the one hand, and those that worked for~$\Db(\kkG)$ and~$\stab(\kkG)$ on the other, all fall short in the case of~$\cK(G)$:
\[
\begin{tabular}{|c||c|c|c|}
  \hline
  (\checkmark=works, \ding{55}=fails)${}^{\vphantom{I^I}}$ & $\SH(G)^c$ {\footnotesize\&} $\DMack(G;\kk)^c$ & $\cK(G;\kk)$ & $\Db(\kkG)$ {\footnotesize\&} $\stab(\kkG)$ \\[.1em] \hline\hline
   Geom.\ fixed-pts${}^{\vphantom{I^I}}$ & \checkmark  & \ding{55} & \ding{55} \\[.1em] \hline
   Elem.\ ab.\ subgps${}^{\vphantom{I^I}}$ & \ding{55} & \ding{55} & \checkmark \\[.1em] \hline
   Comp.\ map \& BIK${}^{\vphantom{I^I}}$ & \ding{55} &  \ding{55} & \checkmark \\[.1em] \hline
\end{tabular}
\]
This turn of events could be surprising considering that $\cK(G)$ might seem to be the most accessible one among the five tt-categories in our list.
On the one side, the mere construction of $\SH(G)$ and $\DMack(G;\kk)$ is quite involved, whereas~$\cK(G)$ is simply the bounded homotopy category of an additive category. And on the other side, the modular representations that make up $\Db(\kkG)$ and $\stab(\kkG)$ are notoriously wild, whereas there are only finitely many isomorphism classes of indecomposables permutation $kG$-modules.

As we shall demonstrate in this article, the tt-geometry of~$\cK(G)$ just \emph{is} very complex.
It combines the complexity of its neighbors in~\eqref{eq:5-tt-cats}, $\DMack(G;\kk)^c$ and $\Db(\kkG)$, in a way reminiscent of how $\SH(G)^c$ combines the complexity of $\DMack(G;\kk)^c$ and $\SH^c$.
More precisely, just as the underlying set
\begin{align}
  \label{eq:SpcSHG-as-set}%
  \Spc(\SH(G)^c)&=\amalg_{H}\Spc(\SH^c)\\
  \intertext{decomposes, over conjugacy classes of subgroups $H\leq G$, into chromatic strata, so it will be shown that}
  \SpcKG&=\amalg_{H}\Spc(\Db(\kk(\WGH)))
          \label{eq:SpcKG-as-set}
\end{align}
decomposes, over conjugacy classes of $p$-subgroups $H\leq G$, into cohomological support varieties for the associated Weyl groups~$\WGH$.
\Cref{fig:landscape} may help the reader visualize the various phenomena at play.
Each of them can be thought of as contributing a `dimension' to the spectrum.
\begin{figure}[H]
\centering
\scalebox{.8}{\begin{tikzpicture}[set/.style={fill opacity=0.1}]

  \draw[yellow!10,fill=yellow,set,
xshift=4.2cm,
yshift=-0.705cm,
rotate =-45,] (0,0) ellipse (3cm and 1.5cm);

\draw[cyan!10,fill=cyan,set,
xshift=-2.2cm,
yshift=-0.705cm,
rotate =-45,] (0,0) ellipse (3cm and 1.5cm);

\draw[magenta!10,fill=magenta,set,
xshift=1cm,
yshift=-0.705cm,
rotate =45] (0,0) ellipse (3cm and 1.5cm);

\node at (-2.5,0) {$\SH^c$};
\node at (-.6,-2) {$\SH(G)^c$};
\node at (1.1,-.9) {$\DMack(G;\kk)^c$};
\node at (2.6,.5) {$\cK(G;\kk)$};
\node at (4.5,-1.1) {$\Db(\kkG)$};
\node at (5.2,-2.1) {$\stab(\kkG)$};

\end{tikzpicture}
}
\caption{Related tt-categories and tt-geometric phenomena: chromatic ({\color{cyan}cyan}), group-combinatorial ({\color{magenta}magenta}), and Mackey-cohomological ({\color{yellow}yellow}).}
\label{fig:landscape}
\end{figure}
So, how do we approach the tt-geometry of~$\cK(G)$ given that the classical methods fail us?
Let us discuss in broad strokes how the ideas behind those methods can still guide us to the solution, with suitably reinvented tools.
\subsection*{New methods}
While geometric fixed-points definitely remain insufficient, a different type of fixed-points functors, the \emph{modular fixed-points}, will prove very useful.
They are the correct analogue at the level of~$\cK(G)$ of Brauer quotients, a well-known tool for the study of permutation modules.
Firstly, we will use them to describe all points of the spectrum of~$\cK(G)$, arriving at~\eqref{eq:SpcKG-as-set} above.
In addition, they allow us to circumvent the failure of Serre's theorem for~$\cK(G)$ and instead of a reduction to elementary abelian \emph{subgroups}, obtain a reduction to elementary abelian \emph{subquotients} of~$G$.
Finally, for $G$ elementary abelian, although the comparison with cohomology and the BIK method still cannot be used globally on~$\cK(G)$, we will produce an open cover of $\Spc(\cK(G))$ over which the comparison map is indeed a homeomorphism. In other words, BIK will work on small enough pieces of the category~$\cK(G)$. Their determination will involve a `twisted' version of cohomology which epimorphically maps to the group cohomology for each Weyl group.

We explain these ideas in more detail and state precise theorems in the introduction to Part~I (\Cref{sec:intro-I}), where we discuss modular fixed-points and the `stratification' results about the big tt-category~$\cT(G)$, and in the introduction to Part~II (\Cref{sec:intro-II}) where we focus on the topology of~$\Spc(\cK(G))$ and produce the announced local analysis for $G$ elementary abelian.

\subsection*{Illustration}

A geometric paper should include pictures and there will be many of those below.
The title page shows what happens for $G=D_8$, the dihedral group of order~8, at the prime~$p=2$.
Hopefully, the beauty of \Cref{fig:mackey-perm-stab} will entice the reader to proceed beyond this preamble.

At the bottom right of~\Cref{fig:mackey-perm-stab}, we recognize the projective support variety of~$D_8$ consisting of two copies of~$\mathbb{P}^1_{\!\kk}$ glued together at an~$\FF_{\!2}$-rational point. It is the spectrum of the stable module category~$\stab(\kk D_8)$ and also the spectrum of~$\Db(\kk D_8)$ with its `irrelevant' closed point punctured out. This `puncturing' process produces more geometric pictures, displaying classical projective varieties associated to graded rings instead of their full homogeneous spectra.
At the bottom left, we recognize the lattice of conjugacy classes of subgroups of~$D_8$, with the Alexandrov topology, which is the spectrum of~$\DMack(D_8;\kk)^c$ with the closed point (the trivial subgroup) punctured out for coherence. In the center of this triptych sits the spectrum of~$\cK(D_8)$ in majesty, with its closed points removed.
It has three irreducible components, each of which is a~$\mathbb{P}^1_{\!\kk}$ with multiple $\FF_{\!2}$-rational points doubled.
The components meet in some of these doubled points.
This spectrum is presented in detail at the end of the paper, in \Cref{Exa:D8}.

The two maps in \Cref{fig:mackey-perm-stab} are the images under the contravariant functor~$\Spc(-)$ of the tt-functors in~\eqref{eq:5-tt-cats}, ignoring~$\SH(D_8)$.
The colors are chosen to indicate where each point goes in a hopefully self-explanatory way.
We see that the right-hand projective support variety $\Spc(\stab(\kk D_8))$ embeds as an open subset of $\Spc(\cK(D_8))$, meeting two of the three irreducible components.
These two components are detected by the two Klein-four subgroups of~$D_8$.
The third component is detected by the announced modular fixed-points and relies on the elementary abelian Klein-four~$D_8/Z(D_8)$ that appears as a \emph{quotient}.

\medskip
\begin{Ack}
We thank Tobias Barthel, Henning Krause and Peter Symonds for precious conversations and for their stimulating interest.
We also thank Ivo Dell'Ambrogio, Colin Ni and Beren Sanders for comments and suggestions.
Finally, we thank an anonymous referee for their thorough comments.
\end{Ack}

\tristar

\begin{Ter}
\label{Ter:0}%
A `tensor category' is an additive category with a symmetric-monoidal product additive in each variable.
We say `tt-category' for `tensor triangulated category' and `tt-ideal' for `thick (triangulated) $\otimes$-ideal'.
We say `big' tt-category for a rigidly-compactly generated tt-category, as in~\cite{balmer-favi:idempotents}.
We write $\SpcK$ for the tt-spectrum of a tt-category~$\cK$.
For an object~$x\in\cK$, we write $\open(x)=\SET{\cP\in\SpcK}{x\in\cP}$ to denote the open complement of~$\supp(x)$.

For subgroups~$H,K\le G$, we write $H\le_G K$ to say that $H$ is $G$-conjugate to a subgroup of~$K$, that is, $H^g\le K$ for some~$g\in G$.
We write $\sim_G$ for $G$-conjugation.
As always $H^g=g\inv H\,g$ and~${}^{g\!}H=g\,H\,g\inv$.
We write $\Sub{p}(G)$ for the set of $p$-subgroups of~$G$ and $\Sub{p}(G)/_G$ for its $G$-orbits under $G$-conjugation.
We write $N_G(H,K)$ for $\SET{g\in G}{H^g\le K}$ and $N_GH=N_G(H,H)$ for the normalizer.
For each subgroup~$H\le G$, its Weyl group is $\WGH=(N_G H)/H$.
\end{Ter}
\begin{Conv}
\label{Conv:light}%
When a notation involves a subgroup~$H$ of an ambient group~$G$, we drop the mention of~$G$ if no ambiguity can occur, for instance $\Res_H$ for~$\Res^G_H$.
The mention of the field~$\kk$ is sometimes dropped, for readability.
\end{Conv}

\bigbreak
\part{Modular fixed-points and stratification}
\label{part:stratification}
\bigbreak

\section{Introduction to Part~I}
\label{sec:intro-I}

\smallbreak
Having sketched the broad context and the aims of the article, let us turn to the content of \Cref{part:stratification} in more detail.

\subsection*{Stratification}
In colloquial terms, one of our main results says that the big derived category $\cT(G)$ of permutation modules given in~\eqref{eq:T(G)} is strongly controlled by its compact part~$\cK(G)$ described in~\eqref{eq:K(G)}:
\begin{Thm}[\Cref{Thm:stratification}]
\label{Thm:stratification-intro}%
The derived category of permutation modules $\cT(G)$ is \emph{stratified} by~$\SpcKG$ in the sense of Barthel-Heard-Sanders~\cite{barthel-heard-sanders:stratification-Mackey}.
\end{Thm}

Let us remind the reader of BHS-stratification.
What we establish in \Cref{Thm:stratification} is an inclusion-preserving bijection between the localizing $\otimes$-ideals of~$\cT(G)$ and the subsets of the spectrum~$\SpcKG$.
This bijection is defined via a canonical support theory on~$\cT(G)$ that exists once we know that~$\SpcKG$ is a noetherian space (\Cref{Prop:Spc-noetherian}).
Note that \Cref{Thm:stratification-intro} cannot be obtained via `BIK-stratification' as in Benson-Iyengar-Krause~\cite{BIK:stratifying-stmod-kG}, since the endomorphism ring of the unit~$\End^\sbull_{\cK(G)}(\unit)=\kk$ is too small.
However, we shall see that \cite{BIK:stratifying-stmod-kG} plays an important role in our proof, albeit indirectly.
An immediate consequence of stratification is the Telescope Property (\Cref{Cor:telescope}):
\begin{Cor}
\label{Cor:smash-intro}%
Every smashing $\otimes$-ideal of~$\cT(G)$ is generated by its compact part.
\end{Cor}

The key question is now to understand the spectrum~$\SpcKG$.
Recall from \cite[Theorem~5.13]{balmer-gallauer:resol-small} that the innocent-looking category $\cK(G)$ actually captures much of the wilderness of modular representation theory. It admits as Verdier quotient the derived category $\Db(\kk G)$ of \emph{all} finitely generated $\kkG$-modules.
By Benson-Carlson-Rickard~\cite{benson-carlson-rickard:tt-class-stab(kG)}, the spectrum of $\Db(\kkG)$ is the homogeneous spectrum of the cohomology ring~$\rmH^\sbull(G,\kk)$.
We deduce in \Cref{Prop:VeeG} that $\SpcKG$ contains an open piece~$\VG$
\begin{equation}
\label{eq:VeeG}%
\Spech(\rmH^\sbull(G,\kk))\cong\Spc(\Db(\kkG))=:\,{\VG} \ \hook \,\SpcKG
\end{equation}
that we call the \emph{cohomological open} of~$G$.
\goodbreak

In good logic, the closed complement of~$\Vee{G}$ is
\begin{equation}
\label{eq:Z_G}%
\SpcKG\sminus \VG\quad = \quad \Supp(\Kac(G))
\end{equation}
the support of the tt-ideal $\Kac(G)=\Ker(\cK(G)\onto \Db(\kkG))$ of acyclic objects.
The problem becomes to understand this closed subset~$\Supp(\Kac(G))$.
To appreciate the issue, let us say a word of closed points.
\Cref{Cor:closed-pts} gives the complete list: There is one closed point~$\cM(H)$ of~$\SpcKG$ for every conjugacy class of $p$-subgroups~$H\le G$.
The cohomological open~$\VG$ only contains one closed point, for the trivial subgroup~\mbox{$H=1$}. All other closed points $\cM(H)$ for $H\neq 1$ are to be found in the complement $\Supp(\Kac(G))$.
It will turn out that $\SpcKG$ is substantially richer than the cohomological open~$\VG$, in a way that involves $p$-local information about~$G$.
To understand this, we need the right notion of fixed-points.

\smallbreak
\subsection*{Modular fixed-points}
Let $H\le G$ be a subgroup. We abbreviate by
\begin{equation}\label{eq:Weyl}%
\WGH:=W_G(H)=N_G(H)/H
\end{equation}
the Weyl group of~$H$ in~$G$.
If $H\normaleq G$ is normal then of course~$\WGH=G/H$.

For every $G$-set~$X$, its $H$-fixed-points~$X^H$ is canonically a $(\WGH)$-set.
We also have a naive fixed-points functor $M\mapsto M^H$ on $\kkG$-modules but it does not `linearize' fixed-points of $G$-sets, that is, $\kk(X)^H$ differs from~$\kk(X^H)$ in general. And it does not preserve the tensor product.
We would prefer a \emph{tensor}-triangular functor
\begin{equation}
\label{eq:Psi^H}%
\Psi^H\colon \cT(G)\to \cT(\WGH)
\end{equation}
such that $\Psi^H(\kk(X))=\kk(X^H)$ for every $G$-set~$X$.

A related problem was encountered long ago for the $G$-equivariant stable homotopy category~$\SH(G)$, see~\cite{LMSM:equivariant-stable-htpy}:
The naive fixed-points functor (\aka the `genuine' or `categorical' fixed-points functor) is not compatible with taking suspension spectra, and it does not preserve the smash product.
To solve both issues, topologists invented \emph{geometric} fixed-points~$\Phi^H$.
As we saw in the preamble, those geometric fixed-points functors already played an important role in tensor-triangular geometry~\cite{balmer-sanders:SH(G)-finite,barthel-greenlees-hausmann:SH(G)-compact-Lie,psw:derived-mackey} and it would be reasonable, if not very original, to try the same strategy for~$\cT(G)$.
Unfortunately they do \emph{not} give us the wanted~$\Psi^H$ of~\eqref{eq:Psi^H}, as we explain in \Cref{Rem:geom-fixed-pts}.

In summary, we need a third notion of fixed-points functor~$\Psi^H$, which is neither the naive one~$(-)^H$, nor the `geometric' one~$\Phi^H$ imported from topology.
It turns out (see \Cref{Rem:only-p-subgroups}) that it can only exist in characteristic~$p$ when $H$ is a~$p$-subgroup.
The good news is that this is the only restriction (see~\Cref{sec:modular-fixed-pts}):
\begin{Prop}
\label{Prop:modular-fixed-points-intro}%
For every $p$-subgroup~$H\le G$ there exists a coproduct-preserving tensor-triangular functor on the big derived category of permutation modules~\eqref{eq:T(G)}
\[
\Psi^H\colon \quad \cT(G)\too \cT(\WGH)
\]
such that $\Psi^H(\kk(X))\cong\kk(X^H)$ for every $G$-set~$X$.
In particular, this functor preserves compacts and restricts to a tt-functor $\Psi^H\colon \cK(G)\to \cK(\WGH)$ on~\eqref{eq:K(G)}.
\end{Prop}
We call the $\Psi^H$ the \emph{modular $H$-fixed-points functors}. These functors already exist at the level of additive categories~$\perm(G;\kk)^\natural\to \perm(\WGH;\kk)^\natural$, where they agree with the classical Brauer quotient, although our construction is quite different.
See~\Cref{Rem:Brauer}.
These $\Psi^H$ also recover motivic functors considered by Bachmann in~\cite[Corollary~5.48]{bachmann:thesis}.
Equipped with those~$\Psi^H$, let us return to~$\SpcKG$.

\smallbreak
\subsection*{The spectrum}
Each tt-functor~$\Psi^{H}$ induces a continuous map on spectra
\begin{equation}
\label{eq:psi^N}%
\psi^{H}:=\Spc(\Psi^{H})\,\colon \quad \Spc(\cK(\WGH))\too \SpcKG.
\end{equation}
In particular $\SpcKG$ receives via this map~$\psi^H$ the cohomological open~$\Vee{\WGH}$ of the Weyl group of~$H$:
\begin{equation}
\label{eq:psi^H-intro}
\Vee{\WGH}=\Spc(\Db(\kk(\WGH)))\hook\Spc(\cK(\WGH))\xto{\psi^H}\SpcKG.
\end{equation}
Using this, we can describe the \emph{set} underlying $\SpcKG$ in~\Cref{Thm:all-points}:

\begin{Thm}
\label{Thm:Spc(K(G))-intro}%
Every point of~$\SpcKG$ is the image ${\psi}^H(\gp)$ of a point~$\gp\in \Vee{\WGH}$ for some $p$-subgroup~$H\le G$, in a unique way up to $G$-conjugation, \ie we have ${\psi}^H(\gp)={\psi}^{H'}(\gp')$ if and only if there exists $g\in G$ such that $H^g=H'$ and $\gp^g=\gp'$.
\end{Thm}

In this description, the trivial subgroup~$H=1$ contributes the cohomological open~$\VG$ (since $\Psi^1=\Id$).
Its closed complement $\Supp(\Kac(G))$, introduced in~\eqref{eq:Z_G}, is covered by images of the modular fixed-points maps~\eqref{eq:psi^H-intro}, for~$H$ running through all non-trivial $p$-subgroups of~$G$.
The main ingredient in proving \Cref{Thm:Spc(K(G))-intro} is our Conservativity \Cref{Thm:conservativity} on the associated big categories:
\begin{Thm}
\label{Thm:conservativity-intro}%
The family of functors $\{\cT(G)\xto{\Psi^H}\cT(\WGH)\onto \K\Inj(\kk(\WGH))\}_{H}$, indexed by the (conjugacy classes of) $p$-subgroups~$H\le G$, is conservative.
\end{Thm}

This determines the set $\SpcKG$. The topology of~$\SpcKG$ involves new characters and we postpone its discussion to \Cref{part:twisted-cohomology}.

\smallbreak
\subsection*{Measuring progress by examples}
Before the present work, we only knew the case of cyclic group~$C_p$ of order~$p=2$, where $\Spc(\cK(C_2))$ is a 3-point space\,{\rm(\footnote{\,A line indicates specialization: The higher point is in the closure of the lower one.})}
\begin{equation}
\label{eq:C_2}%
\vcenter{\xymatrix@R=.5em@C=.2em{
{\color{Brown}\scriptstyle\Supp(\Kac(C_2))}
& {\color{Brown}\bullet} \ar@{-}@[Gray][rd]
&&
{\color{OliveGreen}\bullet}
\\
&& {\color{OliveGreen}\bullet} \ar@{-}@[OliveGreen][ru]_-{{\color{OliveGreen}\Vee{C_2}}}
}}\kern5em
\end{equation}
This was the starting point of our study of real Artin-Tate motives~\cite[Theorem~3.14]{balmer-gallauer:rage}.
It appears independently in Dugger-Hazel-May~\cite[Theorem~5.4]{dugger-et-al:C2}.

The present paper gives a description of~$\Spc(\cK(G))$ for arbitrary finite groups~$G$.
We gather several examples in \Cref{sec:examples} to illustrate the progress made since~\eqref{eq:C_2}, and also for later use in~\cite{balmer-gallauer:Artin-finite-fields}.
Let us highlight the case of the quaternion group~$G=Q_8$ (\Cref{Exa:Q_8}).
By Quillen, we know that the cohomological open~$\Vee{Q_8}$ is the same as for its center~$Z(Q_8)=C_2$, that is, the 2-point Sierpi\'{n}ski space displayed in green on the right-hand side of~\eqref{eq:C_2}, and again below:
\[
\vcenter{\xymatrix@R=.5em@C=.2em{
{\color{Brown}\scriptstyle\Supp(\Kac(Q_8))=?} \ar@{..}@[Gray][rd]
&&
{\color{OliveGreen}\bullet}
\\
& {\color{OliveGreen}\bullet} \ar@{-}@[OliveGreen][ru]_-{{\color{OliveGreen}\Vee{Q_8}\,\cong\,\Vee{C_2}}}
}}
\]
If intuition was solely based on~\eqref{eq:C_2} one could believe that $\SpcKG$ is just~$\Vee{G}$ with some discrete decoration for the acyclics, like the single (brown) point on the left-hand side of~\eqref{eq:C_2}. The quaternion group offers a stark rebuttal.

Indeed, the spectrum $\Spc(\cK(Q_8))$ is the following space:
\begin{equation}\label{eq:Q_8}%
\kern2em\vcenter{\xymatrix@C=0.1em@R=.4em{
{\color{Brown}\overset{}{\bullet}} \ar@{-}@[Brown][rrdd] \ar@{-}@[Brown][rrrrdd] \ar@{-}@[Brown][rrrrrrdd] \ar@{~}@[Brown][rrrrrrrrdd] &&& {\color{Brown}\overset{}{\bullet}} \ar@{-}@[Brown][ldd] \ar@{-}@[Brown][rrrrrrdd]
&& {\color{Brown}\overset{}{\bullet}} \ar@{-}@[Brown][ldd] \ar@{-}@[Brown][rrrrrrdd]
&& {\color{Brown}\overset{}{\bullet}} \ar@{-}@[Brown][ldd] \ar@{-}@[Brown][rrrrrrdd]
&& {\color{Brown}\overset{}{\bullet}} \ar@{~}@[Brown][ldd] \ar@{-}@[Brown][dd] \ar@{-}@[Brown][rrdd] \ar@{-}@[Brown][rrrrdd]  \ar@{-}@[Gray][rrrrrrrdd]
&&&&&&&& {\color{OliveGreen}\bullet} \ar@{-}@[OliveGreen][ldd]^-{{\color{OliveGreen}\Vee{Q_8}\,\cong\,\Vee{C_2}}}
\\ \\
&& {\color{Brown}\bullet} \ar@{-}@[Brown][rrrrrrrdd]_-{\color{Brown}\Supp(\Kac(Q_8))\,\cong\,\Spc(\cK(C_2^{\times2}))\kern5em}
&& {\color{Brown}\bullet} \ar@{-}@[Brown][rrrrrdd]
&& {\color{Brown}\bullet} \ar@{-}@[Brown][rrrdd]
& \ar@{.}@[Brown][r]
& {\scriptstyle\color{Brown}\tinyPone} \ar@{~}@[Brown][rdd] \ar@{.}@[Brown][rrrrrrr]
& {\color{Brown}\bullet} \ar@{-}@[Brown][dd]
&& {\color{Brown}\bullet} \ar@{-}@[Brown][lldd]
&& {\color{Brown}\bullet} \ar@{-}@[Brown][lllldd]
&&& {\color{OliveGreen}\bullet}
\\ \\
&&&&&&&&& {\color{Brown}\bullet}
&&&&
}}
\end{equation}
Its support of acyclics (in brown) is actually way more complicated than the cohomological open itself: It has Krull dimension two and contains a copy of the projective line~$\bbP^1_{\!\kk}$. In fact, the map~$\psi^{C_2}$ given by modular fixed-points identifies the closed piece $\Supp(\Kac(Q_8))$ with the whole spectrum for $Q_8/C_2$, which is a Klein-four. We discuss the latter in \Cref{Exa:Klein4} where we also explain the meaning of~$\Pone$ and the undulated lines in~\eqref{eq:Q_8}.

\section{Recollections and Koszul objects}
\label{sec:rec-red}%

\begin{Rec}
\label{Rec:ttg}%
We refer to~\cite{balmer:icm} for elements of tensor-triangular geometry.
Recall simply that the \emph{spectrum} of an essentially small tt-category~$\cK$ is $\SpcK=\SET{\cP\subsetneq\cK}{\cP\textrm{ is a prime tt-ideal}}$.
For every object~$x\in\cK$, its support is $\supp(x):=\SET{\cP\in\SpcK}{x\notin \cP}$.
These form a basis of closed subsets for the topology.
\end{Rec}

\begin{Rec}
\label{Rec:perm}%
(Here $\kk$ can be a commutative ring.)
Recall our reference~\cite{balmer-gallauer:Dperm} for details on permutation modules.
Linearizing a $G$-set~$X$, we let $\kk(X)$ be the free $\kk$-module with basis~$X$ and $G$-action $\kk$-linearly extending the $G$-action on~$X$.
A \emph{permutation $\kkG$-module} is a $\kkG$-module isomorphic to one of the form~$\kk(X)$.
These modules form an additive subcategory $\Perm(G;\kk)$ of~$\MMod{\kkG}$, with all $\kkG$-linear maps.
We write $\perm(G;\kk)$ for the full subcategory of finitely generated permutation $\kkG$-modules and $\perm(G;\kk)^\natural$ for its idempotent-completion.

We tensor $\kkG$-modules in the usual way, over~$\kk$ with diagonal $G$-action. The linearization functor~$\kk(-)\colon \Gasets[G] \too \Perm(G;\kk)$ turns the cartesian product of $G$-sets into this tensor product. For every finite~$X$, the module $\kk(X)$ is self-dual.

We consider the idempotent-completion $(-)^\natural$ of the homotopy category of bounded complexes in the additive category~$\perm(G;\kk)$
\begin{equation*}
\cK(G)=\cK(G;\kk):=\Kb(\perm(G;\kk))^\natural\cong \Kb(\perm(G;\kk)^\natural).
\end{equation*}
As $\perm(G;\kk)$ is an essentially small tensor-additive category, $\cK(G)$ becomes an essentially small tensor triangulated category. As $\perm(G;\kk)$ is rigid so is~$\cK(G)$, with degreewise duals. Its tensor-unit $\unit=\kk$ is the trivial $\kkG$-module $\kk=\kk(G/G)$.

The `big' \emph{derived category of permutation $\kkG$-modules}~\cite[Definition~3.6]{balmer-gallauer:Dperm} is
\begin{equation*}
\DPerm(G;\kk)=\K(\Perm(G;\kk))\big[\{G\textrm{-quasi-isos}\}\inv\big],
\end{equation*}
where a $G$-quasi-isomorphism $f\colon P\to Q$ is a morphism of complexes such that the induced morphism on $H$-fixed points $f^H$ is a quasi-isomorphism for every subgroup~$H\le G$.
It is also the localizing subcategory of $\K(\Perm(G;\kk))$ generated by~$\cK(G)$, and it follows that $\cK(G)=\DPerm(G;\kk)^c$.
\end{Rec}

\begin{Exa}
For $G$ trivial, the category~$\cK(1;\kk)=\Dperf(\kk)$ is that of perfect complexes over~$\kk$ (any ring) and $\DPerm(1;\kk)$ is the derived category of~$\kk$.
\end{Exa}

\begin{Rem}
The tt-category $\cK(G)$ depends functorially on~$G$ and~$\kk$. It is contravariant in the group. Namely if $\alpha\colon G\to G'$ is a homomorphism then restriction along~$\alpha$ yields a tt-functor $\alpha^*\colon \cK(G')\to \cK(G)$.
When $\alpha$ is the inclusion of a subgroup~$G\le G'$, we recover usual restriction
\[
\Res^{G'}_{G}\colon \cK(G')\to \cK(G).
\]
When~$\alpha$ is a quotient $G\onto G'=G/N$ for~$N\normaleq G$, we get \emph{inflation}, denoted here\,(\footnote{\,We avoid the traditional $\Infl^{G}_{G/N}$ notation which is not coherent with the restriction notation.})
\[
\Infl^{G/N}_G\colon \cK(G/N;\kk)\to \cK(G).
\]
The covariance of $\cK(G)$ in~$\kk$ is simply obtained by extension-of-scalars.
All these functors are the `compact parts' of similarly defined functors on $\DPerm$.
\end{Rem}

Let us say a word of $\kkG$-linear morphisms between permutation modules.
\begin{Rec}
\label{Rec:Hom}%
Let $H,K\le G$ be subgroups. Then $\Hom_{\kkG}(\kk(G/H),\kk(G/K))$ admits a $\kk$-basis $\{f_g\}_{[g]}$ indexed by classes $[g]\in \doublequot HGK$. Namely, choosing a representative in each class $[g]\in \doublequot HGK$, one defines
\begin{equation}
\label{eq:Hom-f_g}%
f_g\colon \quad \kk(G/H)\underset{\eta}{\ \into\ } \kk(G/L)\underset{c_g}{\isoto} \kk(G/L^g)\underset{\eps}{\ \onto\ } \kk(G/K)
\end{equation}
where we set $L:=H\cap {}^{g\!}K$, where $\eta$ and~$\eps$ are the usual maps using that $L\le H$ and $L^g\le K$ (thus $\eta$ maps $[e]_H$ to $\sum_{\gamma\in H/L} \gamma$ and $\eps$ extends $\kk$-linearly the projection $G/L^g\onto G/K$), and finally where the middle isomorphism~$c_g$ is
\begin{equation}\label{eq:Hom-c_g}%
\vcenter{\xymatrix@R=.1em{
c_g\colon & \kk(G/L) \ar[r] & \kk(G/L^g)
\\
& [x]_{L} \ar@{|->}[r] & [x\cdot g]_{L^g}\,.
}}
\end{equation}
This is a standard computation, using the adjunction $\Ind_H^G\adj\Res^G_H$ and the Mackey formula for~$\Res^G_H(\kk(G/K))\simeq\oplus_{[g]\in \doublequot HGK}\,\kk(H/H\cap {}^{g\!}K)$.
\end{Rec}

We can now begin our analysis of the spectrum of the tt-category~$\cK(G)$.

\begin{Prop}
\label{Prop:index-invertible}%
Let $G\le G'$ be a subgroup of index invertible in~$\kk$. Then the map $\Spc(\Res^{G'}_{G})\colon \Spc(\cK(G))\to \Spc(\cK(G'))$ is surjective.
\end{Prop}
\begin{proof}
This is a standard argument.
For a subgroup $G\le G'$, the restriction functor $\Res^{G'}_{G}$ has a two-sided adjoint $\Ind_{G}^{G'}\colon \cK(G)\to \cK(G')$ such that the composite of the unit and counit of these adjunctions $\Id\to \Ind\Res\to \Id$ is multiplication by the index. If the latter is invertible, it follows that $\Res^{G'}_{G}$ is a faithful functor.
The result now follows from~\cite[Theorem~1.3]{balmer:surjectivity}.
\end{proof}
\begin{Cor}
Let $\kk$ be a field of characteristic zero and $G$ be a finite group. Then $\Spc(\cK(G))=\ast$ is a singleton.
\end{Cor}
\begin{proof}
Direct from \Cref{Prop:index-invertible} since $\Spc(\cK(1;\kk))=\Spc(\Dperf(\kk))=\ast$.
\end{proof}

\begin{Rem}
\label{Rem:reduc}%
In view of these reductions, the fun happens with coefficients in a field~$\kk$ of positive characteristic~$p$ dividing the order of~$G$.
\end{Rem}

Let us now identify what the derived category tells us about $\SpcKG$.
\begin{Not}
\label{Not:Kac(G)}%
We can define a tt-ideal of~$\cK(G)=\Kb(\perm(G;\kk)^\natural)$ by
\[
\cK_\ac(G):=\SET{x\in\cK(G)}{x\textrm{ is \emph{acyclic} as a complex of $kG$-modules}}.
\]
It is the kernel of the tt-functor $\Upsilon_G\colon \cK(G)\to \Db(\kk G):=\Db(\mmod{\kkG})$ induced by the inclusion $\perm(G;\kk)^\natural\hook\mmod{\kkG}$ of the additive category of $p$-permutation $\kkG$-modules inside the abelian category of all finitely generated $\kkG$-modules.
\end{Not}

\begin{Rec}
\label{Rec:K(G)/Kac}%
The canonical functor induced by~$\Upsilon_G$ on the Verdier quotient
\[
\frac{\cK(G)}{\cK_\ac(G)}\too \Db(\kk G)
\]
is an equivalence of tt-categories. This is \cite[Theorem~5.13]{balmer-gallauer:resol-small}. In other words,
\begin{equation}
\label{eq:Upsilon_G}%
\Upsilon_G\colon \cK(G)\onto \Db(\kk G)
\end{equation}
realizes the derived category of finitely generated $\kkG$-modules as a localization of our~$\cK(G)$, away from the Thomason subset~$\Supp(\Kac(G))$ of~\eqref{eq:Z_G}.
Following Neeman-Thomason, the above localization~\eqref{eq:Upsilon_G} is the compact part of a finite localization of the corresponding `big' tt-categories $\cT(G)\onto \K\Inj(\kkG)$, where the latter is the homotopy category of complexes of injectives. See~\cite[Remark~4.21]{balmer-gallauer:resol-big}. We return to this localization of big categories in \Cref{Rec:J-lambda}.
\end{Rec}

We want to better understand the tt-ideal of acyclics~$\Kac(G)$ and in particular show that it has closed support.

\begin{Rec}
Let $H\le G$ be a subgroup. Tensor-induction $\tInd_H^G$ is a standard process to turn $\kk H$-modules into $\kk G$-modules, see for instance~\cite[\S\,I.3.15]{benson:representation-cohomology}. It can be applied to \emph{complexes} of modules, as spelled out in~\cite[\S\,~II.4.1]{benson:representation-cohomology} or~\cite[\S\,3]{balmer-gallauer:resol-small}.
Recall that tensor-induction is not additive and therefore the functor $\tInd_H^G\colon\Ch(\kk H)\to\Ch(\kk G)$ does not preserve contractibility of complexes. This is why the following construction does not produce trivial objects.
\end{Rec}
\begin{Cons}
\label{Cons:kos}%
Let $H\le G$ be a subgroup. We define a complex of~$\kkG$-modules by tensor-induction (recall \Cref{Conv:light})
\[
\sH=\sHG:=\tInd_H^G(0\to \kk\xto{1} \kk\to 0)
\]
where $0\to \kk\xto{1} \kk\to 0$ is non-trivial in homological degrees~1 and~0; hence $\sH$ lives in degrees between~$[G\!:\!H]$ and~0.
The underlying complex of $k$-vector spaces is $\otimes_{G/H}(0\to\kk\xto{1} \kk\to 0)$.
Since $H$ acts trivially on~$k$, the action of~$G$ on~$\sH$ is the action of~$G$ by permutation of the factors $\otimes_{G/H}(0\to\kk\xto{1} \kk\to 0)$. This can also be described as a Koszul complex. For every $0\le d\le[G\!:\!H]$, the complex~$\sH$ in degree~$d$ is the $k$-vector space $\Lambda^d(\kk(G/H))$ of dimension~${[G:H]\choose d}$. If we choose a numbering of the elements of~$G/H=\{v_1,\ldots,v_{[G:H]}\}$ then $\sH_d$ has a $k$-basis $\SET{v_{i_1}\wedge\cdots\wedge v_{i_d}}{1\le i_1<\cdots<i_d\le [G\!:\!H]}$. The canonical diagonal action of~$G$ permutes this basis but introduces signs when re-ordering the $v_i$'s so that indices increase. When $p=2$ these signs are irrelevant. When $p>2$, every such `sign-permutation' $\kkG$-module is isomorphic to an actual permutation $\kkG$-module (by changing some signs in the basis, see~\cite[Lemma~3.8]{balmer-gallauer:resol-small}).
\end{Cons}
\begin{Prop}
\label{Prop:s(H;G)}%
Let $H\le G$ be a subgroup. Then~$\sHG$ is an acyclic complex of finitely generated permutation $\kkG$-modules which is concentrated in degrees between $[G\!:\!H]$ and~$0$ and such that it is~$k$ in degree~0 and~$\kk(G/H)$ in degree~1.
\end{Prop}
\begin{proof}
See \Cref{Cons:kos}. Exactness is obvious since the underlying complex of $\kk$-modules is $(0\to k\to k\to 0)\potimes{[G:H]}$.
The values in degrees~$0,1$ are immediate.
\end{proof}

\begin{Exa}
\label{Exa:s(1;G)}
We have $\kos[G]{G}=0$ in~$\cK(G)$. The complex $\kos[G]{1}$ is an acyclic complex of permutation modules that was important in~\cite[\S\,3]{balmer-gallauer:resol-small}:
\[
\kos[G]{1}=\xymatrix{\cdots 0\ar[r] & P_n\ar[r] & \cdots \ar[r] & P_2 \ar[r] & \kk G \ar[r] & \kk \ar[r] & 0 \cdots
}
\]
\end{Exa}

\begin{Lem}
\label{Lem:kos(N;G)}%
Let $H\normaleq G$ be a normal subgroup and $H\le K\le G$. Then $\kos[G]{K}\cong \Infl^{G/H}_G(\kos[G/H]{K/H})$.
In particular, $\kos[G]H\cong\Infl^{G/H}_G(\kos[G/H]{1})$.
\end{Lem}
\begin{proof}
The construction of~$\kos[G]{K}=\otimes_{G/K}(0\to\kk\xto{1} \kk\to 0)$ depends only on the $G$-set~$G/K$ which is inflated from the $G/H$-set $(G/H)/(K/H)$.
\end{proof}

In fact, $\sHG$ is not only exact. It is split-exact on~$H$. More generally:

\begin{Lem}
\label{Lem:Res(s(G;H))}%
For all subgroups $H,K\le G$ and all choices of representatives in $K\bs G/H$, we have a non-canonical isomorphism of complexes of~$\kk K$-modules
\[
\Res^G_K(\sHG)\simeq\ \bigotimes_{[g]\in K\bs G/H}\ \kos[K]{K\cap {}^{g\!}H}.
\]
In particular, if $K\le_G H$, we have $\Res^G_K(\sHG)=0$ in~$\cK(K)$.
\end{Lem}
\begin{proof}
By the Mackey formula for tensor-induction \cite[Proposition~I.3.15.2.(v)]{benson:representation-cohomology},
we have in~$\Ch(\perm(K;\kk))$
\[
\Res^G_K(\sHG)\simeq\ \bigotimes_{[g]\in K\bs G/H}\ \tInd_{K\cap {}^{g\!}H}^K \big( {}^{g\!}\Res^H_{K\cap{}^{g\!}H}(0\to \kk\xto{1}\kk\to 0)\big).
\]
The result follows since $\Res(0\to \kk\xto{1}\kk\to 0)=(0\to \kk\xto{1}\kk\to 0)$.
If $K\le_GH$, the factor $\kos[K]{K}$ appears in the tensor product and $\kos[K]{K}=0$ in~$\cK(K)$.
\end{proof}

We record a general technical argument that we shall use a couple of times.
\begin{Lem}
\label{Lem:s-generates}%
Let $\cA$ be a rigid tensor category and $s=(\cdots s_2\to s_1\to s_0 \to 0\cdots)$ a complex concentrated in non-negative degrees.
Let $x\in\Ch(\cA)$ be a bounded complex such that $s_1\otimes x=0$ in~$\Kb(\cA)$. Then there exists $n\gg0$ such that $s_0\potimes{n}\otimes x$ belongs to the smallest thick subcategory $\ideal{s}'$ of~$\K(\cA)$ that contains~$s$ and is closed under tensoring with~$\Kb(\cA)\cup\{s\}$ in~$\K(\cA)$.
In particular, if $s\in\Kb(\cA)$ is itself bounded, then $s_0\potimes{n}\otimes x$ belongs to the tt-ideal $\ideal{s}$ generated by~$s$ in~$\Kb(\cA)$.
\end{Lem}
\begin{proof}
Let $u:=s_{\geq 1}[-1]$ be the truncation of~$s$ such that $s=\cone(d:u\to s_0)$.
Similarly we have $u=\cone(u_{\geq 1}[-1]\to s_1)$.
Note that $u_{\ge 1}$ is concentrated in positive degrees.
Since $x\otimes s_1=0$ we have $u\otimes x\cong u_{\ge 1}\otimes x$ in~$\K(\cA)$ and thus
\[
u\potimes{n}\otimes x\cong (u_{\geq 1})\potimes{n}\otimes x
\]
for all~$n\ge 0$. For $n$ large enough there are no non-zero maps of complexes from $(u_{\geq 1})\potimes{n}\otimes x$ to $s_0\potimes{n}\otimes x$, simply because the former `moves' further and further away to the left and $x$ is bounded. So $d\potimes{n}\otimes x\colon u\potimes{n}\otimes x\too s_0\potimes{n}\otimes x$ is zero in~$\cK(\cA)$.

Let $\cL$ be the tt-subcategory of~$\K(\cA)$ generated by~$\Kb(\cA)\cup\{s\}$; then $\ideal{s}'$ is a tt-ideal in~$\cL$, and similarly we write $\ideal{\cone(d\potimes{n})}'$ for the tt-ideal in~$\cL$ generated by $\cone(d\potimes{n})$.
By the argument above, we have $s_0\potimes{n}\otimes x\in\ideal{\cone(d\potimes{n})}'\subseteq\ideal{s}'$
where the last inclusion holds by the octahedron axiom.
\end{proof}

\begin{Cor}
\label{Cor:s-generates}%
Let $\cA$ be a rigid tensor category and $\cI\subseteq \Kb(\cA)$ a tt-ideal. Let $s\in\cI$ be a (bounded) complex concentrated in non-negative degrees such that
\begin{enumerate}[label=\rm(\arabic*), ref=\rm(\arabic*)]
\item
\label{it:s-generates-1}%
$\supp(s_0)\supseteq\supp(\cI)$ in~$\Spc(\Kb(\cA))$ (for instance if $s_0=\unit_{\cA}$), and
\smallbreak
\item
\label{it:s-generates-2}%
$\supp(s_1)\cap\supp(\cI)=\varnothing$, meaning that $s_1\otimes x=0$ in~$\Kb(\cA)$ for all~$x\in\cI$.
\end{enumerate}
Then $s$ generates $\cI$ as a tt-ideal in~$\Kb(\cA)$, that is, $\supp(\cI)=\supp(s)$ in~$\Spc(\Kb(\cA))$.
\end{Cor}
\begin{proof}
Let $x\in \cI$. By~\ref{it:s-generates-2}, \Cref{Lem:s-generates} gives us $s_0\potimes{n}\otimes x\in\ideal{s}$ for $n\gg0$. Hence $\supp(s_0)\cap \supp(x)\subseteq \supp(s)$. By~\ref{it:s-generates-1} we have $\supp(x)\subseteq\supp(s_0)$. Therefore $\supp(x)=\supp(s_0)\cap \supp(x)\subseteq\supp(s)$. In short $x\in\ideal{s}$ for all~$x\in\cI$.
\end{proof}

We apply this to the object $s=\sHG$ of \Cref{Cons:kos}.
\begin{Prop}
\label{Prop:Ker(Res)}%
For every subgroup~$H\le G$, the object $\sHG$ generates the tt-ideal $\Ker(\Res^G_H)$ of~$\cK(G)$.
\end{Prop}
\begin{proof}
We apply \Cref{Cor:s-generates} to $\cI=\Ker(\Res^G_H)$ and $s=\sHG$. We have $s\in\cI$ by \Cref{Lem:Res(s(G;H))}. Conditions~\ref{it:s-generates-1} and~\ref{it:s-generates-2} hold since $s_0=\kk$ and $s_1=\kk(G/H)$ by \Cref{Prop:s(H;G)} and Frobenius gives $s_1\otimes\cI=\kk(G/H)\otimes\cI=\Ind_H^G\Res^G_H(\cI)=0$.
\end{proof}

We can apply the above discussion to $H=1$ and $\cI=\Ker(\Res^G_1)=\cK_\ac(G)$.
\begin{Prop}
\label{Prop:VeeG}%
The tt-functor $\Upsilon_G \colon \cK(G)\onto\Db(\kkG)$ induces an \emph{open} inclusion $\upsilon_G\colon \VG\hook \SpcKG$ where $\VG=\Spc(\Db(\kkG))\cong\Spech(\rmH^\sbull(G,\kk))$. The closed complement of~$\VG$ is the support of~$\kos[G]{1}=\tInd_1^G(0\to \kk\xto{1}\kk\to 0)$.
\end{Prop}
\begin{proof}
The homeomorphism $\Spc(\Db(\kkG))\cong \Spech(\rmH^\sbull(G,\kk))$ follows from the tt-classification~\cite{benson-carlson-rickard:tt-class-stab(kG)}; see~\cite[Theorem~57]{balmer:icm}. By \Cref{Rec:K(G)/Kac}, the map $\upsilon_G:=\Spc(\Upsilon_G)$ is a homeomorphism onto its image, and the complement of this image is $\supp(\cK_\ac(G))=\supp(\kos[G]{1})$, by \Cref{Prop:Ker(Res)} applied to $H=1$. In particular, $\supp(\cK_\ac(G))$ is a closed subset, not just a Thomason.
\end{proof}
\begin{Rem}
The notation for the so-called \emph{cohomological open}~$\VG$ has been chosen to evoke the classical \emph{projective support variety}~$\mathcal{V}_G(\kk)=\Proj(\rmH^\sbull(G,\kk))\cong\Spc(\stmod(\kkG))$, which consists of~$\VG$ without its unique closed point, $\rmH^+(G;\kk)$.
\end{Rem}

We can also describe the kernel of restriction for classes of subgroups.
\begin{Cor}
\label{Cor:Ker(Res)}%
For every collection $\cH$ of subgroups of~$G$, we have an equality of tt-ideals in~$\cK(G)$
\[
\bigcap_{H\in\cH}\Ker(\Res^G_H)=\big\langle\ \bigotimes_{H\in\cH}\ \sHG\ \big\rangle.
\]
\end{Cor}
\begin{proof}
This is direct from \Cref{Prop:Ker(Res)} and the general fact that $\ideal{x}\cap \ideal{y}=\ideal{x\otimes y}$. (In the case of $\cH=\varnothing$, the intersection is~$\cK(G)$ and the tensor is~$\unit$.)
\end{proof}

\section{Restriction, induction and geometric fixed-points}
\label{sec:Spc(Res)}

In the previous section, we saw how much of $\SpcKG$ comes from~$\Db(\kkG)$. We now want to discuss how much is controlled by restriction to subgroups, to see how far the `standard' strategy of~\cite{balmer-sanders:SH(G)-finite} gets us.

\begin{Rem}
\label{Rem:not}%
The tt-categories~$\cK(G)$ and $\Db(\kk G)$, as well as the Weyl groups~$\WGH$ are functorial in~$G$. To keep track of this, we adopt the following notational system.

Let $\alpha\colon G\to G'$ be a group homomorphism.
We write $\alpha^*\colon \cK(G')\to \cK(G)$ for restriction along~$\alpha$, and similarly for~$\alpha^*\colon \Db(\kk G')\to \Db(\kk G)$.
When applying the contravariant $\Spc(-)$, we simply denote $\Spc(\alpha^*)$ by~$\alpha_*\colon\SpcKG\to \Spc(\cK(G'))$ and similarly for $\alpha_*\colon \Vee{G}\to \Vee{G'}$ on the spectrum of derived categories.

As announced, Weyl groups $\WGH=(N_G H)/H$ of subgroups~$H\le G$ will play a role.
Since $\alpha(N_G H)\le N_{G'}(\alpha(H))$, every homomorphism $\alpha\colon G\to G'$ induces a homomorphism $\bar \alpha\colon \WGH\to \Weyl{G'}{\alpha(H)}$.
Combining with the above, these homomorphisms $\bar{\alpha}$ define functors~$\bar{\alpha}^*$ and maps~$\bar{\alpha}_*$.
For instance, $\bar\alpha_*\colon \Vee{\WGH}\to \Vee{\Weyl{G'}{\alpha(H)}}$ is the continuous map induced on $\Spc(\Db(\kk(-)))$ by $\bar{\alpha}\colon \WGH\to \Weyl{G'}{\alpha(H)}$.

Following tradition, we have special names when $\alpha$ is an inclusion, a quotient or a conjugation. For the latter, we choose the lightest notation possible.

\smallbreak
\begin{enumerate}[wide, labelindent=0pt, label=\rm(\alph*), ref=\rm(\alph*)]
\item
For \emph{conjugation}, for a subgroup $G\le G'$ and an element ${x}\in G'$, the isomorphism $c_{x}\colon G\isoto G^{x}$ induces a tt-functor $c_{x}^*\colon \cK(G^{x})\isoto \cK(G)$ and a homeomorphism %
\[
\xymatrix@R=.1em{
(-)^{x}:=(c_{x})_*=\Spc(c_{x}^*)\colon
& \Spc(\cK(G)) \ar[r]^-{\sim}
& \Spc(\cK(G^{x}))
\\
& \cP \ar@{|->}[r]
& \cP^{x}.
}
\]
Note that if $x=g\in G$ belongs to~$G$ itself, the functor $c_g^*\colon \cK(G)\to \cK(G)$ is isomorphic to the identity and therefore we get the useful fact that
\begin{equation}\label{eq:no-conj}%
\qquad g\in G \qquad\Longrightarrow\qquad \cP^g=\cP\qquad\textrm{for all~}\cP\in \Spc(\cK(G)).
\end{equation}
Similarly we have a conjugation homeomorphism $\gp\mapsto \gp^x$ on the cohomological opens $\Vee{G}\isoto \Vee{G^x}$, which is the identity if~$x\in G$.
When $H\le G$ is a further subgroup then conjugation yields homeomorphisms~$\Vee{\Weyl{G}{H}}\isoto \Vee{\Weyl{G^x}{H^x}}$ still denoted $\cP\mapsto \cP^x$.
Again, if $x=g\in N_{G}H$, so $[g]_H$ defines an element in~$\Weyl{G}{H}$, the equivalence $(c_g)_*\colon \Db(\Weyl{G}{H})\isoto \Db(\Weyl{G}{H})$ is isomorphic to the identity. Thus
\begin{equation}\label{eq:no-conj-2}%
\qquad g\in N_G(H) \qquad\Longrightarrow\qquad \gp^g=\gp\qquad\textrm{for all~}\gp\in\Vee{\Weyl{G}{H}}.
\end{equation}

\smallbreak
\item
\label{it:rho}%
For \emph{restriction}, take $\alpha$ the inclusion $K\hook G$ of a subgroup. We write
\begin{equation}
\label{eq:rho}%
\rho_K=\rho^G_K:=\Spc(\Res^G_K)\,\colon\ \Spc(\cK(K))\to\SpcKG
\end{equation}
and similarly for derived categories.
When $H\le K$ is a subgroup, we write $\bar{\rho}_K\colon \Vee{\Weyl{K}{H}}\to \Vee{\WGH}$ for the map induced by restriction along~$\Weyl{K}{H}\hook \WGH$.
Beware that $\rho_K$ is not necessarily injective, already on~$\Vee{K}\to \Vee{G}$, as `fusion' phenomena can happen: If $g\in G$ normalizes~$K$, then $\cQ$ and~$\cQ^g$ in~$\Vee{K}$ have the same image in~$\VG$ by~\eqref{eq:no-conj} but are in general different in~$\Vee{K}$.

\smallbreak
\item
For \emph{inflation}, let $N\normaleq G$ be a normal subgroup and let~$\alpha=\proj\colon G\onto G/N$ be the quotient homomorphism.
We write
\begin{equation}
\label{eq:pi}%
\pi^{G/N}=\pi^{G/N}_G:=\Spc(\Infl^{G/N}_G)\,\colon\ \SpcKG\to\Spc(\cK(G/N))
\end{equation}
and similarly for derived categories. For $H\le G$ a subgroup, we write $\bar{\pi}^{G/N}_G\colon \Vee{\Weyl{G}{H}}\to \Vee{\Weyl{(G/N)}{(HN/N)}}$ for the map induced by~$\overline{\proj}\colon \Weyl{G}{H}\to\Weyl{(G/N)}{(HN/N)}$.
(Note that this homomorphism is not always surjective, \eg\ with $G=D_8$ and $N\simeq C_2^{\times 2}$.)
\end{enumerate}
\end{Rem}

\begin{Rec}
\label{Rec:tt-ring}%
One verifies that the $\Res^G_H\adj \Ind_H^G$ adjunction is monadic, see for instance~\cite[\S\,4]{balmer:tt-separable}, and that the associated monad~$A_H\otimes-$ is separable, where $A_H:=\kk(G/H)=\Ind^G_H\kk\in\perm(G;\kk)$. The ring structure on~$A_H$ is given by the usual unit~$\eta\colon \kk\to \kk(G/H)$, mapping $1$ to~$\sum_{\gamma\in G/H}\,\gamma$, and the multiplication~$\mu\colon A_H\otimes A_H\to A_H$ that is characterized by $\mu(\gamma\otimes\gamma)=\gamma$ and $\mu(\gamma\otimes\gamma')=0$ for all~$\gamma\neq\gamma'$ in~$G/H$.
The ring $A_H$ is separable and commutative. The tt-category $\MMod{A_H}=\Mod_{\cK(G)}(A_H)$ of~$A_H$-modules in~$\cK(G)$ identifies with~$\cK(H)$, in such a way that extension-of-scalars to~$A_H$ (\ie along $\eta$) coincides with restriction~$\Res^G_H$.
Similarly, extension-of-scalars along the isomorphism $c_{g\inv}\colon A_{H^g}\isoto A_{H}$, being an equivalence, is the inverse of its adjoint, that is~$((c_{g\inv})^*)\inv=c_g^*$, hence is the conjugation tt-functor~$c_g^*\colon \cK(H^g)\isoto \cK(H)$ of \Cref{Rem:not}.
\end{Rec}

\begin{Prop}
\label{Prop:Spc-Res}%
The continuous map~$\rho_H\colon \Spc(\cK(H))\to \SpcKG$ of~\eqref{eq:rho} is a closed map and for every $y\in \cK(H)$, we have $\rho_H(\supp(y))=\supp(\Ind_H^G(y))$ in~$\SpcKG$. In particular, $\Img(\rho_H)=\supp(\kk(G/H))$. Moreover, there is a coequalizer of topological spaces (independent of the choices of representatives~$g$)
\[
\coprod_{[g]\in \doublequot{H}{G}{H}}\Spc(\cK(H\cap {}^{g\!}H)) \ \rightrightarrows\ \Spc(\cK(H))\ \xto{\rho_H} \ \supp(k(G/H))
\]
where the two left horizontal maps are, on the $[g]$-component, induced by the restriction functor and by conjugation by~$g$ followed by restriction, respectively.
\end{Prop}
\begin{proof}
We invoke~\cite[Theorem~3.19]{balmer:tt-separable}. In particular, we have a coequalizer
\begin{equation}
\label{eq:aux-equaliz}%
\Spc(\MMod{A_H\otimes A_H})\rightrightarrows\Spc(\MMod{A_H}) \to \supp(A_H)
\end{equation}
where the two left horizontal maps are induced by the canonical ring morphisms $A_H\otimes \eta$ and $\eta\otimes A_H\colon A_H\to A_H\otimes A_H$.
For any choice of representatives $[g]\in \HGH$ the Mackey isomorphism
\[
\bigoplus_{[g]\in \doublequot{H}{G}{H}}A_{H\cap {}^{g\!}H}\isoto A_H\otimes A_H
\]
maps~$[x]_{H\cap {}^{g\!}H}$ to $[x]_H\otimes[x\cdot g]_H$. We can then plug this identification in~\eqref{eq:aux-equaliz}. The second homomorphism $\eta\otimes A_H$ followed by the projection onto the factor indexed by~$[g]$ becomes the composite $A_H\xto{c_{g\inv}} A_{{}^{g\!}H} \xinto{\eta} A_{H\cap {}^{g\!}H}$.
See \Cref{Rec:tt-ring}.
\end{proof}

\begin{Cor}
\label{Cor:Spc-Res}
For $\cP,\cP'\in\Spc(\cK(H))$ we have $\rho_H(\cP)=\rho_H(\cP')$ in $\SpcKG$ if and only if there exists $g\in G$ and $\cQ\in\Spc(\cK(H\cap {}^{g\!}H))$ such that
\[
\cP=\rho^{H}_{H\cap {}^{g\!}H}(\cQ)
\qquadtext{and}
\cP'=\big(\rho^{{}^{g\!}H}_{H\cap {}^{g\!}H}(\cQ)\big)^g
\]
using \Cref{Rem:not} for the notation $(-)^g\colon \Spc(\cK({}^{g\!}H))\isoto \Spc(\cK(H))$.
\end{Cor}
\begin{proof}
This is~\cite[Corollary~3.12]{balmer:tt-separable}, which implies the set-theoretic part of the coequalizer of \Cref{Prop:Spc-Res}.
\end{proof}

We single out a particular case.
\begin{Cor}
\label{Cor:Spc-Res-central}
If $H\le Z(G)$ is central in $G$ (for example, if $G$ is abelian) then restriction induces a closed immersion $\rho_H\colon\Spc(\cK(H))\hook\Spc(\cK(G))$.
\qed
\end{Cor}

\begin{Rem}
\label{Rem:geom-fixed-pts}%
In view of \Cref{Prop:Spc-Res}, the image of the map induced by restriction $\Img(\rho_H)=\supp(\kk(G/H))$ coincides with the support of the tt-ideal generated by~$\Ind_H^G(\cK(H))$.
Following the construction of the \emph{geometric} fixed-points functor~$\Phi^G\colon \SHc(G)\to \SHc$ in topology, we can consider the Verdier quotient
\[
\tilde{\cK}(G):=\frac{\cK(G)}{\ideal{\Ind_H^G(\cK(H))\mid H\lneqq G}}
\]
obtained by modding-out, in tensor-triangular fashion, everything induced from all proper subgroups~$H$.
This tt-category $\tilde{\cK}(G)$ has a smaller spectrum than~$\cK(G)$, namely the `geometric open' of the preamble, the complement in~$\SpcKG$ of the closed subset $\cup_{H\lneqq G}\Img(\rho_H)$ covered by proper subgroups.
This method has worked nicely in~\cite{balmer-sanders:SH(G)-finite,barthel-greenlees-hausmann:SH(G)-compact-Lie,psw:derived-mackey} because, in those instances, this Verdier quotient is equivalent to the non-equivariant version of the tt-category under consideration.
However, this fails for~$\tilde{\cK}(G)$, for instance $\tilde\cK(C_2)$ is \emph{not} equivalent to $\cK(1)=\Db(\kk)$:
\[
\frac{\SHc(G)}{\ideal{\Ind_H^G(\SHc(H))\mid H\lneqq G}} \cong\SHc
\qquadtext{but}
\frac{\cK(G)}{\ideal{\Ind_H^G(\cK(H))\mid H\lneqq G}}\not\cong \cK(1).
\]
For small groups, for instance for cyclic $p$-groups~$C_{p^n}$, the tt-category $\tilde{\cK}(G)$ is reasonably complicated and one could still compute $\SpcKG$ through an analysis of~$\tilde{\cK}(G)$.
However, the higher the $p$-rank, the harder it becomes to control~$\tilde{\cK}(G)$.

One can already see the germ of the problem with~$G=C_2$, see~\eqref{eq:C_2}:
\[
\Spc(\cK(C_2))=\qquad\vcenter{\xymatrix@R=1em@C=.5em{
{\color{Brown}\cM(C_2)}
\ar@{-}@[Gray][rd]
&&
{\color{OliveGreen}\cM(1)}
\\
& {\color{OliveGreen}\cP}
\ar@{-}@[Gray][ru]
}}
\]
We have given names to the three primes.
The only proper subgroup is~$H=1$ and the image of~$\rho_1=\Spc(\Res_1)$ is simply the single closed point~$\{\cM(1)\}=\supp(\kk C_2)$.
Chopping off this induced part, leaves us with the open $\Spc(\tilde\cK(C_2))=\{\cM(C_2),\cP\}$.
So geometric fixed points $\Phi^{C_2}\colon\cK(C_2)\to \tilde\cK(C_2)$ detects both of these points.
(This also proves that $\tilde\cK(G)\neq\cK(1)=\Db(\kk)$ since $\Db(\kk)$ would have only one point in its spectrum.)
However there is no need for a tt-functor detecting $\cM(C_2)$ \emph{and~$\cP$ again}, since $\cP$ is already in the cohomological open~$\Vee{C_2}$ detected by ~$\Db(\kk C_2)$.
In other words, geometric fixed points see \emph{too much}, not too little: The target category~$\tilde{\cK}(G)$ is too complicated in general.
And as the group grows, this issue only gets worse, as the reader can check with Klein-four in~\Cref{Exa:Klein4}.

In conclusion, we need tt-functors better tailored to the task, namely tt-functors that detect just what is missing from~$\VG$.
In the case of~$C_2$, we expect a tt-functor to~$\Db(\kk)$, to catch~$\cM(C_2)$, but for larger groups the story gets more complicated and involves more complex subquotients of~$G$, as we explain in the next section.
\end{Rem}

\section{Modular fixed-points functors}
\label{sec:modular-fixed-pts}%

Motivated by \Cref{Rem:geom-fixed-pts}, we want to find a replacement for geometric fixed points in the setting of modular representation theory.
In a nutshell, our construction amounts to taking classical Brauer quotients~\cite[\S\,1]{broue:permutation-brauer} on the level of permutation modules and then passing to the tt-categories~$\cK(G)$ and~$\cT(G)$.
We follow a somewhat different route than~\cite{broue:permutation-brauer} though, more in line with the construction of the geometric fixed-points discussed in \Cref{Rem:geom-fixed-pts}.
We hope some readers will benefit from our exposition.

It is here important that $\chara(\kk)=p$ is positive.

\begin{War}
\label{Rem:only-p-subgroups}%
A tt-functor $\Psi^H\colon \cK(G)\to \cK(\WGH)$ such that $\Psi^H(\kk(X))\cong\kk(X^H)$, as in~\eqref{eq:Psi^H}, cannot exist unless $H$ is a $p$-subgroup.
Indeed, if $P\le G$ is a $p$-Sylow then since $[G\!:\!P]$ is invertible in~$\kk$, the unit $\unit=\kk$ is a direct summand of~$\kk(G/P)$ in~$\cK(G)$.
A tt-functor~$\Psi^H$ cannot map~$\unit$ to zero.
Thus $\Psi^H(\kk(G/P))=\kk((G/P)^H)$ must be non-zero, forcing $(G/P)^H\neq\varnothing$.
If $[g]\in G/P$ is fixed by~$H$ then $H^g\le P$ and therefore $H$ must be a $p$-subgroup.
(If $\chara(\kk)=0$ this would force $H=1$.)
\end{War}

\begin{Rec}
\label{Rec:family}%
A collection~$\cF$ of subgroups of $G$ is called a \emph{family} if it is closed under conjugation and subgroups.
For instance, given $H\le G$, we have the family
\[
\cF_{H}=\SET{K\le G}{(G/K)^H=\varnothing}=\SET{K\le G}{H\not\le_G K}.
\]
For $N\normaleq G$ a normal subgroup, it is~$\cF_N=\SET{K\le G}{N\not\le K}$.
\end{Rec}

In view of \Cref{Rem:only-p-subgroups}, we must focus attention on $p$-subgroups.
The following standard lemma would not be true without the characteristic $p$ hypothesis.

\begin{Lem}
\label{Lem:infl-stab}%
Let $N\normaleq G$ be a normal $p$-subgroup.
Let $H,K\le G$ be subgroups such that~$N\le H$ and~$N\not\le K$.
Then every $\kkG$-linear homomorphism that factors as $f\colon \kk(G/H)\xto{\ell}\kk(G/K)\xto{m}\kk$ must be zero.
\end{Lem}
\begin{proof}
By \Cref{Rec:Hom} and $\kk$-linearity, we can assume that $m$ is the augmentation and that $\ell=\eps\circ c_g\circ \eta$ as in~\eqref{eq:Hom-f_g}, where $g\in G$ is some element, where we set $L=H\cap {}^{g\!}K$ and where $\eps\colon \kk(G/L^g)\onto \kk(G/K)$, $c_g\colon \kk(G/L)\isoto \kk(G/L^g)$ and $\eta\colon \kk(G/H)\into \kk(G/L)$ are the usual maps, using $L\le H$ and $L^g\le K$. The composite $m\circ \eps\circ c_g$ is an augmentation map again, hence our map $f$ is the composite
\[
f\,\colon \quad\kk(G/H) \xinto{\eta} \kk(G/L) \xonto{\eps} \kk.
\]
So $f$ maps~$[e]_H$ to $\sum_{\gamma\in H/L} 1=|H/L|$ in~$\kk$.
Now, the $p$-group $N\le H$ acts on the set $H/L$ by multiplication on the left. This action has no fixed point, for otherwise we would have $N\le_H L\le_G K$ and thus $N\le K$, a contradiction. Therefore the $N$-set $H/L$ has order divisible by~$p$. So $|H/L|=0$ in~$\kk$ and~$f=0$ as claimed.
\end{proof}
\begin{Prop}
\label{Prop:infl-pstab}%
Let $N\normaleq G$ be a normal $p$-subgroup.
Then the permutation category of the quotient~$G/N$ is an additive quotient of the permutation category of~$G$.
More precisely, consider $\proj(\cF_N)=\add^\natural\SET{\kk(G/K)}{K\in\cF_N}$, the closure of~$\SET{\kk(G/K)}{N\not\le K}$ under direct sum and summands in~$\perm(G;\kk)^\natural$.
Consider the additive quotient of~$\perm(G;\kk)^\natural$ by~this $\otimes$-ideal.\,{\rm\,(\footnote{\,Keep the same objects as $\perm(G;\kk)^\natural$ and define Hom groups by modding out all maps that factor via objects of~$\proj(\cF_N)$, as in the ordinary construction of the stable module category.})}
Then the composite
\begin{equation}
\label{eq:modular-fixed-equiv}%
\vcenter{\xymatrix@C=4em{
\perm(G/N;\kk)^\natural \ar@{ >->}[r]^-{\ \Infl^{G/N}_G}
& \perm(G;\kk)^\natural \ar@{->>}[r]^-{\quot}
& \frac{\perm(G;\kk)^\natural}{\proj(\cF_N)}
}}
\end{equation}
is an equivalence of tensor categories.
\end{Prop}

\begin{proof}
By the Mackey formula and since~$\cF_N$ is a family, $\proj(\cF_N)$ is a tensor ideal, hence $\quot$ is a tensor-functor.
Inflation $\Infl^{G/N}_G\colon \perm(G/N;\kk)^\natural\to \perm(G;\kk)^\natural$ is also a tensor-functor. It is moreover fully faithful with essential image the subcategory $\add^\natural\SET{\kk(G/H)}{N\le H}$. So we need to show that the composite
\[
\add^\natural\SET{\kk(G/H)}{N\le H}
\hook
\perm(G;\kk)^\natural \onto
\frac{\perm(G;\kk)^\natural}{\add^\natural\SET{\kk(G/K)}{N\not\le K}}
\]
is an equivalence.
Both functors in the composite are full.
The composite is faithful by \Cref{Lem:infl-stab}, rigidity, additivity and the Mackey formula.
Essential surjectivity is then easy (idempotent-completion is harmless since the functor is fully-faithful).
\end{proof}

\begin{Cons}
\label{Cons:fixed-pts}%
Let $N\normaleq G$ be a normal $p$-subgroup. The composite of the additive quotient functor with the inverse of the equivalence of \Cref{Prop:infl-pstab} yields a tensor-functor on the categories of $p$-permutation modules
\begin{equation}
\label{eq:Psi^N-perm}%
\Psi^N\colon\perm(G;\kk)^\natural\onto \frac{\perm(G;\kk)^\natural}{\proj(\cF_N)}\isoto\perm(G/N;\kk)^\natural.
\end{equation}
Applying the above degreewise, we get a tt-functor on homotopy categories~$\Kb(-)$
\[
\Psi^N=\Psi^{N\inn G}\colon\cK(G)\too \cK(G/N).
\]
\end{Cons}

\begin{Rem}
\label{Rem:Brauer}%
Following up on \Cref{Rem:geom-fixed-pts}, we have constructed~$\Psi^N$ by modding-out in \emph{additive} fashion this time, everything induced from subgroups not containing~$N$. We did it on the `core' additive category and only then passed to homotopy categories. Such a construction would not make sense on bounded derived categories, as $\Psi^N$ has no reason to preserve acyclic complexes.

The classical Brauer quotient seems different at first sight. It is typically defined at the level of individual $\kkG$-modules~$M$ by a formula like
\begin{equation}
\label{eq:dirty-brauer}%
\coker\big(\oplus_{Q\lneq N} M^Q \xto{\,(\mathrm{Tr}_{Q}^N)_Q\,} M^N\big).
\end{equation}
A priori, this definition uses the ambient abelian category of modules and one then needs to verify that it preserves permutation modules, the tensor structure, etc.
Our approach is a categorification of~\eqref{eq:dirty-brauer}: \Cref{Prop:infl-pstab} recovers the category~$\perm(G/N;\kk)^\natural$ as a tensor-additive quotient of~$\perm(G;\kk)^\natural$, at the categorical level, not at the individual module level. Amusingly, one can verify that it yields the same answer (\Cref{Prop:fixed-pts}) -- a fact that we shall not use at all.
\end{Rem}

We relax the condition that the $p$-subgroup is normal in the standard way.

\begin{Def}
\label{Def:fixed-pts}%
Let $H\le G$ be an arbitrary $p$-subgroup. We define the \emph{modular (or Brauer) $H$-fixed-points functor} by the composite
\[
\Psi^{H\inn G}\,\colon\quad \cK(G)\xto{\Res^G_{N_G H}}\cK(N_G H)\xto{\Psi^{H\inn N_G H}}\cK(\WGH)
\]
where $N_G H$ is the normalizer of~$H$ in~$G$ and $\WGH=(N_G H)/H$ its Weyl group. The second functor comes from \Cref{Cons:fixed-pts}.
Note that $\Psi^{H\inn G}$ is computed degreewise, applying the functors $\Res^G_{N_G H}$ and $\Psi^{H\inn N_G H}$ at the level of~$\perm(-;\kk)^\natural$.
\end{Def}
\begin{Rem}
We prefer the phrase `modular fixed-points' to `Brauer fixed-points', out of respect for L.\,E.\,J.\ Brouwer and his fixed points.
It also fits nicely in the flow: naive fixed-points, geometric fixed-points, modular fixed-points.
Finally, the phrase `Brauer quotient' would be unfortunate, as $\Psi^H\colon\cK(G)\to \cK(\WGH)$ is \emph{not} a quotient of categories in any reasonable sense.
\end{Rem}

Let us verify that our~$\Psi^H$ linearize the $H$-fixed-points of~$G$-sets, as promised.
\begin{Prop}
\label{Prop:fixed-pts}%
Let $H\le G$ be a $p$-subgroup. The following square commutes up to isomorphism:
\[
\xymatrix@C=4em{
\gasets[G] \ar[r]^-{\kk(-)} \ar[d]_-{(-)^H}
& \perm(G;\kk)^\natural \ \ar@{^(->}[r] \ar[d]^-{\Psi^H}
& \cK(G) \ar[d]^-{\Psi^H}
\\
\gasets[(\WGH)] \ar[r]^-{\kk(-)}
& \perm(\WGH;\kk)^\natural\ \ar@{^(->}[r]
& \cK(\WGH).
}
\]
In particular, for every~$K\le G$, we have an isomorphism of~$\kk(\WGH)$-modules
\begin{equation}
\label{eq:fixed-pts}%
\Psi^H(\kk(G/K))\cong \kk((G/K)^H)=\kk(N_G(H,K)/K).
\end{equation}
This module is non-zero if and only if $H$ is subconjugate to~$K$ in~$G$.
\end{Prop}
\begin{proof}
We only need to prove the commutativity of the left-hand square.
As restriction to a subgroup commutes with linearization, we can assume that $H\normaleq G$ is normal. Let $X$ be a $G$-set. Consider its $G$-subset~$X^H$ (which is inflated from~$G/H$).
Inclusion yields a morphism in~$\perm(G;\kk)$, natural in~$X$,
\begin{equation}\label{eq:fixed-example-aux}%
f_X\colon\kk(X^H)\to \kk(X).
\end{equation}
We claim that this morphism becomes an isomorphism in the quotient $\frac{\perm(G;\kk)^\natural}{\proj(\cF_H)}$. By additivity, we can assume that $X=G/K$ for $K\le G$. It is a well-known exercise that $(G/K)^H=N_G(H,K)/K$, which in the normal case $H\normaleq G$ boils down to $G/K$ or $\varnothing$, depending on whether $H\le K$ or not, \ie whether $K\notin\cF_H$ or $K\in\cF_H$. In both cases, $f_X$ becomes an isomorphism (an equality or $0\isoto \kk(G/K)$, respectively) in the quotient by~$\proj(\cF_H)$. Hence the claim.

Let us now discuss the commutativity of the following diagram
\[
\xymatrix@C=3em{
\gasets[G] \ar[r]^-{\kk(-)} \ar[dd]_-{(-)^H}
& \perm(G;\kk)^\natural \ar@{->>}[rd]^-{\quot} \ar@{=}[rr]
&& \perm(G;\kk)^\natural \ar[dd]^-{\Psi^H}_-{\textrm{(Def.\,\ref{Def:fixed-pts})}}
\\
& \perm(G;\kk)^\natural \ar@{->>}[r]^-{\quot}
& \frac{\perm(G;\kk)^\natural}{\proj(\cF)}
\\
\gasets[G/H] \ar[r]^-{\kk(-)}
& \perm(G/H;\kk)^\natural \ar[u]^-{\Infl^{G/H}_G} \ar[ru]_-{\cong}^-{\textrm{(Cor.\,\ref{Prop:infl-pstab})\ \ }} \ar@{=}[rr]
&& \perm(G/H;\kk)^\natural }
\]
The module~$\kk(X^H)$ in~\eqref{eq:fixed-example-aux} can be written more precisely as $\kk(\Infl^{G/H}_G(X^H))\cong \Infl^{G/H}_G\kk(X^H)$. So the first part of the proof shows that the left-hand `hexagon' of the diagram commutes, \ie the two functors $\gasets[G]\to \frac{\perm(G;\kk)^\natural}{\proj(\cF)}$ are isomorphic. The result follows by definition of~$\Psi^H$, recalled on the right-hand side.
\end{proof}

Here is how modular fixed points act on restriction.
\begin{Prop}
\label{Prop:Psi-Res}%
Let $\alpha:G\to G'$ be a homomorphism and $H\le G$ a $p$-subgroup. Set $H'=\alpha(H)\le G'$. Then the following square commutes up to isomorphism
\[
\xymatrix@R=2em{\cK(G') \ar[r]^-{\alpha^*} \ar[d]_-{\Psi^{H',G'}}
& \cK(G) \ar[d]^-{\Psi^{H\inn G}}
\\
\cK(\Weyl{G'}{H'}) \ar[r]^-{\bar \alpha^*}
& \cK(\WGH).
}
\]
\end{Prop}
\begin{proof}
Exercise. This already holds at the level of $\perm(-;\kk)^\natural$.
\end{proof}
\begin{Cor}
\label{Cor:Psi-Infl}%
Let $N\normaleq G$ be a normal $p$-subgroup. Then the composite functor $\Psi^N\circ \Infl^{G/N}_G\colon \cK(G/N)\to \cK(G)\to \cK(G/N)$ is isomorphic to the identity.
Consequently, the map $\Spc(\Psi^H)$ is a split injection retracted by~$\Spc(\Infl^{G/H}_G)$.
\end{Cor}
\begin{proof}
Apply \Cref{Prop:Psi-Res} to $\alpha\colon G\onto G/N$ and $H=N$, and thus $H'=1$.
The second statement is just contravariance of~$\Spc(-)$.
\end{proof}

Composition of two `nested' modular fixed-points functors almost gives another modular fixed-points functor. We only need to beware of Weyl groups.
\begin{Prop}
\label{Prop:Psi-Psi}%
Let $H\le G$ be a $p$-subgroup and $\bar{K}=K/H$ a $p$-subgroup of~$\WGH$, for $H\le K\le N_G H$. Then there is a canonical inclusion
\[
\Weyl{(\WGH)}{\bar{K}}=(N_{\WGH}\bar{K})/\bar{K} \hook (N_G K)/K=\WGK
\]
and the following square commutes up to isomorphism
\[
\xymatrix@R=2em{
\cK(G) \ar[r]^-{\Psi^{H\inn G}} \ar[d]_-{\Psi^{K\inn G}}
& \cK(\WGH) \ar[d]^-{\Psi^{\bar{K}\inn\WGH}}
\\
\cK(\WGK) \ar[r]^-{\Res}
& \cK\big(\Weyl{(\WGH)}{\bar{K}}\big).
}
\]
\end{Prop}
\begin{proof}
The inclusion comes from $N_{N_{G}(H)}K\hook N_{G}K$ and the rest is an exercise. Again, everything already holds at the level of~$\perm(-;\kk)^\natural$.
\end{proof}
\begin{Cor}
\label{Cor:Psi-Psi}%
Let $H\le K\le G$ be two $p$-subgroups with $H\normaleq G$ normal. Then $\Weyl{(G/H)}{(K/H)}\cong \Weyl{G}{K}$ and the following diagram commutes up to isomorphism
\[
\xymatrix@R=2em{
\cK(G) \ar[r]^-{\Psi^{H\inn G}} \ar[rd]_-{\Psi^{K\inn G}}
& \cK(G/H) \ar[d]^-{\Psi^{K/H\inn\, G/H}}
\\
& \cK\big(\Weyl{G}{K}\big).
}
\]
\end{Cor}
\begin{proof}
The surjectivity of the canonical inclusion $\Weyl{(G/H)}{(K/H)}\hook \Weyl{G}{K}$ of \Cref{Prop:Psi-Psi} holds since $H$ is normal in~$G$. The result follows.
\end{proof}

\begin{Rem}
\label{Rem:Psi-big}%
We have essentially finished the proof of \Cref{Prop:modular-fixed-points-intro}.
It only remains to verify that there are variants of the constructions and results of this section for the big categories of \Cref{Rec:perm}.
For a normal $p$-subgroup $N\normaleq G$, the canonical functor on big additive categories
\begin{equation}\label{eq:Psi-big}%
\Add^\natural(\SET{\kk(G/H)}{N\le H})\to\frac{\Perm(G;\kk)^\natural}{\Proj(\cF_N)}
\end{equation}
is an equivalence of tensor categories, where
\[
\Proj(\cF_N)=\Add^\natural\SET{\kk(G/K)}{N\not\le K}
\]
is the closure of $\proj(\cF_N)$ under coproducts and summands.
Since the tensor product commutes with coproducts, $\Proj(\cF_N)$ is again a $\otimes$-ideal in $\Perm(G;\kk)^\natural$.
Fullness and essential surjectivity of~\eqref{eq:Psi-big} are easy, and faithfulness reduces to the finite case by compact generation.
(A map $f:P\to Q$ in $\Perm(G;\kk)$ is zero if and only if all composites $P'\xto{\alpha} P\xto{f} Q$ are zero, for $P'$ finitely generated. Such a composite necessarily factors through a finitely generated direct summand of~$Q$, etc.)
As a consequence, the analogue of \Cref{Prop:infl-pstab} also holds for big categories.

Let us write $\cS(G)$ for $\K(\Perm(G;\kk))=\K(\Perm(G;\kk)^\natural)$, which is a compactly generated tt-category with compact unit.
(Compactly generated is not obvious: see~\cite[Remark~5.12]{balmer-gallauer:Dperm} or the precursors J\o rgensen~\cite{jorgensen:projectives} and Neeman~\cite{neeman:K-flat}.)
By the above discussion, the modular fixed-points functor with respect to a $p$-subgroup $H\le G$ extends to the big categories~$\cS(-)$:
\[
\Psi^H=\Psi^{H\inn G}\colon\ \cS(G)\xto{\Res^G_{N_G H}}\cS(N_G H)\to \K\left(\frac{\Perm(N_G H;\kk)}{\Proj(\cF_H)}\right)\xleftarrow[\sim]{\Infl^{\WGH}_{N_G H}}\cS(\WGH).
\]
Note that $\Psi^H$ is a tensor triangulated functor that commutes with coproducts and that maps~$\cK(G)$ into~$\cK(\WGH)$.
It follows that it restricts to $\Psi^H\colon\DPerm(G;\kk)\to\DPerm(\WGH;\kk)$.
The analogues of \Cref{Prop:fixed-pts,Prop:Psi-Res,Cor:Psi-Infl,Prop:Psi-Psi,Cor:Psi-Psi} all continue to hold for both $\cS(-)$ and $\DPerm(-;\kk)$.
\end{Rem}

This finishes our exposition of modular fixed-points functors~$\Psi^H$ on derived categories of permutation modules. We now start using them to analyze the tt-geometry.
First, we apply them to the Koszul complexes~$\kos[G]{K}$ of \Cref{Cons:kos}.
\begin{Lem}
\label{Lem:Psi^H-of-s}
Let $H,K\le G$ be two subgroups, with $H$ a $p$-subgroup.
\begin{enumerate}[label=\rm(\alph*), ref=\rm(\alph*)]
\item
\label{it:Psi^H-1}%
If $H\not\le_G K$, then $\Psi^H(\sKG)$ generates~$\cK(\WGH)$ as a tt-ideal.
\smallbreak
\item
\label{it:Psi^H-2}%
If $H\le_G K$, then $\Psi^H(\sKG)$ is acyclic in~$\cK(\WGH)$.
\smallbreak
\item
\label{it:Psi^H-3}%
If $H\sim_G K$, then $\Psi^H(\sKG)$ generates~$\Kac(\WGH)$ as a tt-ideal.
\end{enumerate}
\end{Lem}
\begin{proof}
For~\ref{it:Psi^H-1}, we have $N_G(H,K)=\varnothing$ and thus $\Psi^H(\kk(G/K))=0$ by \Cref{Prop:fixed-pts}. It follows that $\Psi^H(\sKG)=(\cdots \to \ast \to 0\to \kk\to 0)$ by \Cref{Prop:s(H;G)}. Thus the $\otimes$-unit $\unit_{\cK(\WGH)}=\kk[0]$ is a direct summand of $\Psi^H(\sKG)$.

For~\ref{it:Psi^H-2} and~\ref{it:Psi^H-3}, by invariance under conjugation, we can assume that $H\le K$. Let $N:=N_G H$ be the normalizer of~$H$. We have by \Cref{Lem:Res(s(G;H))} that
\begin{equation}
\label{eq:res-s}
\Psi^{H\inn G}(\sKG)=\Psi^{H\inn N}\Res^G_{N}(\sKG)\simeq\!\!\bigotimes_{[g]\in N\backslash G/K}\kern-.5em\Psi^{H\inn N}\big(\kos[N]{N\cap {}^{g\!}K}\big).
\end{equation}
For the index $g=e$ (or simply $g\in N_GK$), we can use $H\normaleq N\cap K$ and compute
\[
\begin{array}{rll}
\Psi^{H\inn N}(\kos[N]{N\cap K}) \kern-.7em & \cong\Psi^{H\inn N}(\Infl^{N/H}_N\kos[N/H]{(N\cap K)/H}) & \textrm{by \Cref{Lem:kos(N;G)}}
\\[2pt]
& \cong \kos[N/H]{(N\cap K)/H} & \textrm{by \Cref{Cor:Psi-Infl}.}
\end{array}
\]
As this object is acyclic in~$\cK(N/H)$ so is the tensor in~\eqref{eq:res-s}. Hence~\ref{it:Psi^H-2}. Continuing in the special case~\ref{it:Psi^H-3} with $H=K$, we have $(N\cap K)/H=1$ and the above $\kos[N/H]{1}$ generates~$\Kac(N/H)$ by \Cref{Prop:Ker(Res)}. It suffices to show that all the other factors in the tensor product~\eqref{eq:res-s} generate the whole~$\cK(\WGH)$. This follows from Part~\ref{it:Psi^H-1} applied to the group~$N$; indeed when $g\in G\sminus N$ we have $H\not\le_N N\cap {}^{g\!}H$ (as $H\le_N N\cap {}^{g\!}H$ and $H\normaleq N$ would imply $H={}^{g\!}H$).
\end{proof}

\section{Conservativity via modular fixed-points}
\label{sec:conservativity}%

In this section, we explain why the spectrum of~$\cK(G)$ is controlled by modular fixed-points functors~$\Psi^H$ together with the localizations~$\Upsilon_G\colon \cK(G)\onto \Db(\kkG)$. It stems from a conservativity result on the `big' category~$\cT(G)=\DPerm(G;\kk)$, namely \Cref{Thm:conservativity}, for which we need some preparation.

\begin{Lem}
\label{Lem:nilpotence}%
Suppose that $G$ is a $p$-group. Let $H\le G$ be a subgroup and let $\bar{G}=\WGH$ be its Weyl group. The modular $H$-fixed-points functor $\Psi^H\colon \perm(G;\kk)^\natural\to \perm(\bar{G};\kk)^\natural$ induces a ring homomorphism
\begin{equation}\label{eq:Psi-nilpotence}%
\Psi^H\colon \End_{\kk G}(\kk(G/H))\too \End_{\kk \bar{G}}(\kk(\bar{G})).
\end{equation}
This homomorphism is surjective with nilpotent kernel: $(\ker(\Psi^H))^n=0$ for $n\gg1$. More precisely, it suffices to take~$n\in\bbN$ such that $\Rad(\kk G)^n=0$.
\end{Lem}
\begin{proof}
The reader can check this with Brauer quotients. We outline the argument.
By \eqref{eq:fixed-pts} we have $\Psi^H(\kk(G/H))\cong \kk(N_G(H,H)/H)=\kk(\bar{G})$, so the problem is well-stated.
\Cref{Rec:Hom} provides $\kk$-bases for both vector spaces in~\eqref{eq:Psi-nilpotence}, namely
\[
\{\,f_{g}=\eps\circ c_g\circ\eta\,\}_{[g]\in \HGH}
\qquadtext{and}
\{\,c_{\bar{g}}\,\}_{\bar{g}\in \bar{G}}
\]
using the notation of~\eqref{eq:Hom-f_g} and~\eqref{eq:Hom-c_g}.
The homomorphism~$\Psi^H$ in~\eqref{eq:Psi-nilpotence} respects those bases.
Even better, it is a bijection from the part of the first basis indexed by~$H\bs(N_G H)/H$ onto the second basis, and it maps the rest of the first basis to zero.
Indeed, when $g\in N_G H$, we have $f_g=c_g$ and $\Psi^H(f_g)=\Psi^H(c_g)=c_{\bar{g}}$ for $\bar{g}=[g]_H$.
On the other hand, when $g\in G\sminus N_G H$ then $\Psi^H(f_g)=0$, by the factorization~\eqref{eq:Hom-f_g} and the fact that $\Psi^H(\kk(G/L))=0$ for $L=H\cap {}^{g\!} H$ with $g\notin N_G H$, using again \eqref{eq:fixed-pts}.
Hence \eqref{eq:Psi-nilpotence} is surjective and $\ker(\Psi^H)$ has basis $\{f_g=\eps\circ c_g\circ\eta\}_{[g]\in \HGH,\,g\notin N_G H}$.
One easily verifies that such an~$f_g$ has image contained in $\Rad(\kk G)\cdot \kk(G/H)$, using that $H\cap {}^{g\!}H$ is strictly smaller than~$H$.
Composing $n$ such generators $f_{g_1}\circ\cdots \circ f_{g_n}$ then maps $\kk(G/H)$ into $\Rad(\kkG)^n\cdot \kk(G/H)$ which is eventually zero for~$n\gg1$, since~$G$ is a $p$-group.
\end{proof}

We now isolate a purely additive result that we shall of course apply to the case where $\Psi$ is a modular fixed-points functor.

\begin{Lem}
\label{Lem:conservativity}%
Let $\Psi\colon \cA\to \cD$ be an additive functor between additive categories. Let $\cB,\cC\subseteq \cA$ be full additive subcategories such that:
\begin{enumerate}[label=\rm(\arabic*), ref=\rm(\arabic*)]
\item
\label{it:cons-1}%
Every object of~$\cA$ decomposes as $B\oplus C$ with $B\in\cB$ and~$C\in\cC$.
\item
\label{it:cons-2}%
The functor~$\Psi$ vanishes on~$\cC$, that is, $\Psi(\cC)=0$.
\item
\label{it:cons-3}%
The restricted functor $\Psi\restr{\cB}\colon \cB\to \cD$ is full with nilpotent kernel.\,{\rm(\footnote{\,There exists $n\gg1$ such that if $n$ composable morphisms $f_1,\ldots,f_n$ in~$\cB$ all go to zero in~$\cD$ under~$\Psi$ then their composite $f_n\circ \cdots\circ f_1$ is zero in~$\cB$.})}
\end{enumerate}
Let $X\in\Chain(\cA)$ be a complex such that $\Psi(X)$ is contractible in~$\Chain(\cD)$. Then $X$ is homotopy equivalent to a complex in~$\Chain(\cC)$.
\end{Lem}
\begin{proof}
Decompose every $X_i=B_i\oplus C_i$ in~$\cA$, using~\ref{it:cons-1}, for all~$i\in \bbZ$.
We are going to build a complex on the objects~$C_i$ in such a way that $X_\sbull$ becomes homotopy equivalent to~$C_\sbull$ in~$\Chain(\cA^\natural)$, where $\cA^\natural$ is the idempotent-completion of~$\cA$. As both~$X_\sbull$ and~$C_\sbull$ belong to~$\Chain(\cA)$, this proves the result.
By~\ref{it:cons-2}, the complex $\cdots \to \Psi(B_i)\to \Psi(B_{i-1})\to \cdots$ is isomorphic to~$\Psi(X)$, hence it is contractible. So each $\Psi(B_i)$ decomposes in~$\cD^\natural$ as $D_{i}\oplus D_{i-1}$ in such a way that the differential $\Psi(B_i)=D_i\oplus D_{i-1} \too \Psi(B_{i-1})=D_{i-1}\oplus D_{i-2}$ is just $\smat{0&1\\0&0}$. Since $\Psi\restr{\cB}\colon \cB\to \cD$ is full with nilpotent kernel by~\ref{it:cons-3}, we can lift idempotents. In other words, we can decompose each $B_i$ in the idempotent-completion~$\cB^\natural$ (hence in~$\cA^\natural$) as
\[
B_i\simeq B_i'\oplus B_i''
\]
with $\Psi(B_i')\simeq D_i$ and $\Psi(B_i'')\simeq D_{i-1}$ in a compatible way with the decomposition in~$\cD^\natural$.
This means that when we write the differentials in~$X$ in components in~$\cA^\natural$
\[
\cdots \to X_i=B_i'\oplus B_i''\oplus C_i \xto{\smat{*&b_i&*\\ *&*&*\\ *&*&*}} X_{i-1}=B_{i-1}'\oplus B_{i-1}''\oplus C_{i-1} \to \cdots
\]
the component $b_i\colon B_i''\to B_{i-1}'$ maps to the isomorphism $\Psi(B_i'')\simeq D_{i-1}\simeq \Psi(B_{i-1}')$ in~$\cD^\natural$.
Hence $b_i$ is already an isomorphism in~$\cB^\natural$ by~\ref{it:cons-3} again.
(Note that~\ref{it:cons-3} passes to~$\cB^\natural\to\cD^\natural$.)
Using elementary operations on~$X_i$ and~$X_{i-1}$ we can replace $X$ by an isomorphic complex in~$\cA^\natural$ of the form
\begin{equation}\label{eq:aux-cons-1}%
\cdots \to X_{i+1} \to B_i'\oplus B_i''\oplus C_i \xto{\smat{0&b_i&0\\ *&0&*\\ *&0&*}} B_{i-1}'\oplus B_{i-1}''\oplus C_{i-1} \to X_{i-2}\to \cdots
\end{equation}
This being a complex forces the `previous' differential $X_{i+1}\to X_i$ to be of the form~$\smat{*&*&*\\ 0&0&0\\ *&*&*}$ and the `next' differential $X_{i-1}\to X_{i-2}$ to be of the form~$\smat{0&*&*\\ 0&*&*\\ 0&*&*}$. We can then remove from $X$ a direct summand in~$\Chain(\cA^\natural)$ that is a homotopically trivial complex of the form~$\cdots0\to B_i''\isoto B_{i-1}'\to 0\cdots$.

The reader might be concerned about how to perform this reduction in all degrees at once, since we do not put boundedness conditions on~$X$ (thus preventing the `obvious' induction argument). The solution is simple. Do the above for all differentials in \emph{even} indices $i=2j$. By elementary operations on~$X_{2j}$ and~$X_{2j-1}$ for all~$j\in \bbZ$, we can replace~$X$ up to isomorphism into a complex whose \emph{even} differentials are of the form~\eqref{eq:aux-cons-1}. We then remove the contractible complexes $\cdots 0\to B_{2j}''\isoto B_{2j-1}'\to 0\cdots$.
We obtain in this way a homotopy equivalent complex in~$\cA^\natural$ that we call~$\tilde X$, where $B_i',B_i''\in\cB^\natural$ and~$C_i\in \cC$
\begin{equation}\label{eq:aux-X}%
\cdots \to B_{2j+1}''\oplus C_{2j+1} \xto{\smat{a_{2j+1}&*\\ *&*}} B_{2j}'\oplus C_{2j} \xto{\smat{*&*\\ *&*}} B_{2j-1}''\oplus C_{2j-1} \to \cdots
\end{equation}
in which the even differentials go to zero under~$\Psi$, by the above construction. In particular the homotopy trivial complex $\Psi(\tilde X)\simeq\Psi(X)$ in~$\cD^\natural$ has the form $\cdots \xto{0} \Psi(B_{2j+1}'')\xto{\Psi(a_{2j+1})} \Psi(B_{2j}')\xto {0} \cdots$ hence its odd-degree differentials $\Psi(a_{2j+1})$ are isomorphisms. It follows that $a_{2j+1}\colon B_{2j+1}''\to B_{2j}'$ is itself an isomorphism by~\ref{it:cons-3} again. Using new elementary operations (again in all degrees), we change the odd-degree differentials of the complex~$\tilde X$ in~\eqref{eq:aux-X} into diagonal ones and we remove the contractible summands $0\to B_{2j+1}''\isoto B_{2j}'\to 0$ as before, to get a complex consisting only of the $C_i$ in each degree~$i\in \bbZ$.
In summary, we have shown that~$X$ is homotopy equivalent to a complex $C\in\Chain(\cC)$ inside~$\Chain(\cA^\natural)$, as announced.
\end{proof}

\begin{Rem}
Of course, it would be silly to discuss conservativity of the functors $\{\Psi^H\}_{H\le G}$ since among them we find~$\Psi^1=\Id$. The interesting result appears when each $\Psi^H$ is used in conjunction with the derived category of~$\WGH$, or, in `big' form, its homotopy category of injectives.
Let us remind the reader.
\end{Rem}

\begin{Rec}
\label{Rec:J-lambda}%
In \cite{balmer-gallauer:resol-big}, we prove that the homotopy category of injective $RG$-modules, with coefficients in any regular ring~$R$ (\eg\ our field~$\kk$), is a localization of~$\DPerm(G;R)$. In our case, we have an inclusion $J_G\colon \K\Inj(\kk G)\hook \DPerm(G;\kk)$, inside~$\K(\Perm(G;\kk))$, and this inclusion admits a left adjoint~$\Upsilon_G$
\begin{equation}
\label{eq:J-lambda}%
\vcenter{\xymatrix{
\DPerm(G;\kk) \ar@<-.5em>@{->>}[d]_-{\Upsilon_G}
\\
\K\Inj(\kkG). \ar@<-.5em>@{ >->}[u]_-{J_G}
}}
\end{equation}
This realizes the finite localization of~$\DPerm(G;\kk)$ with respect to the subcategory~$\Kac(G)\subseteq\cK(G)=\DPerm(G;\kk)^c$. In particular, $\Upsilon_G$ preserves compact objects and yields the equivalence~$\Upsilon_G\colon\cK(G)/\Kac(G)\cong \Db(\kkG)\cong\K\Inj(\kkG)^c$ of~\eqref{eq:Upsilon_G}, also denoted~$\Upsilon_G$ for this reason. Note that $\Upsilon_G\circ J_G\cong\Id$ as usual with localizations.

Let $P\le G$ be a subgroup. Observe that induction~$\Ind_P^G$ preserves injectives so that $J_G\circ\Ind_P^G\cong\Ind_P^G\circ J_P$. Taking left adjoints, we see that
\begin{equation}\label{eq:Upsilon-Res}%
\Res^G_P\circ\Upsilon_G\cong \Upsilon_P\circ\Res^G_P.
\end{equation}
\end{Rec}

\begin{Not}
\label{Not:bar-Psi}%
For each $p$-subgroup~$H\le G$, we are interested in the composite
\[
\xymatrix@C=3em{\check{\Psi}^H=\check{\Psi}^{H\inn G}\colon \quad \DPerm(G;\kk) \ar[r]^-{\Psi^{H\inn G}} & \DPerm(\WGH;\kk) \ar@{->>}[r]^-{\Upsilon_{\WGH}} & \K\Inj(\kk(\WGH))}
\]
of the modular $H$-fixed-points functor followed by localization to the homotopy category of injectives~\eqref{eq:J-lambda}. We use the same notation on compacts
\begin{equation}
\label{eq:bar-Psi}%
\xymatrix@C=3em{\check{\Psi}^H=\check{\Psi}^{H\inn G}\colon \quad \cK(G;\kk) \ar[r]^-{\Psi^{H\inn G}} & \cK(\WGH;\kk) \ar@{->>}[r]^-{\Upsilon_{\WGH}} & \Db(\kk(\WGH)).}
\end{equation}
\end{Not}

We are now ready to prove the first important result of the paper.
\begin{Thm}[Conservativity]
\label{Thm:conservativity}%
Let $G$ be a finite group.
The above family of functors~$\check{\Psi}^H\colon \cT(G)\to \K\Inj(\kk(\WGH))$, indexed by all the (conjugacy classes of) $p$-sub\-groups~$H\le G$, collectively detects vanishing of objects of~$\DPerm(G;\kk)$.
\end{Thm}
\begin{proof}
Let $P\le G$ be a $p$-Sylow subgroup. For every subgroup~$H\le P$, we have $\Weyl{P}{H}\hook\WGH$ and $\check{\Psi}^{H\inn P}\circ \Res^G_P$ can be computed as $\Res^{\WGH}_{\Weyl{P}{H}}\circ\check{\Psi}^{H\inn G}$ thanks to \Cref{Prop:Psi-Res} and~\eqref{eq:Upsilon-Res}. On the other hand, $\Res^G_P$ is (split) faithful, as $\Ind_P^G\circ\Res^G_P$ admits the identity as a direct summand. Hence it suffices to prove the theorem for the group~$P$, \ie we can assume that $G$ is a $p$-group.

Let $G$ be a $p$-group and $\cF$ be a family of subgroups (\Cref{Rec:family}). We say that a complex~$X$ in~$\Chain(\Perm(G;\kk))$ is \emph{of type~$\cF$} if every $X_i$ is $\cF$-free, \ie a coproduct of $\kk(G/K)$ for $K\in\cF$. So every complex is of type~$\cF_{\textrm{all}}=\{\textrm{all }H\le G\}$. Saying that $X$ is of type $\cF_1=\varnothing$ means $X=0$.
We want to prove that if $X$ defines an object in $\DPerm(G;\kk)$ and $\check{\Psi}^H(X)=0$ for all~$H\le G$ then $X$ is homotopy equivalent to a complex $X'$ of type~$\cF_1=\varnothing$.
We proceed by a form of `descending induction' on~$\cF$. Namely, we prove:
\smallbreak
\noindent \textbf{Claim}: \textit{Let $X\in \DPerm(G;\kk)$ be a complex of type~$\cF$ for some family~$\cF$ and let $H\in \cF$ be a maximal element of~$\cF$ for inclusion.
  If $\check{\Psi}^H(X)=0$ then $X\cong X'$ is homotopy equivalent to a complex $X'\in\Chain(\Perm(G;\kk))$ of type~$\cF'\subsetneqq \cF$.
}
\smallbreak
By the above discussion, proving this claim proves the theorem. Explicitly, we are going to prove this claim for $\cF'=\cF\cap \cF_H$, that is, $\cF$ with all conjugates of~$H$ removed.
By maximality of~$H$ in~$\cF$, every $K\in\cF$ is either conjugate to~$H$ or in~$\cF'$. Of course, for $H'$ conjugate to~$H$ we have $\kk(G/H')\simeq \kk(G/H)$ in~$\Perm(G;\kk)$.

We apply \Cref{Lem:conservativity} for $\cA=\Add\SET{\kk(G/K)}{K\in\cF}$, $\cB=\Add(\kk(G/H))$, $\cC=\Add\SET{\kk(G/K)}{K\in\cF'}$, $\cD=\Perm(\WGH;\kk)$ and the functor $\Psi=\Psi^H$ naturally.
Let us check the hypotheses of \Cref{Lem:conservativity}. Regrouping the terms $\kk(G/K)$ into those for which $K$ is conjugate to~$H$ and those not conjugate to~$H$, we get Hypothesis~\ref{it:cons-1}.
Hypothesis~\ref{it:cons-2} follows immediately from \eqref{eq:fixed-pts} since~$(G/K)^H=\varnothing$ for every $K\in\cF'$.
Finally, Hypothesis~\ref{it:cons-3} follows from \Cref{Lem:nilpotence} and additivity.
So it remains to show that $\Psi^H(X)$ is contractible.
Since $X$ is of type~$\cF$ and $H$ is maximal, we see that $\Psi^H(X)\in\Chain(\Inj(\kk(\WGH)))$ and applying~$\Upsilon_{\WGH}$ gives the same complex (up to homotopy). In other words, $\check{\Psi}^H(X)=0$ forces $\Psi^H(X)$ to be contractible and we can indeed get the above Claim from \Cref{Lem:conservativity}.
\end{proof}

\section{The spectrum as a set}
\label{sec:spectrum}%

In this section, we obtain the description of all points of~$\SpcKG$, as well as some elements of its topology.
We start with a general fact, which is now folklore.

\begin{Prop}
Let $F\colon \cT\to \cS$ be a coproduct-preserving tt-functor between `big' tt-categories. Suppose that $F$ is conservative. Then $F$ detects $\otimes$-nilpotence of morphisms~$f\colon x\to Y$ in~$\cT$, whose source $x\in \cT^c$ is compact, \ie if $F(f)=0$ in~$\cS$ then there exists $n\gg1$ such that $f\potimes{n}=0$ in~$\cT$. In particlar, $F\colon \cT^c\to \cS^c$ detects nilpotence of morphisms and therefore $\Spc(F)\colon \Spc(\cS^c)\to \Spc(\cT^c)$ is surjective.
\end{Prop}
\begin{proof}
Using rigidity of compacts, we can assume that $x=\unit$. Given a morphism $f\colon \unit\to Y$ we can construct in~$\cT$ the homotopy colimit $Y^\infty:=\hocolim_n Y\potimes{n}$ under the transition maps $f\otimes\id\colon Y\potimes{n}\to Y\potimes{(n+1)}$. Let $f^\infty\colon \unit\to Y^\infty$ be the resulting map. Now since $F(f)=0$ it follows that $F(Y^\infty)=0$ in~$\cS$, as it is a sequential homotopy colimit of zero maps. By conservativity of~$F$, we get $Y^\infty=0$ in~$\cT$. Since $\unit$ is compact, the vanishing of~$f^\infty\colon \unit\to\hocolim Y\potimes{n}$ must already happen at a finite stage, that is, the map $f\potimes{n}\colon \unit \to Y\potimes{n}$ is zero for~$n\gg1$, as claimed.
The second statement follows from this, together with~\cite[Theorem~1.4]{balmer:surjectivity}.
\end{proof}

Combined with our Conservativity \Cref{Thm:conservativity} we get:

\begin{Cor}
\label{Cor:Spc-surjection-derived}%
The family of functors $\check{\Psi}^H\colon\cK(G)\to\Db(k(\WGH))$, indexed by conjugacy classes of $p$-subgroups $H\le G$, detects $\otimes$-nilpotence. So the induced map
\[
\coprod_{H\in\Sub{p}(G)}\Spc(\Db(k(\WGH)))\onto \SpcKG
\]
is surjective.
\qed
\end{Cor}

\begin{Def}
\label{Def:psi}%
Let $H\le G$ be a $p$-subgroup. We write (under \Cref{Conv:light})
\[
\psi^H=\psi^{H\inn G}:=\Spc(\Psi^{H})\colon\Spc(\cK(\WGH))\to \SpcKG
\]
for the map induced by the modular $H$-fixed-points functor. We write
\[
\check{\psi}^H=\check{\psi}^{H\inn G}:=\Spc(\check{\Psi}^{H})\colon\Spc(\Db(\WGH))\to \SpcKG
\]
for the map induced by the tt-functor~$\check{\Psi}^H=\Upsilon_{\WGH}\circ\Psi^{H}$ of \eqref{eq:bar-Psi}.
In other words, $\check{\psi}^H$ is the composite of the inclusion of \Cref{Prop:VeeG} with the above~$\psi^H$
\[
\check{\psi}^H\ \colon \quad \Vee{\WGH}=\Spc(\Db(\kk(\WGH)))\ \overset{\upsilon_{\WGH}}{\hook}\ \Spc(\cK(\WGH))\ \xto{\psi^H}\ \SpcKG.
\]
\end{Def}

\begin{Def}
\label{Def:P(H;p)}%
Let $H\le G$ be a $p$-subgroup and $\gp\in\Vee{\WGH}$ a `cohomological' prime over the Weyl group of~$H$ in~$G$.
We define a point in $\SpcKG$ by
\[
\cP(H,\gp)=\cP_G(H,\gp):=\check{\psi}^H(\gp)=(\check{\Psi}^H)\inv(\gp).
\]
\Cref{Cor:Spc-surjection-derived} tells us that every point of~$\SpcKG$ is of the form~$\cP(H,\gp)$ for some $p$-subgroup~$H\le G$ and some cohomological point~$\gp\in\Vee{\WGH}$. Different subgroups and different cohomological points can give the same~$\cP(H,\gp)$. See \Cref{Thm:all-points}.
\end{Def}

\begin{Rem}
Although we shall not use it, we can unpack the definitions of~$\cP_G(H,\gp)$ for the nostalgic reader.
Let us start with the bijection~$\Vee{G}=\Spc(\Db(\kkG))\cong\Spech(\rmH^\sbull(G,\kk))$.
Let $\gp^{\sbull}\subset \rmH^\sbull(G;\kk)=\End^\sbull_{\Db(\kkG)}(\unit)$ be a homogeneous prime ideal of the cohomology.
The corresponding prime~$\gp$ in~$\Db(\kkG)$ can be described as
\[
\gp=\SET{x\in \Db(\kkG)}{\exists\, \zeta\in\rmH^\sbull(G;\kk)\textrm{ such that }\zeta\notin\gp^{\sbull}\textrm{ and }\zeta\otimes x=0}.
\]
Consequently, the prime~$\cP_G(H,\gp)$ of \Cref{Def:P(H;p)} is equal to
\[
\SET{x\in\cK(G)}{\exists\,\zeta\in\rmH^\sbull(\WGH;\kk)\sminus\gp^{\sbull}\textrm{ such that }\zeta\otimes \Psi^H(x)=0\textrm{ in }\Db(\kk(\WGH))}.
\]
\end{Rem}

\begin{Rem}
\label{Rem:Spc-Res}%
By \Cref{Prop:Psi-Res} and functoriality of~$\Spc(-)$, the primes~$\cP_G(H,\gp)$ are themselves functorial in~$G$. To wit, if $\alpha:G\to G'$ is a group homomorphism and $H$ is a $p$-subgroup of~$G$ then $\alpha(H)$ is a $p$-subgroup of~$G'$ and we have
\begin{equation}
\label{eq:Spc-Res}%
\alpha_*(\cP_G(H,\gp))=\cP_{G'}(\alpha(H),\bar\alpha_*\gp)
\end{equation}
in~$\Spc(\cK(G'))$, where $\alpha_*\colon \SpcKG\to \Spc(\cK(G'))$ and $\bar{\alpha}_*\colon \Vee{\WGH}\to \Vee{\Weyl{G'}{\alpha(H)}}$ are as in \Cref{Rem:not}.
We single out the usual suspects. Fix $H\le G$ a $p$-subgroup.

\begin{enumerate}[wide, labelindent=0pt, label=\rm(\alph*), ref=\rm(\alph*)]
\smallbreak
\item
\label{it:Spc-conj}%
For \emph{conjugation}, let $G\le G'$ and $x\in G'$.
We get $\cP_{G}(H,\gp)^x=\cP_{G^x}(H^x,\gp^x)$ for every~$\gp\in\Vee{\WGH}$.
In particular, when~$x$ belongs to~$G$ itself, we get by~\eqref{eq:no-conj}
\begin{equation}
\label{eq:Spc-conjugation}%
\qquad g\in G \qquad\Longrightarrow\qquad \cP_G(H,\gp)=\cP_G(H^g,\gp^g).
\end{equation}

\smallbreak
\item
For \emph{restriction}, let $K\le G$ be a subgroup containing~$H$ and let $\gp\in \Vee{\Weyl{K}{H}}$ be a cohomological point over the Weyl group of~$H$ in~$K$.
Then we have
\begin{equation}
\label{eq:Spc-Res-K}%
\rho_K(\cP_K(H,\gp))=\cP_G(H,\bar\rho_K(\gp)),
\end{equation}
in $\SpcKG$, where the maps $\rho_K=(\Res^G_K)_*\colon\Spc(\cK(K))\to \SpcKG$ and $\bar\rho_K\colon \Vee{\Weyl{K}{H}}\to \Vee{\WGH}$ are spelled out around~\eqref{eq:rho}.

\smallbreak
\item
For \emph{inflation}, let $N\normaleq G$ be a normal subgroup. Set $\bar G=G/N$ and $\bar H=HN/N$.
Then for every $\gp\in \Vee{\WGH}$, we have
\begin{equation}
\label{eq:Spc-infl}%
\pi^{\bar G}\,(\cP_{G}(H,\gp))=\cP_{\bar G}(\bar H,\overline{\pi}^{\bar G}\,\gp),
\end{equation}
in $\Spc(\cK(\bar G))$, where the maps $\pi^{\bar{G}}=(\Infl^{\bar{G}}_G)_*\colon\SpcKG\to\Spc(\cK(\bar{G}))$ and $\bar{\pi}^{\bar{G}}\colon \Vee{\Weyl{G}{H}}\to \Vee{\Weyl{\bar{G}}{\bar{H}}}$ are spelled out around~\eqref{eq:pi}.
\end{enumerate}
\end{Rem}

Our primes also behave nicely under modular fixed-points maps:
\begin{Prop}
\label{Prop:Spc-Psi}
Let $H\le G$ be a $p$-subgroup and let $H\le K\le N_G H$ defining a `further' $p$-subgroup~$K/H\le \WGH$.
Then for every $\gp\in \Vee{\Weyl{(\WGH)}{(K/H)}}$, we have
\[
\psi^{H\inn G}(\cP_{\WGH}(K/H,\gp))=\cP_G(K,\bar{\rho}(\gp))
\]
in~$\SpcKG$, where $\bar{\rho}=\Spc(\overline{\Res}^{\WGK}_{\Weyl{(\WGH)}{(K/H)}})\colon \Vee{\Weyl{(\WGH)}{(K/H)}}\too\Vee{\Weyl{G}{K}}$. In particular, if $H\normaleq G$ is normal, we have
\[
\psi^{H\inn G}(\cP_{G/H}(K/H,\gp))=\cP_G(K,\gp)
\]
in~$\SpcKG$, using that $\gp\in\Vee{\Weyl{(G/H)}{(K/H)}}=\Vee{\Weyl{G}{K}}$.
\end{Prop}
\begin{proof}
This is immediate from \Cref{Prop:Psi-Psi} and \Cref{Cor:Psi-Psi}.
\end{proof}

The relation between Koszul objects and modular fixed-points functors, obtained in \Cref{Lem:Psi^H-of-s}, can be reformulated in terms of the primes~$\cP_G(H,\gp)$.
\begin{Lem}
\label{Lem:P-to-s}%
Let $H\le G$ be a $p$-subgroup and $\gp\in \Vee{\WGH}$. Let $K\le G$ be a subgroup and~$\sKG$ be the Koszul object of~\Cref{Cons:kos}. Then $\sKG\in\cP_G(H,\gp)$ if and only if~$H\le_G K$. (Note that the latter condition does not depend on~$\gp$.)
\end{Lem}
\begin{proof}
We have seen in \Cref{Lem:Psi^H-of-s}\,\ref{it:Psi^H-2} that if $H\le_G K$ then $\check{\Psi}^H(\sKG)=0$ in~$\Db(\kk(\WGH))$, in which case $\sKG\in(\check{\Psi}^H)\inv(0)\subseteq (\check{\Psi}^H)\inv(\gp)=\cP_G(H,\gp)$ for every~$\gp$.
Conversely, we have seen in \Cref{Lem:Psi^H-of-s}\,\ref{it:Psi^H-1} that if $H\not\le_G K$ then $\check{\Psi}^H(\sKG)$ generates~$\Db(\kk(\WGH))$, hence is not contained in any cohomological point~$\gp$, in which case $\sKG\notin(\check{\Psi}^H)\inv(\gp)=\cP_G(H,\gp)$.
\end{proof}

\begin{Cor}
\label{Cor:P<P'}%
If $\cP_G(H,\gp)\subseteq \cP_G(H',\gp')$ then $H'\le_G H$.
Therefore if $\cP_G(H,\gp)=\cP_G(H',\gp')$ then $H$ and~$H'$ are conjugate in~$G$.
\end{Cor}
\begin{proof}
Apply \Cref{Lem:P-to-s} to $K=H$ twice, for~$H$ being once~$H$ and once~$H'$.
\end{proof}

\begin{Prop}
\label{Prop:bar-psi-injective}%
Let $H\le G$ be a $p$-subgroup. Then the map~$\check{\psi}^H\colon \Vee{\WGH}\to \Spc(\cK(G))$ is injective, that is, $\cP_G(H,\gp)=\cP_G(H,\gp')$ implies $\gp=\gp'$.
\end{Prop}
\begin{proof}
Let $N=N_G H$. By assumption we have $\rho^G_N(\cP_N(H,\gp))=\rho^G_N(\cP_N(H,\gp'))$. By \Cref{Cor:Spc-Res}, there exists $g\in G$ and a prime~$\cQ\in\Spc(\cK(N\cap{}^{g\!}N))$ such that
\begin{equation}
\label{eq:aux-inj-1}%
\cP_N(H,\gp)=\rho_{N\cap{}^{g\!}N}^{N}(\cQ)
\quadtext{and}
\cP_N(H,\gp')=\big(\rho_{N\cap{}^{g\!}N}^{{}^{g\!}N}(\cQ)\big)^g.
\end{equation}
By \Cref{Cor:Spc-surjection-derived} for the group~$N\cap{}^{g\!}N$, there exists a $p$-subgroup $L\le N\cap{}^{g\!}N$ and some~$\gq\in \Vee{\Weyl{(N\cap{}^{g\!}N)}{L}}$ such that $\cQ=\cP_{N\cap{}^{g\!}N}(L,\gq)$.
By \eqref{eq:Spc-Res} we know where such a prime $\cP_{N\cap{}^{g\!}N}(L,\gq)$ goes under the maps~$\rho=\Spc(\Res)$ of~\eqref{eq:aux-inj-1} and, for the second one, we also know what happens under conjugation by~\Cref{Rem:Spc-Res}\,\ref{it:Spc-conj}.
Applying these properties to the above relations~\eqref{eq:aux-inj-1} we get
\begin{equation*}
\label{eq:aux-inj-2}%
\cP_N(H,\gp)=\cP_N(L,\gq')
\quadtext{and}
\cP_N(H,\gp')=\cP_N(L^g,\gq'')
\end{equation*}
for suitable cohomological points $\gq'\in\Vee{\Weyl{N}{L}}$ and $\gq''\in\Vee{\Weyl{N}{L^g}}$ that we do not need to unpack.
By \Cref{Cor:P<P'} applied to the group~$N$, we must have $H\sim_N L$ and $H\sim_N L^g$. But since $H\normaleq N$, this forces $H=L=L^g$ and therefore $g\in N_G H=N$.
In that case, returning to~\eqref{eq:aux-inj-1}, we have $N\cap N^g=N=N^g$ and therefore
\[
\cP_N(H,\gp)=\cQ
\quadtext{and}
\cP_N(H,\gp')=\cQ^g=\cQ
\]
where the last equality uses~$g\in N$ and~\eqref{eq:no-conj}. Hence $\cP_N(H,\gp)=\cQ=\cP_N(H,\gp')$.
As $H$ is normal in~$N$ the map $\psi^{H\inn N}\colon \Spc(\cK(N/H))\too \Spc(\cK(N))$ is split injective by \Cref{Cor:Psi-Infl}, and we conclude that $\gp=\gp'$.
\end{proof}

We can now summarize our description of the set~$\SpcKG$.

\begin{Thm}
\label{Thm:all-points}
Every point in $\SpcKG$ is of the form $\cP_G(H,\gp)$ as in \Cref{Def:P(H;p)}, for some $p$-subgroup $H\le G$ and some point~$\gp\in\Vee{\WGH}$ of the cohomological open of the Weyl group of~$H$ in~$G$.
Moreover, we have $\cP_G(H,\gp)=\cP_G(H',\gp')$ if and only if there exists $g\in G$ such that
$H'=H^g$ and $\gp'=\gp^g$.
\end{Thm}
\begin{proof}
The first statement follows from \Cref{Cor:Spc-surjection-derived}.
For the second statement, the ``if''-direction follows from~\eqref{eq:Spc-conjugation}.
For the ``only if''-direction assume $\cP_G(H,\gp)=\cP_G(H',\gp')$. By \Cref{Cor:P<P'}, this forces $H\sim_G H'$.
Using~\eqref{eq:Spc-conjugation}, we can replace $H'$ by~$H^g$ and assume that $\cP_G(H,\gp)=\cP_G(H,\gp')$ for $\gp,\gp'\in\Vee{\WGH}$.
We can then conclude by \Cref{Prop:bar-psi-injective}.
\end{proof}

Here is an example of support, for the Koszul objects of \Cref{Cons:kos}.
\begin{Cor}
\label{Cor:supp(kos)}%
Let $K\le G$. Then $\supp(\sKG)=\SET{\cP(H,\gp)}{H\not\le_G K}$.
\end{Cor}
\begin{proof}
Since all primes are of the form~$\cP(H,\gp)$, it is a simple contraposition on \Cref{Lem:P-to-s}, for $\cP(H,\gp)\in\supp(\sKG)\Leftrightarrow\sKG\notin \cP(H,\gp)\Leftrightarrow H\not\le_G K$.
\end{proof}

We can use this result to identify the image of $\psi^H$. First, in the normal case:

\begin{Prop}
\label{Prop:Im(psi^N)}%
Let $H\normaleq G$ be a normal $p$-subgroup. Then the continuous map
\[
\psi^{H}=\Spc(\Psi^H)\colon \Spc(\cK(G/H))\to\SpcKG
\]
is a closed immersion, retracted by~$\Spc(\Infl^{G/H}_G)$. Its image is the closed subset
\begin{equation}
\label{eq:Im(psi^N)}%
\Img(\psi^H)=\SET{\cP_G(L,\gp)}{H\le L\in\Sub{p}\!G,\ \gp\in\Vee{\WGL}}=\!\cap_{K\not\ge H}\!\supp(\sKG).\kern-1.1em
\end{equation}
Furthermore, this image of~$\psi^H$ is also the support of the object
\begin{equation}
\label{eq:y-normal}%
\bigotimes_{K\in\cF_H}\sKG
\end{equation}
and it is also the support of the tt-ideal $\cap_{K\in \cF_H}\Ker(\Res^G_K)$.
\end{Prop}
\begin{proof}
By \Cref{Cor:Psi-Infl}, the map $\psi^H$ has a continuous retraction hence is a closed immersion as soon as we know that its image is closed.
So let us prove~\eqref{eq:Im(psi^N)}.

By \Cref{Prop:Spc-Psi} and the fact that all points are of the form~$\cP(L,\gp)$, the image of~$\psi^H$ is the subset $\SET{\cP_G(L,\gp)}{H\le L,\ \gp\in\Vee{\WGL}}$. Here we use $H\normaleq G$.

\Cref{Cor:supp(kos)} tells us that every such point $\cP(L,\gp)$ belongs to the support of~$\sKG$ as long as $L\not\le_G K$, which clearly holds if $H\le L$ and $H\not\le K$. Therefore $\Img(\psi^H)\subseteq \cap_{K\in\cF_H}\supp(\sKG)$.

Conversely, let $\cP(L,\gp)\in \cap_{K\in\cF_H}\supp(\sKG)$ and let us show that $H\le L$. If \ababs, $H\not\le L$ then $L\in\cF_H$ is one of the indices~$K$ that appear in the intersection $\cap_{K\in\cF_H}\supp(\sKG)$. In other words, $\cP(L,\gp)\in\supp(\kos[G]{L})$. By \Cref{Cor:supp(kos)}, this means $L\not\le_G L$, which is absurd. Hence the result.

The `furthermore part' follows: The first claim is~\eqref{eq:Im(psi^N)} since $\supp(x)\cap \supp(y)=\supp(x\otimes y)$ and the second claim follows from \Cref{Cor:Ker(Res)}.
 (For $H=1$, the result does not tell us much, as $\psi^1=\id$ and $\otimes_{\varnothing}=\unit$.)
\end{proof}

Let us extend the above discussion to not necessarily normal subgroups~$H$.

\begin{Not}
\label{Not:zul}%
Let $H\le G$ be an arbitrary subgroup. We define an object of~$\cK(G)$
\begin{equation}
\label{eq:zul}%
\tHG:=\Ind_{N_G H}^G \bigg(\bigotimes_{\SET{K\le N_G H}{H\not\le K}}\kos[N_{G}(H)]{K}\bigg).
\end{equation}
(Note that we use plain induction here, not tensor-induction as in \Cref{Cons:kos}.) If $H\normaleq G$ is normal this $\tHG$ is simply the object displayed in~\eqref{eq:y-normal}.
\end{Not}

\begin{Cor}
\label{Cor:Im(psi^H)}%
Let $H\le G$ be a $p$-subgroup. Then the continuous map
\[
\psi^{H\inn G}=\Spc(\Psi^{H\inn G})\colon \Spc(\cK(\WGH))\to\SpcKG
\]
is a closed map, whose image is $\supp(\tHG)$ where $\tHG$ is as in~\eqref{eq:zul}.
\end{Cor}
\begin{proof}
By definition $\Psi^{H\inn G}=\Psi^{H\inn N_G H}\circ\Res^G_{N_G H}$. We know the map induced on spectra by the second functor~$\Psi^{H\inn N_G H}$ by \Cref{Prop:Im(psi^N)} and we can describe what happens under the closed map~$\Spc(\Res)$ by \Cref{Prop:Spc-Res}.
\end{proof}

We record the answer to a question stated in the Introduction (\Cref{sec:intro-I}):
\begin{Cor}
\label{Cor:supp(Kac)}%
The support of the tt-ideal of acyclics~$\Kac(G)$ is the union of the images of the modular $H$-fixed-points maps~$\psi^H$, for non-trivial $p$-subgroups~$H\le G$.
\end{Cor}
\begin{proof}
The points of~$\SpcKG$ are of the form~$\cP(H,\gp)$. Such primes belong to~$\VG=\SET{\cP(1,\gq)}{\gq\in\VG}$ if and only if~$H$ is trivial. The complement is then $\Supp(\Kac(G))$. Hence $\Supp(\Kac(G))\subseteq \cup_{H\neq 1}\Img(\psi^H)$. Conversely, for every $p$-subgroup $H\neq 1$, the object~$\tHG$ of \Cref{Cor:Im(psi^H)} is acyclic, since the tensor is non-empty and any $\kos[N_G H]{K}$ is acyclic. So $\Img(\psi^H)\subseteq\Supp(\Kac(G))$.
\end{proof}

Let us now describe all closed points of~$\SpcKG$.
\begin{Rem}
\label{Rem:Db(kG)-local}%
Recall that in tt-geometry closed points~$\cM\in \Spc(\cK)$ are exactly the minimal primes for inclusion. Also every prime contains a minimal one.

For instance, the tt-category $\Db(\kkG)$ is local, with a unique closed point~$0=\Ker(\Db(\kkG)\to \Db(\kk))$. (In terms of homogeneous primes in~$\Spech(\rmH^\sbull(G,\kk))$ the zero tt-ideal~$\gp=0$ corresponds to the closed point~$\gp^\sbull=\rmH^+(G,\kk)$.)
\end{Rem}

\begin{Def}
\label{Def:m_H}%
Let $H\le G$ be a $p$-subgroup. (This definition only depends on the conjugacy class of~$H$ in~$G$.)
By \Cref{Prop:Psi-Res}, the following diagram commutes
\begin{equation}
\label{eq:F_H}%
\vcenter{\xymatrix@C=4em{
\cK(G) \ar[r]^-{\Res^G_H} \ar[d]_-{\Psi^{H\inn G}} \ar[rd]|-{\ \bbF^H\ }
& \cK(H) \ar[d]^-{\Psi^{H\inn H}}
\\
\cK(\Weyl{G}{H}) \ar[r]_-{\Res^{\WGH}_1}
& \cK(1)=\Db(\kk).\kern-3em
}}
\end{equation}
We baptize $\bbF^H=\bbF^{H\inn G}$ the diagonal. Its kernel is one of the primes of \Cref{Def:P(H;p)}
\begin{equation}
\label{eq:M(H)}%
\cM(H)=\cM_G(H):=\Ker(\bbF^H)=\cP_G(H,0)
\end{equation}
where $0\in\Spc(\Db(\kk(\WGH)))$ is the zero tt-ideal, \ie the unique closed point of the cohomological open~$\Vee{\WGH}$ of the Weyl group.
(See \Cref{Rem:Db(kG)-local}.)
We can think of $\bbF^H\colon \cK(G)\to \Db(\kk)$ as a tt-residue field functor at the (closed) point~$\cM(H)$.
\end{Def}

\begin{Exa}
\label{Exa:M(1)}%
For $H=1$, we have $\cM(1)=\Ker\big(\Res^G_1:\cK(G)\to \Db(\kk)\big)=\Kac(G)$. In other words, $\cM(1)=\Upsilon_G\inv(0)$ is the image under the open immersion $\upsilon_G\colon \VG\hook \SpcKG$ of \Cref{Prop:VeeG} of the unique closed point~$0\in\VG$ of \Cref{Rem:Db(kG)-local}. In general, a closed point of an open is not necessarily closed in the ambient space. Here $\cM(1)$ is closed since by definition $\{\cM(1)\}=\Img(\rho^G_1)$ where $\rho^G_1=\Spc(\Res^G_1)$. By \Cref{Prop:Spc-Res}, we know that $\Img(\rho^G_1)=\supp(\kk(G))$ is closed.
\end{Exa}

\begin{Exa}
\label{Exa:M(G)}%
For $H=G$ a $p$-group, we can give generators of the closed point
\[
\cM(G)=\ideal{\kk(G/K)\mid K\neq G}.
\]
As $\cM(G)=\ker(\Psi^G:\cK(G)\to\Db(k))$, inclusion~$\supseteq$ follows from \Cref{Prop:fixed-pts}.
For~$\subseteq$, let $X\in\cM(G)$ be a complex that vanishes under~$\Psi^G$.
Splitting the modules~$X_n$ in each homological degree~$n$ into a trivial (\ie a $\kk$-vector space with trivial action) and non-trivial permutation modules, \Cref{Lem:conservativity} shows that $X$ is homotopy equivalent to a complex in the additive category generated by $\kk(G/K)$, $K\neq G$.
\end{Exa}

\begin{Cor}
\label{Cor:closed-pts}%
The closed points of $\SpcKG$ are exactly the tt-primes~$\cM_G(H)$ of \eqref{eq:M(H)} for the $p$-subgroups~$H\le G$. Furthermore, we have $\cM_G(H)=\cM_G(H')$ if and only if~$H$ is conjugate to~$H'$ in~$G$.
\end{Cor}
\begin{proof}
Let us first verify that $\cM_G(H)$ is closed for every~$H\le G$. For $H=1$, we checked it in \Cref{Exa:M(1)}. For $H\neq 1$, we have $\cM_G(H)=\cP_G(H,0)=\Psi^{H}(\cM_{\WGH}(1))$. This gives the result since $\cM_{\WGH}(1)$ is closed in~$\Spc(\cK(\WGH))$, by \Cref{Exa:M(1)} again, and since $\psi^H$ is a closed map by \Cref{Cor:Im(psi^H)}.

Now, every point $\gp\in\Vee{\WGH}$ admits $0$ in its closure in~$\Spc(\Db(\kk(\WGH)))=\Vee{\WGH}$. (See \Cref{Rem:Db(kG)-local}.) By continuity of~$\check{\psi}^H\colon \Vee{\WGH}\to \SpcKG$, it follows that $\check{\psi}^H(0)=\cM_G(H)$ belongs to the closure of~$\check{\psi}^H(\gp)=\cP_G(H,\gp)$, which proves that the $\cM_G(H)$ are the only closed points.

We already saw that $\cP(H,0)=\cP(H',0)$ implies $H\sim_G H'$, in \Cref{Thm:all-points}.
\end{proof}

We wrap up this section on the spectrum by discussing the strata defined by modular fixed-points.
\begin{Prop}
\label{Prop:fibered}%
For every $p$-subgroup~$H\le G$, consider the subset
\[
\Vee{G}(H):=\Img(\check{\psi}^H)=\check{\psi}^H(\Vee{\WGH})
\]
of~$\Spc(\cK(G))$. Then $\cM_G(H)$ is the unique closed point of~$\SpcKG$ that belongs to~$\Vee{G}(H)$.
We have a set-partition indexed by conjugacy classes of $p$-subgroups
\begin{equation}
\label{eq:SpcKG-partition}%
\SpcKG=\coprod_{H\in(\Sub{p}\!G)/G}\ \Vee{G}(H)
\end{equation}
where each $\Vee{G}(H)$ is open in its closure.
\end{Prop}

\begin{proof}
The partition is immediate from \Cref{Thm:all-points}.
Each subset $\Vee{G}(H)=\SET{\cP(H,\gp)}{\gp\in\Vee{\WGH}}$ is a subset of the closed set~$\Img(\psi^H)$.
By \Cref{Cor:supp(Kac)} and \Cref{Prop:Spc-Psi}, the complement of~$\Vee{G}(H)$ in~$\Img(\psi^H)$ consists of the images $\Img(\psi^K)$ for every `further'~$p$-group~$K$, \ie such that $H\lneqq K\le N_G(H)$ and these are closed by \Cref{Cor:Im(psi^H)}.
Thus~$\Vee{G}(H)$ is an open in the closed subset~$\Img(\psi^H)$.
\end{proof}

\begin{Rem}
\label{Rem:filtered}%
\label{Rem:psi^H-V(H)}%
We can use~\eqref{eq:SpcKG-partition} to define a map~$\SpcKG\to (\Sub{p}\!G)/G$.
\Cref{Cor:P<P'} tells us that this map is continuous for the (Alexandrov) topology on $(\Sub{p}\!G)/G$ whose open subsets are the ones stable under subconjugacy.

Moreover, for $H\le G$ a $p$-subgroup, the square
\[
\xymatrix@R=1.5em{
  \Spc(\cK(\WGH))
  \ar[d]
  \ar[r]_-{\psi^H}
  &
  \SpcKG
  \ar[d]
  \\
  (\Sub{p}(\WGH))/(\WGH)\
  \ar@{^(->}[r]
  &
  (\Sub{p}\!G)/G
}
\]
commutes, where the bottom horizontal arrow is the canonical inclusion that sends $H\le K\le N_G(H)$ to $K$.
This follows from \Cref{Prop:Spc-Psi}.
Consequently, while $\psi^H$ might not be injective in general, we still have $(\psi^H)\inv(\Vee{G}(H))=\Vee{\WGH}$.
\end{Rem}

\section{Examples}
\label{sec:examples}%

Although the full treatment of the topology of~$\SpcKG$ will require the additional technology of~\Cref{part:twisted-cohomology}, we can already present the answer for small groups.
Some of the most interesting phenomena are already visible once we reach $p$-rank two in \Cref{Exa:Klein4}.
Let us start with the easy examples.

\begin{Not}
\label{Not:W*}%
Fix an integer $n\geq 0$ and consider the following space $\W^n$ consisting of $2n+1$ points, with specialization relations pointing upward as usual:
\begin{equation}
\label{eq:W*}%
\W^n=\qquad
\vcenter{\xymatrix@R=1em@C=.7em{
{\scriptstyle\gm_0\kern-1em}
& {\bullet} \ar@{-}@[Gray] '[rd] '[rr] '[drrr]
&&
{\bullet}
&{\kern-1em\scriptstyle\gm_1}
&&{\scriptstyle\gm_{n-1}\kern-1em}&\bullet&&\bullet&{\kern-1em\scriptstyle\gm_n}
\\
&{\scriptstyle\gp_1\kern-1em}& {\bullet}&&& {\cdots}& \ar@{-}@[Gray] '[ru] '[rr] '[rrru] &&\bullet&{\kern-1em}\scriptstyle\gp_{n}}}
\end{equation}
The closed subsets of~$\W^n$ are simply the specialization-closed subsets, \ie those that contain a $\gp_i$ only if they contain~$\gm_{i-1}$ and~$\gm_i$.
So the $\gm_i$ are closed points and the $\gp_i$ are generic points of the $n$ irreducible V-shaped closed subsets~$\{\gm_{i-1},\gp_i,\gm_i\}$.
\end{Not}

\begin{Prop}
\label{Prop:cyclic}%
Let $G=C_{p^n}$ be a cyclic $p$-group. Then $\Spc(\cK(C_{p^n}))$ is homeomorphic to the space~$\W^n$ of~\eqref{eq:W*}.

More precisely, if we denote by $1=N_n<N_{n-1}<\cdots<N_0=G$ the $n+1$ subgroups of~$C_{p^n}$\,{\rm(\footnote{\,The numbering of the~$N_i$ keeps track of the index, that is, $G/N_i\cong C_{p^i}$. This choice will allow simple formulas for inflation and fixed-points, and for procyclic groups in~\cite{balmer-gallauer:Artin-finite-fields}.})},
then the points~$\gp_i$ and~$\gm_i$ in~$\SpcKG$ are given by
\[
\gm_i=(\check{\Psi}^{N_i})\inv(0)
\qquadtext{and}
\gp_i=(\check{\Psi}^{N_i})\inv(\Dperf(\kk(G/N_i)))
\]
where $\check{\Psi}^{N}=\Upsilon_{G/N}\circ\Psi^N\colon \cK(G)\to\cK(G/N)\onto\Db(\kk(G/N))$ is the tt-functor~\eqref{eq:bar-Psi}.
\end{Prop}

\begin{proof}
\label{Exa:Cyclic-revisited}%
By \Cref{Prop:fibered}, we have a partition of the spectrum in subsets
\[
\SpcKG=\coprod_{i=0}^n \ \Vee{G}(N_i)=\coprod_{i=0}^n \ \Img(\check{\psi}^{N_i})
\]
and each $\Vee{G}(N_i)$ is homeomorphic to~$\Spc(\Db(\kk G/N_i))=\Vee{G/N_i}$.
For $i>0$, each $\Vee{G}(N_i)$ is a Sierpi\'{n}ski space $\{\gp_i\rightsquigarrow\gm_i=\cM(N_i)\}$, while $\Vee{G}(N_0)$ is a singleton set~$\{\gm_0:=\cM(G)\}$.
In other words, we know the set~$\SpcKG$ has the announced~$2n+1$ points and the unmarked specializations~$\gp_i\rightsquigarrow\gm_i$ below
\begin{equation}
\label{eq:W*-temp}%
\vcenter{\xymatrix@R=1em@C=.7em{
{\scriptstyle\gm_0\kern-1em}
& {\bullet} \ar@{-}@[Gray] '[rd]|-{?} '[rr] '[rrrd]|-{?}
&& {\bullet} &{\kern-1em\scriptstyle\gm_1}
& \cdots
&{\scriptstyle\gm_{i-1}\kern-1em}&\bullet \ar@{-}@[Gray] '[rd]|-{?} '[rr]
&&\bullet&{\kern-1em\scriptstyle\gm_i}
& \cdots
&&\bullet&{\kern-1em\scriptstyle\gm_{n-1}}
&\bullet&{\kern-1em\scriptstyle\gm_n}
\\
&{\scriptstyle\gp_1\kern-1em}& {\bullet}
&&& {\cdots}
&&& \bullet&{\kern-1em}\scriptstyle\gp_{i}
&& {\cdots}
& \ar@{-}@[Gray] '[ru] '[rr]|-{?} '[rrru]
&&\bullet&{\kern-1em}\scriptstyle\gp_{n}}}
\end{equation}
We need to elucidate the topology.
Since all~$\gm_i=\cM(N_i)$ are closed (\Cref{Cor:closed-pts}), we only need to see where each $\gp_i$ specializes for~$1\le i\le n$.
By \Cref{Cor:P<P'}, the point $\gp_i=\cP(N_i,\gp)$ can only specialize to a $\cP(N_j,\gq)$ for $N_j\ge N_i$, that is, to the points $\gm_j$ or~$\gp_j$ for~$j\le i$.
On the other hand, direct inspection using \eqref{eq:fixed-pts} shows that $\supp(\kk(G/N_{i-1}))=\SET{\gm_j}{j\ge i-1}\cup\SET{\gp_j}{j\ge i}$.
This closed subset contains~$\gp_i$ hence its closure.
Combining those two observations, we have
\[
\adhpt{\gp_i}\subseteq \SET{\gm_j,\gp_j}{j\le i}\cap \big(\{\gm_{i-1}\}\cup\SET{\gm_j,\gp_j}{j\ge i}\big)=\{\gm_{i-1},\gp_{i},\gm_i\}.
\]
If any of the $\adhpt{\gp_{i}}$ was smaller than~$\{\gm_{i-1},\gp_{i},\gm_i\}$, that is, if one of the specialization relations $\gp_{i}\rightsquigarrow\gm_{i-1}$ marked with~`$?$' in~\eqref{eq:W*-temp} did not hold, then $\SpcKG$ would be a disconnected space. This would force the rigid tt-category~$\cK(G)$ to be the product of two tt-categories
(see Stevenson~\cite{stevenson:disconnecting}), which is clearly absurd, \eg because $\End_{\cK(G)}(\unit)=\kk$.
\end{proof}

With this identification, we can record the maps $\psi^H$ of~\Cref{Def:psi} and the maps $\rho_K$ and~$\pi^{G/N}$ of \Cref{Rem:Spc-Res}, that relate different cyclic $p$-groups.
\begin{Lem}
\label{Lem:Spc-cyclic}%
Let $n\ge 0$. We identify~$\Spc(\cK(C_{p^n}))$ with~$\W^n$ as in \Cref{Prop:cyclic}.
\begin{enumerate}[label=\rm(\alph*), ref=\rm(\alph*)]
\item
\label{it:Spc-cyclic.fixed-pts}%
Let $0\le i\le n$ and $H=N_i=C_{p^{n-i}}\le C_{p^n}$, so that $C_{p^n}/H\cong C_{p^i}$.
The map~$\psi^H\colon \W^i\to \W^n$ induced by modular fixed points $\Psi^H$ is the inclusion
\[
\psi\colon\W^i\hook\W^n
\]
that catches the left-most points: $\gp_\ell\mapsto \gp_\ell$ and~$\gm_\ell\mapsto \gm_\ell$.
\smallbreak
\item
\label{it:Spc-cyclic.res}%
Let $0\le j\le n$ and $K=C_{p^j}\le C_{p^n}$.
The map~$\rho_K\colon \W^j\to \W^n$ induced by restriction $\Res_K$ is the inclusion
\[
\rho:\W^{j}\hook\W^n
\]
that catches the right-most points: $\gm_\ell\mapsto\gm_{\ell+n-j}$ and $\gp_\ell\mapsto\gp_{\ell+n-j}$.
\smallbreak
\item
\label{it:Spc-cyclic.infl}%
Let $0\le m\le n$. Inflation along $C_{p^n}\onto C_{p^m}$ induces on spectra the map
\[
\pi\colon\W^n\onto\W^m
\]
that retracts $\psi$ and sends everything else to $\gm_m$, that is, for all~$0\le \ell\le n$
\[
\qquad\pi(\gp_\ell)=\left\{
\begin{array}{cl}
  \gp_\ell & \textrm{if }\ell\le m
  \\
  \gm_{m} & \textrm{otherwise}
\end{array}\right.
\qquadtext{and}
\pi(\gm_\ell)=\left\{
\begin{array}{cl}
  \gm_\ell & \textrm{if }\ell\le m
  \\
  \gm_{m} & \textrm{otherwise}.
\end{array}\right.
\]
\end{enumerate}
\end{Lem}
\begin{proof}
Part~\ref{it:Spc-cyclic.fixed-pts} follows from \Cref{Prop:Spc-Psi}, while parts~\ref{it:Spc-cyclic.res} and~\ref{it:Spc-cyclic.infl} follow from \Cref{Rem:Spc-Res}.
\end{proof}

Let us now move to higher $p$-rank.
\begin{Exa}
\label{Exa:elem-ab}%
Let $E=(C_p)^{\times r}$ be the elementary abelian $p$-group of rank~$r$.
We know that $\Vee{E}=\Spc(\Db(\kk E))\cong\Spech(\rmH^\sbull(E,\kk))$ is homeomorphic to the space
\begin{equation}
\label{eq:Aff}%
\Aff{r}:=\Spech(\kk[x_1,\ldots,x_r]),
\end{equation}
that is, projective space~$\bbP^{r-1}_\kk$ with one closed point `on top'.
For instance, $\Aff{0}$ is a single point and $\Aff{1}$ is a 2-point Sierpi\'{n}ski space.
The example of~$r=1$ (see \Cref{Prop:cyclic} for $n=1$) is not predictive of what happens in higher rank.
Indeed, by \Cref{Prop:fibered}, the closed complement $\Supp(\Kac(E))$ is far from discrete in general.
It contains~$\frac{p^r-1}{p-1}$ copies of~$\Aff{r-1}$ and more generally $|\mathrm{Gr}_p(d,r)|$ copies of the $d$-dimensional~$\Aff{d}$ for $d=0,\ldots,r-1$, where $|\mathrm{Gr}_p(d,r)|$ is the number of rank-$d$ subgroups of~$(C_p)^{\times r}$.
Here is a `low-resolution' picture for Klein-four $r=p=2$:
\begin{equation}
\label{eq:C_2xC_2-artist}%
\vcenter{\xymatrix@C=3.2em@R=.8em{
{\color{Brown} \Aff{0}} \ar@{--}@[Gray][rd] \ar@{--}@[Gray][rrd] \ar@{--}@[Gray][rrrd] \ar@/_.5em/@{--}@[Gray][rrrrdd]
\\
& {\color{Brown} \Aff{1}} \ar@/_.5em/@{--}@[Gray][rrrd]
& {\color{Brown} \Aff{1}} \ar@{--}@[Gray][rrd]
& {\color{Brown} \Aff{1}} \ar@{--}@[Gray][rd]
\\
&&&& {\color{OliveGreen} \Aff{2}}
}}\kern2em
\end{equation}
The dashed lines indicate `partial' specialization relations: \emph{Some} points in the lower variety specialize to \emph{some} points in the higher one; see~ \Cref{Cor:P<P'}.
In rank~3, the similar `low-resolution' picture of~$\Spc(\cK(C_2^{\times 3}))$, still for $p=2$, looks as follows:
\begin{equation}
\label{eq:C_2xC_2xC_2-artist}%
\vcenter{\xymatrix@R=1em@C=.3em{
&&&&&& {\color{Brown}\Aff{0}}
\ar@{--}@[Gray][lllllld]\ar@{--}@[Gray][lllld]\ar@{--}@[Gray][lld]\ar@{--}@[Gray][d]\ar@{--}@[Gray][rrd]\ar@{--}@[Gray][rrrrd]\ar@{--}@[Gray][rrrrrrd]
\ar@{--}@[Gray][llllldd]\ar@{--}@[Gray][llldd]\ar@{--}@[Gray][ldd]\ar@{--}@[Gray][rdd]\ar@{--}@[Gray][rrrdd]\ar@{--}@[Gray][rrrrrdd]\ar@{--}@[Gray][rrrrrrrdd]
\ar@{--}@[Gray][rddd]
\\
{\color{Brown}\Aff{1}}
\ar@{--}@[Gray][rd]\ar@{--}@[Gray][rrrd]\ar@{--}@[Gray][rrrrrd]
&& {\color{Brown}\Aff{1}}
\ar@{--}@[Gray][ld]\ar@{--}@[Gray][rrrrrd]\ar@{--}@[Gray][rrrrrrrd]
&& {\color{Brown}\Aff{1}}
\ar@{--}@[Gray][llld]\ar@{--}@[Gray][rrrrrrrd]\ar@{--}@[Gray][rrrrrrrrrd]
&& {\color{Brown}\Aff{1}}
\ar@{--}@[Gray][llld]\ar@{--}@[Gray][rd]\ar@{--}@[Gray][rrrrrrrd]
&& {\color{Brown}\Aff{1}}
\ar@{--}@[Gray][llld]\ar@{--}@[Gray][ld]\ar@{--}@[Gray][rrrd]
&& {\color{Brown}\Aff{1}}
\ar@{--}@[Gray][llllllld]\ar@{--}@[Gray][ld]\ar@{--}@[Gray][rd]
&& {\color{Brown}\Aff{1}}
\ar@{--}@[Gray][llllllld]\ar@{--}@[Gray][llld]\ar@{--}@[Gray][rd]
\\
& {\color{Brown}\Aff{2}}
&& {\color{Brown}\Aff{2}}
&& {\color{Brown}\Aff{2}}
&& {\color{Brown}\Aff{2}}
&& {\color{Brown}\Aff{2}}
&& {\color{Brown}\Aff{2}}
&& {\color{Brown}\Aff{2}}
\\
&&&&&&& {\color{OliveGreen}\Aff{3}}
\ar@{--}@[Gray][llllllu]\ar@{--}@[Gray][llllu]\ar@{--}@[Gray][llu]\ar@{--}@[Gray][u]\ar@{--}@[Gray][rru]\ar@{--}@[Gray][rrrru]\ar@{--}@[Gray][rrrrrru]
\ar@{--}@[Gray][llllllluu]\ar@{--}@[Gray][llllluu]\ar@{--}@[Gray][llluu]\ar@{--}@[Gray][luu]\ar@{--}@[Gray][ruu]\ar@{--}@[Gray][rrruu]\ar@{--}@[Gray][rrrrruu]
}}
\end{equation}
Each $\Aff{d}$ has Krull dimension~$d\in\{0,1,2,3\}$ and contains one of 16 closed points.
\end{Exa}

Let us now discuss the example of Klein-four and `zoom-in' on~\eqref{eq:C_2xC_2-artist} to display every point at its actual height, as well as all specialization relations.

\begin{Exa}
\label{Exa:Klein4}%
Let $G=C_2\times C_2$ be the Klein four-group, in characteristic~$p=2$.
In \Cref{Exa:Klein-four}, we shall see that the spectrum $\Spc(\cK(E))$ is exactly as follows:
\begin{equation}\label{eq:C_2xC_2}%
\kern2em\vcenter{\xymatrix@C=.0em@R=.4em{
{\color{Brown}\overset{\cM(E)}{\bullet}} \ar@{-}@[Gray][rrdd] \ar@{-}@[Gray][rrrrdd] \ar@{-}@[Gray][rrrrrrdd] \ar@{~}@<.1em>@[Gray][rrrrrrrrdd] &&& {\color{Brown}\overset{\cM(N_0)}{\bullet}} \ar@{-}@[Brown][ldd] \ar@{-}@<.1em>@[Gray][rrrrrrdd]
&& {\color{Brown}\overset{\cM(N_1)}{\bullet}} \ar@{-}@[Brown][ldd] \ar@{-}@[Gray][rrrrrrdd]
&& {\color{Brown}\overset{\cM(N_\infty)}{\bullet}} \ar@{-}@[Brown][ldd] \ar@{-}@[Gray][rrrrrrdd]
&& {\color{OliveGreen}\overset{\cM(1)}{\bullet}} \ar@{~}@[OliveGreen][ldd] \ar@{-}@[OliveGreen][dd] \ar@{-}@[OliveGreen][rrdd] \ar@{-}@[OliveGreen][rrrrdd]
\\ \\
&& {\color{Brown}\underset{\cP(N_0)}{\bullet}} \ar@{-}@<-.4em>@[Gray][rrrrrrrdd]
&& {\color{Brown}\underset{\cP(N_1)}{\bullet}} \ar@{-}@<-.1em>@[Gray][rrrrrdd]
&& {\color{Brown}\underset{\cP(N_\infty)}{\bullet}} \ar@{-}@[Gray][rrrdd]
&\ar@{.}@[OliveGreen][r]
& {\scriptstyle\color{OliveGreen}\tinyPone} \ar@{~}@[OliveGreen][rdd] \ar@{.}@[OliveGreen][rrrrrrr]
& {\color{OliveGreen}\bullet_{0}} \ar@{-}@[OliveGreen][dd]
&& {\color{OliveGreen}\bullet_{1}} \ar@{-}@[OliveGreen][lldd]
&& {\color{OliveGreen}\bullet_{\infty}} \ar@{-}@[OliveGreen][lllldd]
&&
\\ \\
&&&&&&&&& {\color{OliveGreen}\bullet_{\cP_0}}\kern-1em
&&&&
}}\kern-1.11em
\end{equation}
In this picture, $N_0,\,N_1$ and~$N_\infty$ are the three cyclic subgroups of~$G$.
The colors match those of~\eqref{eq:C_2xC_2-artist}.
The green part is the cohomological open~$\Vee{E}\simeq \Aff{2}$ as in \eqref{eq:Aff}, that is, a~$\bbP^1$ with a closed point on top; we marked with~${\color{OliveGreen}\bullet}$ the closed point~$\cM(1)$, the three $\bbF_{\!2}$-rational points $0$, $1$, $\infty$ of~$\bbP^1$ and its generic point~$\cP_0$; the notation $\Pone$ and the dotted line indicate $\bbP^1\sminus\{0,1,\infty,\cP_0\}$.
The brown part is the support of the acyclics, namely the union of the $\Vee{E}(H)$ for non-trivial subgroups~$H\le E$ as in \Cref{Prop:fibered};
it consists of three Sierpi\'{n}ski subspaces $\{\cP(N_i)\rightsquigarrow\cM(N_i)\}\simeq\Vee{E/N_i}\simeq \Aff{1}$ and the singleton~$\{\cM(E)\}\simeq\Vee{E/E}\simeq \Aff{0}$.

The specializations involving points of~$\Pone$ are displayed with undulated lines, indicating that all points share the same behavior.
For instance, the gray undulated line indicates that \emph{all} points of~$\Pone$ specialize to~$\cM(E)$. The proof of this critical fact will require the new tools of \Cref{part:twisted-cohomology}.
\end{Exa}

\begin{Exa}
\label{Exa:Q_8}%
The spectrum of the quaternion group~$Q_8$ is very similar to that of its quotient~$E:=Q_8/Z(Q_8)\cong C_2\times C_2$, as we announced in~\eqref{eq:Q_8}.
The center~$Z:=Z(Q_8)\cong C_2$ is the maximal elementary abelian $2$-subgroup and it follows that $\Res^{Q_8}_{Z}$ induces a homeomorphism $V_{C_2}\isoto V_{Q_8}$. In other words, $V_{Q_8}$ is again a Sierpi\'{n}ski space~$\{\cP,\cM(1)\}$.
On the other hand, the center $Z$ is also the unique minimal non-trivial subgroup.
It follows from~\Cref{Cor:Spc-surjection-derived} and \Cref{Prop:Im(psi^N)} that $\Supp(\Kac(Q_8))$ is the image under the closed immersion~$\psi^Z$ of~$\Spc(\cK(Q_8/Z))$.
It only remains to describe the specialization relations between the cohomological open~$\Vee{Q_8}$ and its closed complement~$\Supp(\Kac(Q_8))$.
Since $\cM(1)\in\Vee{Q_8}$ is also a closed point in~$\Spc(\cK(Q_8))$, we only need to decide where the generic point~$\cP$ of~$V_{Q_8}$ specializes in~$\Spc(\cK(Q_8))$.
Interestingly, $\cP$ will not be generic in the whole of~$\Spc(\cK(Q_8))$.
As $\cP$ belongs to $\Img(\rho_{Z})$, it suffices to determine~$\rho_Z(\cM_{C_2}(C_2))$. The preimage of~$\Img(\rho_{Z})=\supp(\kk(Q_8/Z))$ under~$\psi^Z$ is $\supp_{E}(\Psi^Z(\kk(Q_8/Z)))=\supp_E(\kk(E))=\{\cM_E(1)\}$. It follows that $\cP$ specializes to exactly one point:~$\psi^Z(\cM_E(1))=\cM_{Q_8}(Z)$ as depicted in~\eqref{eq:Q_8}.
\end{Exa}

\section{Stratification}
\label{sec:stratification}%

It is by now well-understood how to deduce stratification in the presence of a noetherian spectrum and a conservative theory of supports. We follow the general method of Barthel-Heard-Sanders~\cite{barthel-heard-sanders:stratification-supph,barthel-heard-sanders:stratification-Mackey}.

\begin{Prop}
\label{Prop:Spc-noetherian}%
The spectrum $\SpcKG$ is a noetherian topological space.
\end{Prop}
\begin{proof}
Recall that a space is noetherian if every open is quasi-compact. It follows that the continuous image of a noetherian space is noetherian.
The claim now follows from \Cref{Cor:Spc-surjection-derived}.
\end{proof}

We start with the key technical fact. Recall that coproduct-preserving exact functors between compactly-generated triangulated categories have right adjoints by Brown-Neeman Representability. We apply this to~$\Psi^H$.
\begin{Lem}
\label{Lem:Psi^N_rho}%
Let $N\normaleq G$ be a normal $p$-subgroup and $\Psi^N_\rho\colon \DPerm(G/N;\kk)\to \DPerm(G;\kk)$ the right adjoint of modular $N$-fixed points~$\Psi^N\colon \DPerm(G;\kk)\to \DPerm(G/N;\kk)$. Then $\Psi^N_\rho(\unit)$ is isomorphic to a complex~$s$ in $\perm(G;\kk)$, concentrated in non-negative degrees
\[
s=\quad\big(\cdots \to s_n \to \cdots \to s_2 \to s_1 \to s_0\to 0\to 0\cdots\big)
\]
with $s_0=\kk$ and $s_1=\oplus_{H\in\cF_N}\kk(G/H)$, where $\cF_N=\SET{H\le G}{N\not\le H}$.
\end{Lem}
\begin{proof}
Following the recipe of Brown-Neeman Representability~\cite{neeman:brown}, we give an explicit description of~$\Psi^N_\rho(\unit)$ as the homotopy colimit in~$\cT(G)$ of a sequence of objects $x_0=\unit\xto{f_0} x_1\xto{f_1} \cdots \to x_n\xto{f_n} x_{n+1}\to \cdots$ in~$\cK(G)$. This sequence is built together with maps~$g_n\colon \Psi^N(x_n)\to \unit$ in~$\cK(G/N)$ making the following commute
\begin{equation}
\label{eq:aux-strat-1}%
\vcenter{\xymatrix@C=2em{
\Psi^N(x_0)=\unit \ar@{=}@/_1em/[rrrd]_(.3){g_0=\id\ } \ar[rr]^-{\Psi^N(f_0)}_-{}
&& {\cdots} \ar[r]|-{} \ar@{}[rd]|-{\cdots}  & \Psi^N(x_n) \ar[rr]^-{\Psi^N(f_n)} \ar[d]_-{g_n} && \Psi^N(x_{n+1}) \ar[lld]^(.4){g_{n+1}} \ar@{}[r]|-{\cdots}&\ar@{}[lld]|-{\cdots}
\\
&&& \quad \unit \quad &
}}
\end{equation}
Note that such $g_n$ yield homomorphisms, natural in $t\in\DPerm(G;\kk)$, as follows
\begin{equation}
\label{eq:aux-strat-2}%
\alpha_{n,t}\colon\Hom_{G}(t,x_n)\xto{\Psi^N}
\Hom_{G/N}(\Psi^N(t),\Psi^N(x_n))\xto{(g_n)_*}
\Hom_{G/N}(\Psi^N(t),\unit)\kern-1em
\end{equation}
where we abbreviate $\Hom_G$ for $\Hom_{\DPerm(G;\kk)}$.
We are going to build our sequence of objects $x_0\to x_1\to \cdots$ and the maps~$g_n$ so that for each~$n\ge 0$
\begin{equation}\label{eq:aux-strat-3}%
\textrm{$\alpha_{n,t}$ is an isomorphism for every $t\in \SET{\Sigma^i\,\kk(G/H)}{i< n,\ H\le G}$.}
\end{equation}
It follows that, if we set~$x_\infty=\hocolim_n x_n$ and $g_\infty\colon \Psi^N(x_\infty)\cong\hocolim_n\Psi^N(x_n)\to \unit$ a colimit of the~$g_n$, then the map
\[
\alpha_{t}\colon\Hom_{G}(t,x_\infty)\xto{\Psi^N}
\Hom_{G/N}(\Psi^N(t),\Psi^N(x_\infty))\xto{(g_\infty)_*}
\Hom_{G/N}(\Psi^N(t),\unit)
\]
is an isomorphism for all~$t\in\SET{\Sigma^i\,\kk(G/H)}{i\in \bbZ,\ H\le G}$.
Since the $\kk(G/H)$ generate~$\DPerm(G;\kk)$, it follows that $\alpha_{t}$ is an isomorphism for all~$t\in\DPerm(G;\kk)$.
Hence~$x_\infty=\hocolim_n x_n$ is indeed the image of~$\unit$ by the right adjoint~$\Psi^N_\rho$.

Let us construct these sequences $x_n,\ f_n$ and~$g_n$, for $n\ge 0$.
In fact, every complex~$x_n$ will be concentrated in degrees between zero and~$n$, so that \eqref{eq:aux-strat-3} is trivially true for $n=0$ (that is, for $i<0$), both source and target of~$\alpha_{n,t}$ being zero in that case.
Furthermore, $x_{n+1}$ will only differ from $x_n$ in degree~$n+1$, with $f_n$ being the identity in degrees~$\le n$. So the verification of~\eqref{eq:aux-strat-3} for~$n+1$ will boil down to checking the cases of~$t=\Sigma^i\,\kk(G/H)$ for~$i=n$.

As indicated, we set $x_0=\unit$ and $g_0=\id$. We define $x_1$ by the exact triangle
\[
s_1\xto{\eps} \unit \xto{f_0} x_1 \to \Sigma(s_1)
\]
where $s_1:=\oplus_{H\in\cF_N}\kk(G/H)$ and $\eps_H\colon \kk(G/H)\to \kk$ is the usual map. Note that $\Psi^N(s_1)=0$ by \eqref{eq:fixed-pts}, hence $\Psi^N(f_0)\colon \unit \to \Psi^N(x_1)$ is an isomorphism. We call~$g_1$ its inverse. One verifies that~\eqref{eq:aux-strat-3} holds for~$n=1$: For $t=\kk(G/H)$ with $H\in \cF_N$, both the source and target of~$\alpha_{1,t}$ are zero thanks to the definition of~$s_1$. For the case where $H\ge N$, there are no non-zero homotopies for maps $\kk(G/H)\to x_1$ thanks to \Cref{Lem:infl-stab}.

Let us construct $x_{n+1}$ and $g_{n+1}$ for $n\ge 1$. For every $H\le G$ let $t=\Sigma^n(\kk(G/H))$ and choose generators $h_{H,1},\ldots,h_{H,r_H}\colon t\to x_n$ of the $\kk$-module~$\Hom_{G}(t,x_n)$, source of~$\alpha_{n,t}$.
Define $s_{n+1}=\oplus_{H\le G}\oplus_{i=1}^{r_H}\kk(G/H)$ in~$\perm(G;\kk)$, a sum of $r_H$ copies of~$\kk(G/H)$ for every~$H\le G$, and define $h_n\colon \Sigma^n(s_{n+1})\to x_n$ as being $h_{H,i}$ on the $i$-th summand~$\Sigma^n\,\kk(G/H)$. Define $x_{n+1}$ as the cone of~$h_n$ in~$\cK(G)$:
\begin{equation}
\label{eq:aux-strat-4}%
\Sigma^n(s_{n+1})\xto{h_n} x_n\xto{f_n} x_{n+1} \to \Sigma^{n+1}(s_{n+1}).
\end{equation}
Note that $x_{n+1}$ only differs from~$x_n$ in homological degree~$n+1$ as announced. Since $n\ge 1$, we get $\Hom_{G/N}(\Psi^N(x_{n+1}),\unit)\cong \Hom_{G/N}(\Psi^N(x_n),\unit)$ and there exists a unique $g_{n+1}\colon \Psi^{N}(x_{n+1})\to \unit$ making~\eqref{eq:aux-strat-1} commute.
It remains to verify that $\alpha_{n+1,t}$ is an isomorphism for $t\in \SET{\Sigma^n\,\kk(G/H)}{H\le G}$. Note that the target of this map is zero.
Applying $\Hom_G(\Sigma^{n}\,\kk(G/H),-)$ to the exact triangle~\eqref{eq:aux-strat-4} shows that the source of~$\alpha_{n+1,t}$ is also zero, by construction. Hence~\eqref{eq:aux-strat-3} holds for~$n+1$.

This realizes the wanted sequence and therefore $\Psi^N_\rho(\unit)\simeq\hocolim_n(x_n)$ has the following form:
\[
\cdots \to s_{n} \to \cdots \to s_2 \to s_1 \to \kk \to 0\to 0\cdots
\]
where $s_1=\oplus_{H\in\cF_N}\kk(G/H)$ and $s_{n}\in\perm(G;\kk)$ for all~$n$.
\end{proof}

\begin{Rem}
The above description of~$\Psi^N_\rho(\unit)$ gives a formula for the right adjoint $\Psi^N_\rho\colon \DPerm(G/N;\kk)\to \DPerm(G;\kk)$ on all objects.
Indeed, for every $t\in\DPerm(G/N;\kk)$, we have a canonical isomorphism in~$\DPerm(G;\kk)$
\[
\Psi^N_\rho(t)\cong \Psi^N_\rho(\Psi^N\Infl^{G/N}_G(t)\otimes\unit)\cong \Infl^{G/N}_G(t)\otimes \Psi^N_\rho(\unit)
\]
using that $\Psi^N\circ\Infl^{G/N}_G\cong\Id$ and the projection formula. In other words, the right adjoint~$\Psi^N_\rho$ is simply inflation tensored with the commutative ring object~$\Psi^N_\rho(\unit)$.
\end{Rem}

\begin{Lem}
\label{Lem:stratification}%
Let $H\normaleq G$ be a normal $p$-subgroup and $\Psi^H_\rho\colon \DPerm(G/H;\kk)\to \DPerm(G;\kk)$ the right adjoint of modular $H$-fixed points~$\Psi^H\colon \DPerm(G;\kk)\to \DPerm(G/H;\kk)$. Then the object~$\zul[G]{H}$ displayed in~\eqref{eq:y-normal} belongs to the localizing tt-ideal of~$\DPerm(G;\kk)$ generated by~$\Psi^H_\rho(\unit)$.
\end{Lem}
\begin{proof}
By \Cref{Prop:Im(psi^N)}, we know that the tt-ideal generated by~$\zul[G]{H}$ is exactly $\cap_{K\in\cF_H}\Ker\Res^G_K$.
By Frobenius, the latter is the tt-ideal $\SET{x\in\cK(G)}{s_1\otimes x=0}$ where $s_1=\oplus_{K\in\cF_H}\kk(G/K)$ is the degree one part of the complex~$s\simeq\Psi^H_\rho(\unit)$ of \Cref{Lem:Psi^N_rho}.
We can now conclude by \Cref{Lem:s-generates} applied to this complex~$s$ and $x=\zul[G]{H}$ that $x$ must belong to the localizing tensor-ideal of~$\DPerm(G;\kk)$ generated by~$\Psi^H_\rho(\unit)$.
(Note that $s_0=\unit$ here.)
\end{proof}

Recall from \Cref{Cor:Im(psi^H)} that the map $\psi^H$ has closed image in~$\SpcKG$.
\begin{Prop}
\label{Prop:stratification}%
Let $H\le G$ be a $p$-subgroup and let ${\Psi}^H_\rho\colon \DPerm(\WGH;\kk)\to \DPerm(G;\kk)$ be the right adjoint of~${\Psi}^H\colon \DPerm(G;\kk)\to \DPerm(\WGH;\kk)$. Then the tt-ideal of~$\cK(G)$ supported on the closed subset~$\Img(\psi^H)$ is contained in the localizing tt-ideal of~$\DPerm(G;\kk)$ generated by~${\Psi}^H_\rho(\unit)$.
\end{Prop}
\begin{proof}
Let $N=N_G H$.
By definition, ${\Psi}^{H\inn G}=\Psi^{H\inn N}\circ \Res^G_{N}$ and therefore the right adjoint is ${\Psi}^{H\inn G}_\rho\cong \Ind_{N}^G\circ\Psi^{H\inn N}_\rho$. By \Cref{Lem:stratification}, we can handle $H\normaleq N$ hence we know (see also~\Cref{Prop:Im(psi^N)}) that the generator $\zul[N]{H}$ of the tt-ideal supported on~$\Img(\psi^{H\inn N})$ belongs to $\Loctens{\Psi^{H\inn N}_\rho(\unit)}$ in~$\DPerm(N;\kk)$.
Since $\Res^G_N$ is surjective up to direct summands (by the Mackey formula), this subcategory $\Loctens{\Psi^{H\inn N}_\rho(\unit)}$ is also $\Loc{\SET{\Res^G_N(t)\otimes\Psi^{H\inn N}_\rho(\unit)}{\textrm{all }t\in\DPerm(G;k)}}$. Using that $\Ind_N^G$ is cocontinuous and the projection formula for~$\Res\adj\Ind$ we conclude that $\zul[G]{H}\overset{\textrm{def}}{\ =\ }\Ind_{N}^G(\zul[N]{H})$ belongs to~$\Ind_{N}^G(\Loc{\SET{\Res^G_N(t)\otimes\Psi^{H\inn N}_\rho(\unit)}{\textrm{all }t}})\subseteq\Loc{\SET{t\otimes\Ind_{N}^G\Psi^{H\inn N}_{\rho}(\unit)}{\textrm{all }t}}=\Loctens{\Psi^{H\inn G}_{\rho}(\unit)}$ in~$\DPerm(G;\kk)$.
\end{proof}

Let us now turn to stratification. By noetherianity, we can define a support for possibly non-compact objects in the `big' tt-category under consideration, here $\DPerm(G;\kk)$, following~\cite[\S\,7]{balmer-favi:idempotents} and~\cite[\S\,2]{barthel-heard-sanders:stratification-Mackey}. We remind the reader.

\begin{Rec}
\label{Rec:idempotents}%
Every Thomason subset $Y\subseteq \SpcKG$ yields a so-called `idempotent triangle' $e(Y)\to\unit\to f(Y)\to \Sigma e(Y)$ in $\cT(G)=\DPerm(G;\kk)$,
meaning that $e(Y)\otimes f(Y)=0$, hence $e(Y)\cong e(Y)\potimes{2}$ and $f(Y)\cong f(Y)\potimes{2}$.
The left idempotent $e(Y)$ is the generator of $\Loctens{\cK(G)_{Y}}$, the localizing tt-ideal of~$\cT(G)$ `supported' on~$Y$. The right idempotent~$f(Y)$ realizes localization of~$\cT(G)$ `away' from~$Y$, that is, the localization on the complement~$Y^c$.

By noetherianity, for every point $\cP\in \SpcKG$, the closed subset~$\adhpt{\cP}$ is Thomason.
Hence $\adhpt{\cP}\cap (Y_{\cP})^c=\{\cP\}$, where $Y_\cP:=\supp(\cP)=\SET{\cQ}{\cP\not\subseteq\cQ}$ is always a Thomason subset.
The idempotent~$g(\cP)$ in~$\cT(G)$ is then defined as
\[
g(\cP)=e(\adhpt{\cP})\otimes f(Y_{\cP}).
\]
It is built to capture the part of~$\DPerm(G;\kk)$ that lives both `over~$\adhpt{\cP}$' (thanks to $e(\adhpt{\cP})$) and `over~$Y_{\cP}^c$' (thanks to~$f(Y_{\cP})$);
in other words, $g(\cP)$ lives exactly `at~$\cP$'.
This idea originates in~\cite{hovey-palmieri-strickland}.
It explains why the support is defined as
\[
\Supp(t)=\SET{\cP\in\SpcKG}{g(\cP)\otimes t\neq 0}
\]
for every (possibly non-compact) object~$t\in\DPerm(G;\kk)$.
\end{Rec}

\begin{Thm}
\label{Thm:stratification}%
Let $G$ be a finite group and let $\kk$ be a field. Then the big tt-category $\cT(G)=\DPerm(G;\kk)$ is stratified, that is, we have an order-preserving bijection
\[
\left\{
\textrm{Localizing tt-ideals $\cL\subseteq \cT(G)$}
\right\}
\overset{\sim}{\longleftrightarrow}
\left\{
\textrm{Subsets of $\SpcKG$}
\right\}
\]
given by sending a subcategory $\cL$ to the union of the supports of its objects; its inverse sends a subset $Y\subseteq \SpcKG$ to $\cL_{Y}:=\SET{t\in \cT(G)}{\Supp(t)\subseteq Y}$.
\end{Thm}

\begin{proof}
By induction on the order of the group, we can assume that the result holds for every proper subquotient~$\WGH$ (with $H\neq 1$).
By~\cite[Theorem~3.22]{barthel-heard-sanders:stratification-Mackey}, noetherianity of the spectrum of compacts reduces stratification to proving \emph{minimality} of $\Loctens{g(\cP)}$ for every $\cP\in\SpcKG$.
This means that $\Loctens{g(\cP)}$ admits no non-trivial localizing tt-ideal subcategory.
If $\cP$ belongs to the cohomological open~$\VG=\Spc(\Db(\kkG))$ then minimality at~$\cP$ in~$\cT=\DPerm(G;\kk)$ is equivalent to minimality at~$\cP$ in $\cT(\VG)\cong\K\Inj(\kk G)$ by~\cite[Proposition~5.2]{barthel-heard-sanders:stratification-Mackey}. Since $\K\Inj(\kk G)$ is stratified by~\cite{BIK:stratifying-stmod-kG}, we have the result in that case.

Let now $\cP\in\Supp(\Kac(G))$. By \Cref{Cor:supp(Kac)}, we know that $\cP=\cP_G(H,\gp)$ for some non-trivial $p$-subgroup $1\neq H\le G$ and some cohomological point~$\gp\in\Vee{\WGH}$. (In the notation of \Cref{Prop:fibered}, this means $\cP\in \Vee{G}(H)$.)
Suppose that $t\in\Loctens{g(\cP)}$ is non-zero. We need to show that $\Loctens{t}=\Loctens{g(\cP)}$, that is, we need to show that $g(\cP)\in \Loctens{t}$.

Recall the tt-functor~$\check{\Psi}^H\colon \DPerm(G;\kk)\to \K\Inj(\kk \WGH))$ from \Cref{Not:bar-Psi}.
By general properties of BF-idempotents~\cite[Theorem~6.3]{balmer-favi:idempotents}, we have $\check{\Psi}^K(g(\cP))=g((\check{\psi}^K)\inv(\cP))$ in~$\K\Inj(\kk(\WGK))$ for every~$K\in\Sub{p}\!G$.
Since $\check{\psi}^K$ is injective by \Cref{Prop:bar-psi-injective}, the fiber $(\check{\psi}^K)\inv(\cP)$ is a singleton (namely~$\gp$) if~$K\sim H$ and is empty otherwise.
It follows that for all~$K\not\sim H$ we have $\check{\Psi}^K(g(\cP))=0$ and therefore $\check{\Psi}^K(t)=0$ as well. Since~$t$ is non-zero, the Conservativity \Cref{Thm:conservativity} forces the only remaining~$\check{\Psi}^H(t)$ to be non-zero in~$\K\Inj(\kk(\WGH))$.
This forces $\Psi^H(t)$ to be non-zero in~$\cT(\WGH)$ as well, since~$\check{\Psi}^H=\Upsilon_{\WGH}\circ\Psi^H$.
This object~$\Psi^H(t)$ belongs to~$\Loctens{\Psi^H(g(\cP))}=\Loctens{g((\psi^H)\inv(\cP))}$.
Note that $\upsilon_{\WGH}(\gp)$ is the only preimage of~$\cP=\cP_G(H,\gp)$ under~$\psi^H$ (see \Cref{Rem:psi^H-V(H)}).
By induction hypothesis, this localizing tt-ideal $\Loctens{\Psi^H(g(\cP))}$ is minimal.
And it contains our non-zero object~$\Psi^H(t)$.
Hence $\Psi^H(g(\cP))\in\Loctens{\Psi^H(t)}$. Applying the right adjoint~$\Psi^H_\rho$, it follows that $\Psi^H_\rho\Psi^H(g(\cP))\in \Psi^H_\rho(\Loctens{\Psi^H(t)})\subseteq\Loctens{t}$ where the last inclusion follows by the projection formula for $\Psi^H\adj \Psi^H_\rho$. Hence by the projection formula again we have in~$\cT(G)$ that
\begin{equation*}
\label{eq:aux-strat-5}%
\Psi^H_\rho(\unit)\otimes g(\cP)\in\Loctens{t}.
\end{equation*}
But we proved in \Cref{Prop:stratification} that the localizing tt-ideal generated by $\Psi^H_\rho(\unit)$ contains~$\cK(G)_{\Img(\psi^H)}$ and in particular $e(\adhpt{\cP})$ and a fortiori~$g(\cP)$. In short, we have $g(\cP)\cong g(\cP)\potimes{2}\in\Loctens{\Psi^H_\rho(\unit)\otimes g(\cP)}\subseteq \Loctens{t}$ as needed to be proved.
\end{proof}

\begin{Cor}
\label{Cor:telescope}%
The Telescope Conjecture holds for~$\DPerm(G;\kk)$. Every smashing tt-ideal~$\cS\subseteq \DPerm(G;\kk)$ is generated by its compact part:~$\cS=\Loctens{\cS^c}$.
\end{Cor}
\begin{proof}
This follows from noetherianity of~$\SpcKG$ and stratification by~\cite[Theorem~9.11]{barthel-heard-sanders:stratification-Mackey}.
\end{proof}

\bigbreak
\part{Topology of the spectrum and twisted cohomology}
\label{part:twisted-cohomology}
\bigbreak

\section{Introduction to Part~II}
\label{sec:intro-II}

After identifying all the points in the spectrum $\Spc(\cK(G))$ of the permutation tt-category~\eqref{eq:K(G)} in \Cref{part:stratification}, we now want to describe the topology.
This knowledge will give us the classification of thick $\otimes$-ideals in~$\cK(G)$.

\subsection*{The colimit theorem}

To discuss the tt-geometry of~$\cK(G)$, it is instructive to keep in mind the bounded derived category of finitely generated $\kkG$-modules, $\Db(\kkG)$, which is a localization of our~$\cK(G)$ by~\cite[Theorem~5.13]{balmer-gallauer:resol-small}.
A theorem of Serre~\cite{serre:bockstein}, famously expanded by Quillen~\cite{quillen:spec-cohomology}, implies that $\Spc(\Db(\kkG))$ is the colimit of the $\Spc(\Db(\kk E))$, for $E$ running through the elementary abelian $p$-subgroups of~$G$; see~\cite[\S\,4]{balmer:tt-separable}. The indexing category for this colimit is an orbit category: Its morphisms keep track of conjugations and inclusions of subgroups.

In Part~I, we proved that $\SpcKG$ is set-theoretically partitioned into spectra of derived categories~$\Db(\kk (\WGK))$ for certain subquotients of~$G$, namely the Weyl groups~$\WGK=(N_G K)/K$ of $p$-subgroups~$K\le G$.
It is then natural to expect a more intricate analogue of Quillen's result for the tt-category~$\cK(G)$, in which subgroups are replaced by \emph{subquotients}.
This is precisely what we prove.
The orbit category has to be replaced by a category~$\EAppG$ whose objects are elementary abelian $p$-sections~$E=H/K$, for $p$-subgroups~$K\normaleq H\le G$.
The morphisms in~$\EAppG$ keep track of conjugations, inclusions \emph{and quotients}.
See \Cref{Cons:EA}.

This allows us to formulate our reduction to elementary abelian groups:
\begin{Thm}[\Cref{Thm:colim}]
\label{Thm:colim-intro}%
There is a canonical homeomorphism
\[
\colim_{E\in\EAppG}\ \Spc(\cK(E))\isoto\SpcKG.
\]
\end{Thm}
The category~$\EAppG$ has been considered before, \eg\ in Bouc-Th\'evenaz~\cite{bouc-thevenaz:gluing-endo-perm}.
Every morphism in~$\EAppG$ is the composite of three special morphisms (\Cref{Rem:EA-generating-morphisms})
\begin{equation}
\label{eq:3-morphisms}%
E\xto{\simeq} E'\to E''\xto{!} E'''
\end{equation}
where $E'$ is a $G$-conjugate of~$E$, where $E'\le E''$ is a subgroup of~$E''$ and where $E''=E'''/N$ is a quotient of~$E'''$ (sic!).
The tt-category $\cK(E)$ is contravariant in~$E\in\EAppG$ and the tt-functors corresponding to~\eqref{eq:3-morphisms}
\begin{equation}
\label{eq:3-functors}%
\xymatrix@C=2em{\cK(E''') \ar[r]^-{\Psi^N}
  & \cK(E'') \ar[r]^-{\Res}
& \cK(E') \ar[r]^-{\simeq}
& \cK(E)}
\end{equation}
yield the \emph{modular $N$-fixed-points functor}~$\Psi^N$ introduced in Part~I, and the standard restriction functor and conjugation isomorphism.
As we saw, the $\Psi^N$ are a type of Brauer quotient that make sense on the homotopy category of permutation modules but do not exist on derived or stable categories.
They distinguish our results and their proofs from the classical theory.

\subsection*{Twisted cohomology}

The above discussion reduces the analysis of~$\SpcKG$ to the case of elementary abelian $p$-groups~$E$.
As often in modular representation theory, this case is far from trivial and can be viewed as the heart of the matter.

So let $E$ be an elementary abelian $p$-group.
Our methods will rely on $\otimes$-invertible objects~$u_N$ in~$\cK(E)$ indexed by the set~$\cN(E)=\SET{N\normal E}{[E\!:\!N]=p}$ of maximal subgroups.
These objects are of the form $u_N=\big(0\to \kk(E/N)\to \kk(E/N)\to\kk \to 0\big)$ for $p$ odd and $u_N=\big(0\to \kk(E/N)\to\kk \to 0\big)$ for $p=2$. See \Cref{Def:u_N}.
We use these $\otimes$-invertibles $u_N$ to construct a multi-graded ring
\begin{equation}
\label{eq:Rall-intro}%
\Rall(E) = \bigoplus_{s\in\bbZ} \ \bigoplus_{q\in\bbN^{\cN(E)}} \Hom_{\cK(E)}\big(\unit,\unit(q)[s]\big),
\end{equation}
where $\unit(q)$ is the $\otimes$-invertible~$\bigotimes_{N\in\cN(E)}\,u_N\potimes{q(N)}$ for every tuple~\mbox{$q\colon \cN(E)\to \bbN$}, that we refer to as a `twist'.
Without these twists we would obtain the standard $\bbZ$-graded endomorphism ring $\End^\sbull(\unit):=\oplus_{s\in\bbZ}\Hom(\unit,\unit[s])$ of~$\unit$ which, for~$\Db(\kk E)$, is the cohomology~$\Hm^\sbull(E,k)$, but for~$\cK(E)$ is reduced to the field~$\kk$ and therefore rather uninteresting.
We call~$\Rall(E)$ the \emph{(permutation) twisted cohomology} of~$E$.
Some readers may appreciate the analogy with cohomology twisted by line bundles in algebraic geometry, or with Tate twists in motivic cohomology.

We can employ this multi-graded ring~$\Rall(E)$ to describe~$\SpcKE$:

\begin{Thm}[\Cref{Cor:abelem-top}]
\label{Thm:abelem-top-intro}%
The space $\Spc(\cK(E))$ identifies with an \emph{open subspace} of the homogeneous spectrum of~$\Rall(E)$ via a canonical `comparison map'.
\end{Thm}

The comparison map in question generalizes the one of~\cite{balmer:sss}, which landed in the homogenous spectrum of~$\End^\sbull(\unit)$ without twist.
We also describe in \Cref{Cor:abelem-top} the open image of this map by explicit equations in~$\Rall(E)$.

\subsection*{Dirac geometry}

If the reader is puzzled by the multi-graded ring~$\Rall(E)$, here is another approach based on a special open cover~$\{\goodopen{H}\}_{H\le E}$ of~$\SpcKE$ indexed by the subgroups of~$E$ and introduced in~\Cref{Prop:U(H)}. Its key property is that over each open~$\goodopen{H}$ all the $\otimes$-invertible objects~$u_N$ are trivial: $(u_N)\restr{\goodopen{H}}\simeq \unit[s]$ for some shift~$s\in\bbZ$ depending on $H$ and~$N$. For the trivial subgroup~\mbox{$H=1$}, the open~$\goodopen{1}$ is the `cohomological open' of~Part~I, that corresponds to the image under~$\Spc(-)$ of the localization $\cK(E)\onto \Db(\kk E)$. See \Cref{Prop:U(1)}.
At the other end, for $H=E$, we show in \Cref{Prop:U(G)} that the open~$\goodopen{E}$ is the `geometric open' that corresponds to the localization of~$\cK(E)$ given by the geometric fixed-points functor. Compare \Cref{Rem:geom-fixed-pts}. For $E$ of rank one, these two opens~$\goodopen{1}$ and~$\goodopen{E}$ are all there is to consider. But as the $p$-rank of~$E$ grows, there is an exponentially larger collection~$\{\goodopen{H}\}_{H\le E}$ of open subsets interpolating between $\goodopen{1}$ and~$\goodopen{E}$.
This cover~$\{\goodopen{H}\}_{H\le E}$ allows us to use the classical comparison map of~\cite{balmer:sss} locally. It yields a homeomorphism between each~$\goodopen{H}$ and the homogeneous spectrum of the $\bbZ$-graded endomorphism ring~$\End^\sbull_{\goodopen{H}}(\unit)$ in the localization~$\cK(E)\restr{\goodopen{H}}$. In compact form, this can be rephrased as follows (a Dirac scheme is to a usual scheme what a $\bbZ$-graded ring is to a non-graded one):
\begin{Thm}[\Cref{Cor:abelem-Dirac}]
\label{Thm:abelem-Dirac-intro}%
The space $\Spc(\cK(E))$, together with the sheaf of $\bbZ$-graded rings obtained locally from endomorphisms of the unit, is a Dirac scheme.
\end{Thm}

\subsection*{Elementary abelian take-home}

Let us ponder the $\bbZ$-graded endomorphism ring of the unit~$\End^\sbull(\unit)$ for a moment longer.
As we know, the ring $\End^\sbull_{\cK(E)}(\unit)=\kk$ is too small to provide geometric information.
So we have developed two substitutes.
Our first approach is to replace the usual $\bbZ$-graded ring $\End^\sbull(\unit)$ by a richer multi-graded ring involving twists.
This leads us to twisted cohomology~$\Rall(E)$ and to \Cref{Thm:abelem-top-intro}.
The second approach is to hope that the endomorphism ring $\End^\sbull(\unit)$, although useless globally, becomes rich enough to control the topology \emph{locally} on~$\SpcKE$, without leaving the world of $\bbZ$-graded rings. This is what we achieve in \Cref{Thm:abelem-Dirac-intro} thanks to the open cover~$\{\goodopen{H}\}_{H\le E}$.
As can be expected, the two proofs are intertwined.

\subsection*{Touching ground}

Combining \Cref{Thm:colim-intro,Thm:abelem-Dirac-intro} ultimately describes the topological space~$\SpcKG$ for all~$G$, in terms of homogeneous spectra of graded rings.
In \Cref{sec:closure,sec:presentation-H**,sec:applications} we improve and apply these results as follows.

In \Cref{sec:closure}, we explain how to go from the `local' rings $\End^\sbull_{\goodopen{H}}(\unit)$ over the open~$\goodopen{H}$, for each subgroup~$H\le E$, to the `global' topology of~$\SpcKE$.

In \Cref{Thm:presentation}, we give a finite presentation by generators and relations of the reduced $\kk$-algebra $(\End^\sbull_{\goodopen{H}}(\unit))_{\red}$ generalizing the usual one for cohomology.

In \Cref{Cor:SpcKG-coeq}, we express~$\SpcKG$ for a general finite group~$G$ as the quotient of a disjoint union of~$\SpcKE$ for the \emph{maximal} elementary abelian $p$-sections~$E$ of~$G$ by \emph{maximal} relations.

In \Cref{Prop:SpcKG-components}, we prove that the irreducible components of~$\SpcKG$ correspond to the maximal elementary abelian $p$-sections of~$G$ up to conjugation. It follows that the Krull dimension of~$\SpcKG$ is the sectional $p$-rank of~$G$, the maximal rank of elementary abelian $p$-sections.
(For comparison, recall that for the derived category these irreducible components correspond to maximal elementary abelian $p$-subgroups, not sections, and the Krull dimension is the usual $p$-rank.)

And of course, we discuss more examples. Using our techniques, we compute~$\SpcKG$ for some notable groups~$G$, in particular Klein-four (\Cref{Exa:Klein-four}) and the dihedral group (\Cref{Exa:D8}).

For the reader's convenience, we tried to keep Part~II somewhat self-contained. Here is a quick summary of the main ingredients we need from \Cref{part:stratification}.
\begin{Rec}
\label{Rec:Part-I}%
The canonical localization $\Upsilon_G\colon \cK(G)\onto \Db(\kkG)$ gives us an open piece $\VG:=\Spc(\Db(\kkG))\cong\Spech(\rmH^\sbull(G,\kk))$ of the spectrum, that we call  the `cohomological open'. We write $\upsilon_{G}=\Spc(\Upsilon_G)\colon \Vee{G}\hook\SpcKG$ for the inclusion.
For every $H\in \Sub{p}(G)$ we denote by $\Psi^H\colon \cK(G)\to \cK(\WGH)$ the \emph{modular $H$-fixed-points} tt-functor constructed in \Cref{sec:modular-fixed-pts}. It is characterized by $\Psi^H(\kk(X))\simeq \kk(X^H)$ on permutation modules and by the same formula degreewise on complexes.
We write $\check{\Psi}^H=\Upsilon_{\WGH}\circ\Psi^H$ for the composite $\cK(G)\to \cK(\WGH)\onto \Db(\kk(\WGH))$ all the way down to the derived category of~$\WGH$.
For every $H\in \Sub{p}(G)$, the tt-prime $\cM(H)=\Ker(\check{\Psi}^H)$ is a closed point of~$\SpcKG$.
It is also~$\cM(H)=\Ker(\bbF^H)$ where $\bbF^H=\Res^{\WGH}_1\circ\Psi^{H}\colon \cK(G)\to \cK(\WGH)\to \Db(\kk)$.
All closed points of~$\SpcKG$ are of this form by \Cref{Cor:closed-pts}.
We write $\psi^H=\Spc(\Psi^H)\colon \Spc(\cK(\WGH))\to \SpcKG$ for the continuous map induced by~$\Psi^H$ and $\check{\psi}^H=\Spc(\check{\Psi}^H)\colon\Vee{\WGH}\overset{\upsilon}{\hook}\Spc(\cK(\WGH))\xto{\psi^H}\SpcKG$ for its restriction to the cohomological open of~$\WGH$.
If we need to specify the ambient group we write~$\psi^{H\inn G}$ for~$\psi^{H}$, etc.
We saw in \Cref{sec:spectrum} that~$\psi^H$ is a closed map, and a closed immersion if~$H\normaleq G$ is normal.
Every prime $\cP\in\SpcKG$ is of the form $\cP=\cP_G(H,\gp):=\check{\psi}^H(\gp)$ for a $p$-subgroup~$H\le G$ and a point $\gp\in\Vee{\WGH}$ in the cohomological open of the Weyl group of~$H$, in a unique way up to $G$-conjugation; see \Cref{Thm:all-points}.
Hence the pieces~$\Vee{G}(H):=\check\psi(\Vee{\WGH})$ yield a partition $\SpcKG=\sqcup_{H\in\Sub{p}(G)/_G}\Vee{G}(H)$ into relatively open strata~$\Vee{G}(H)$, homeomorphic to~$\Vee{\WGH}$.
The crux of the problem is to understand how these strata~$\Vee{G}(H)\simeq\Vee{\WGH}$ attach together topologically, to build the space~$\SpcKG$.
\end{Rec}

\section{The colimit theorem}
\label{sec:topology-colim}%

To reduce the determination of $\SpcKG$ to the elementary abelian case, we invoke the \emph{category $\EAppG$ of elementary abelian $p$-sections} of a finite group~$G$.
Recall that a section of~$G$ is a pair $(H,K)$ of subgroups with $K$ normal in~$H$.
\begin{Cons}
\label{Cons:EA}
We denote by $\EAppG$ the category whose objects are pairs $(H,K)$ where $K\normaleq H$ are $p$-subgroups of~$G$ such that $H/K$ is elementary abelian.
Morphisms $(H,K)\to (H',K')$ are defined to be elements $g\in G$ such that
\[
K'\le K^g\le H^g\le H'.
\]
Composition of morphisms is defined by multiplication in~$G$.
Note that the rank of the elementary abelian group~$H/K$ increases or stays the same along any morphism $(H,K)\to (H',K')$ in this category.
\end{Cons}

\begin{Exas}
\label{Exas:EA-morphisms}%
Let us highlight three types of morphisms in~$\EAppG$.
\begin{enumerate}[wide, labelindent=0pt, label=\rm(\alph*), ref=\rm(\alph*)]
\item
\label{it:mor-a}%
We have an isomorphism $g\colon (H,K)\isoto (H^g,K^g)$ in~$\EAppG$ for every $g\in G$.
Intuitively, we can think of this as the group isomorphism $c_g\colon H/K\isoto H^g/K^g$.
\smallbreak
\item
\label{it:mor-b}%
For every object $(H',K')$ in~$\EAppG$ and every subgroup~$H\le H'$ containing~$K'$, we have a well-defined object $(H,K')$ and the morphism~$1\colon (H,K')\to (H',K')$.
Intuitively, we think of it as the inclusion $H/K'\hook H'/K'$ of a subgroup.
\smallbreak
\item
\label{it:mor-c}%
For $(H,K)$ in~$\EAppG$ and a subgroup~$\bar{L}=L/K$ of~$H/K$, for $K\le L\le H$, there is another morphism in~$\EAppG$ associated to~\mbox{$1\in G$}, namely~$1\colon(H,L)\to (H,K)$.
This one does not correspond to an intuitive group homomorphism $H/L\dashrightarrow H/K$, as~$K$ is \emph{smaller} than~$L$.
Instead, $H/L$ is the quotient of~$H/K$ by~$\bar{L}\normaleq H/K$.
This last morphism will be responsible for the modular $\bar{L}$-fixed-points functor.
\end{enumerate}
\end{Exas}

\begin{Rem}
\label{Rem:EA-generating-morphisms}%
Every morphism $g\colon (H,K)\to (H',K')$ in~$\EAppG$ is a composition of three morphisms of the above types~\ref{it:mor-a}, \ref{it:mor-b} and~\ref{it:mor-c} in the following canonical way:
\[
\xymatrix@R=0em@C=1em{
  H
  &&& H^g
  &&& H'
  &&& H'
  \\
  \nabla
  & \ar[r]^-{\textrm{\ref{it:mor-a}}}
  && \nabla
  & \ar[r]^-{\textrm{\ref{it:mor-b}}}
  && \nabla
  & \ar[r]^-{\textrm{\ref{it:mor-c}}}
  && \nabla
  \\
  K
  &&& K^g
  &&& K^g
  &&& K'
}
\]
where the first is given by~$g\in G$ and the last two are given by~$1\in G$.
This follows from \Cref{Exas:EA-morphisms} and $1\cdot 1\cdot g=g$, as composition in~$\EAppG$ is multiplication in~$G$.
\end{Rem}

\begin{Cons}
\label{Cons:L(g)}
To every object~$(H,K)$ in~$\EAppG$, we associate the tt-category $\cK(H/K)=\Kb(\perm(H/K;\kk))$.
For every morphism $g\colon (H,K)\to (H',K')$ in~$\EAppG$, %
we set~$\bar{K}=K^g/K'$ and we define a functor of tt-categories:
\[
\cK(g)\colon\cK(H'/K')\xto{\Psi^{\bar{K}}}\cK(H'/K^g)\xto{\Res}\cK(H^g/K^g)\xto{c_g^*}\cK(H/K)
\]
using that $(H'/K')/\bar{K}=H'/K^g$ for the modular fixed-points functor~$\Psi^{\bar{K}}$, and using that $H^g/K^g$ is a subgroup of~$H'/K^g$ for the restriction.

It follows from \Cref{Prop:Psi-Res,Cor:Psi-Psi} that~$\cK(-)$ is a contravariant (pseudo)\,functor on~$\EAppG$ with values in tt-categories:
\begin{equation}
\label{eq:K-diagram}%
\cK\colon \EAppG\op\too\ttCat.
\end{equation}
We can compose this with $\Spc(-)$, which incidentally makes the coherence of the 2-isomorphisms accompanying~\eqref{eq:K-diagram} irrelevant, and obtain a covariant functor from~$\EAppG$ to topological spaces.
Let us compare this diagram of spaces (and its colimit) with the space~$\SpcKG$.
For each $(H,K)\in\EAppG$, we have a tt-functor
\begin{equation}
\label{eq:K(G)-to-K(E)}%
\cK(G)\xto{\Res^G_H}\cK(H)\xto{\Psi^K}\cK(H/K)
\end{equation}
which yields a natural transformation from the constant functor $(H,K)\mapsto \cK(G)$ to the functor~$\cK\colon\EAppG\op\to \ttCat$ of~\eqref{eq:K-diagram}.
The above~$\Psi^{K}$ is~$\Psi^{K\inn H}$.
Since $H\le N_G K$, the tt-functor~\eqref{eq:K(G)-to-K(E)} is also~$\Res^{\WGK}_{H/K}\circ\Psi^{K\inn G}\colon \cK(G)\to \cK(\WGK)\to \cK(H/K)$.
Applying~$\Spc(-)$ to this observation, we obtain a commutative square:
\begin{equation}\label{eq:colim-comp}%
\vcenter{\xymatrix@C=4em@R=2em{
  \Spc(\cK(H/K))
  \ar[r]^-{\psi^{K\inn H}}
  \ar[d]_-{\rho_{H/K}}
  \ar[rd]^(.6){\varphi_{(H,K)}}
  &
  \Spc(\cK(H))
  \ar[d]^-{\rho_H}
  \\
  \Spc(\cK(\WGK))
  \ar[r]_-{\psi^{K\inn G}}
  &
  \SpcKG
  }
}\end{equation}
whose diagonal we baptize~$\varphi_{(H,K)}$.
In summary, we obtain a continuous map
\begin{equation}
\label{eq:colim-elab}%
\varphi\colon\colim_{(H,K)\in\EAppG}\Spc(\cK(H/K))\to \SpcKG
\end{equation}
whose component~$\varphi_{(H,K)}$ at~$(H,K)$ is the diagonal map in~\eqref{eq:colim-comp}.
\end{Cons}

\begin{Lem}
\label{Lem:preserve-dimension}%
\begin{enumerate}[label=\rm(\alph*), ref=\rm(\alph*)]
\item
Each of the maps $\Spc(\cK(g))\colon\Spc(\cK(H/K))\to\Spc(\cK(H'/K'))$ in the colimit diagram~\eqref{eq:colim-elab} is a closed immersion.
\smallbreak
\item
Each of the components $\varphi_{(H,K)}\colon\Spc(\cK(H/K))\to\SpcKG$ of~\eqref{eq:colim-elab} is closed and preserves the dimension of points (\ie the Krull dimension of their closure).
\end{enumerate}
\end{Lem}
\begin{proof}
These statements follow from two facts, see \Cref{Rec:Part-I}:
When $N\normaleq G$ is normal the map~$\psi^{N}\colon \Spc(\cK(G/N))\hook \Spc(\cK(G))$ is a closed immersion.
When $H\le G$ is any subgroup, the map~$\rho_H\colon \Spc(\cK(H))\to\Spc(\cK(G))$ is closed, hence lifts specializations, and it moreover satisfies `Incomparability' by~\cite{balmer:tt-separable}.
\end{proof}

We are now ready to prove \Cref{Thm:colim-intro}:
\begin{Thm}
\label{Thm:colim}%
For any finite group~$G$, the map~$\varphi$ in~\eqref{eq:colim-elab} is a homeomorphism.
\end{Thm}
\begin{proof}
Each component $\varphi_{(H,K)}$ is a closed map and thus~$\varphi$ is a closed map.
For surjectivity, by \Cref{Rec:Part-I}, we know that $\SpcKG$ is covered by the subsets~$\psi^K(\Vee{\WGK})$, over all $p$-subgroups~$K\le G$. Hence it suffices to know that the $\Img(\rho_E)$ cover~$\Vee{\WGK}=\Spc(\Db(\WGK))$ as~$E\le \WGK$ runs through all elementary $p$-subgroups. (Such an~$E$ must be of the form~$H/K$ for an object~$(H,K)\in\EAppG$.) This holds by a classical result of Quillen~\cite{quillen:spec-cohomology}; see~\cite[Theorem~4.10]{balmer:tt-separable}.

The key point is injectivity.
Take $\cP\in\Spc(\cK(H/K))$ and $\cP'\in\Spc(\cK(H'/K'))$ with same image in $\SpcKG$.
Write $\cP=\cP_{H/K}(L/K,\gp)$ for suitable arguments ($K\le L\le H$, $\gp\in\Vee{H/L}$) and note that the map induced by $1\colon(H,L)\to (H,K)$ in~$\EAppG$ sends $\cP_{H/L}(1,\gp)\in\Spc(\cK(H/L))$ to $\cP$. So we may assume $L=K$.
By \Cref{Rem:Spc-Res}, the image of $\cP=\cP_{H/K}(1,\gp)$ in $\SpcKG$ is $\cP_G(K,\bar\rho(\gp))$ where $\bar\rho\colon\Vee{H/K}\to\Vee{\WGK}$ is induced by restriction.
Similarly, we may assume $\cP'=\cP_{H'/K'}(1,\gp')$ for $\gp'\in\Vee{H'/K'}$ and we have $\cP_G(K,\bar\rho(\gp))=\cP_G(K',\bar\rho'(\gp'))$ in~$\SpcKG$ and need to show that~$\cP$ and $\cP'$ are identified in the colimit~\eqref{eq:colim-elab}.

By \Cref{Thm:all-points}, the relation $\cP_G(K,\bar\rho(\gp))=\cP_G(K',\bar\rho'(\gp'))$ can only hold because of $G$-conjugation, meaning that there exists~$g\in G$ such that $K'=K^g$ and $\bar\rho'(\gp')=\bar\rho(\gp)^g$ in~$\Vee{\Weyl{G}{K'}}$.
Using the map $g\colon(H,K)\to (H^g,K^g)$ in $\EAppG$ we may replace $H,K,\gp$ by $H^g,K^g,\gp^g$ and reduce to the case $K=K'$. In other words, we have two points $\cP=\cP_{H/K}(1,\gp)\in\Spc(\cK(H/K))$ and~$\cP'=\cP_{H'/K}(1,\gp')\in\Spc(\cK(H'/K))$ corresponding to two $p$-subgroups~$H,H'\le G$ containing the \emph{same} subgroup~$K$ as a normal subgroup and two cohomological primes $\gp\in \Vee{H/K}$ and~$\gp'\in\Vee{H'/K}$ such that $\bar\rho(\gp)=\bar\rho'(\gp')$ in~$\Vee{\WGK}$ under the maps~$\bar\rho$ and~$\bar\rho'$ induced by restriction along $H/K\le \WGK$ and $H'/K\le \WGK$ respectively.

If we let~$\bar{G}=\WGK=(N_G K)/K$, we have two elementary abelian $p$-subgroups $\bar{H}=H/K$ and $\bar{H}'=H'/K$ of~$\bar{G}$, each with a point in their cohomological open, $\gp\in\Vee{\bar{H}}$ and~$\gp'\in\Vee{\bar{H}'}$, and those two points have the same image in the cohomological open $\Vee{\bar{G}}$ of the `ambient' group~$\bar{G}$. By Quillen~\cite{quillen:spec-cohomology} (or~\cite[\S\,4]{balmer:tt-separable}) again, we know that this coalescence must happen because of an element $\bar{g}\in \bar{G}$, that is, a $g\in N_G K$, and a prime~$\gq\in\Vee{\bar{H}\cap{}^{g\!}\bar{H}'}$ that maps to~$\gp$ and to~$\gp'$ under the maps $\Vee{\bar{H}\cap{}^{g\!}\bar{H'}}\to \Vee{\bar{H}}$ and $\Vee{\bar{H}\cap{}^{g\!}\bar{H'}}\to \Vee{\bar{H'}}$ respectively.
But our category~$\EAppG$ contains all such conjugation-inclusion morphisms coming from the orbit category of~$G$. Specifically, we have two morphisms
$1\colon(H\cap {}^{g\!}H',K)\to (H,K)$ and $g\colon(H\cap {}^{g\!}H',K)\to (H',K)$ in~$\EAppG$, under which the point~$\cP_{(H\cap{}^{g\!}H')/K}(1,\gq)$ maps to~$\cP_{H/K}(1,\gp)=\cP$ and $\cP_{H'/K}(1,\gp')=\cP'$ respectively. This shows that $\cP=\cP'$ in the domain of~\eqref{eq:colim-elab} as required.
\end{proof}

\begin{Rem}
\label{Rem:closure}%
By \Cref{Prop:Spc-noetherian}, the space~$\SpcKG$ is noetherian. Hence the topology is entirely characterized by the inclusion of primes.
Now, suppose that~$\cP$ is the image under~$\varphi_{(H,K)}\colon \SpcKE\to\SpcKG$ of some~$\cP'\in \SpcKE$ for an elementary abelian subquotient~$E=H/K$ corresponding to a section~$(H,K)\in\EAppG$.
Then the only way for another prime~$\cQ\in\SpcKG$ to belong to the closure of~$\cP$ is to be itself the image of some point $\cQ'$ of~$\SpcKE$ in the closure of~$\cP'$.
This follows from \Cref{Lem:preserve-dimension}.
In other words, the question of inclusion of primes can also be reduced to the elementary abelian case.
\end{Rem}

\section{Invertible objects and twisted cohomology}
\label{sec:invertibles+H**}%

In this section we introduce a graded ring whose homogeneous spectrum helps us understand the topology on~$\Spc(\cK(G))$, at least for~$G$ elementary abelian.
This graded ring, called the \emph{twisted cohomology ring} (\Cref{Def:H**}), consists of morphisms between~$\unit$ and certain invertible objects.
It all starts in the cyclic case.
\begin{Exa}
\label{Exa:u_p}%
Let $C_p=\ideal{\sigma\mid \sigma^p=1}$ be the cyclic group of prime order~$p$, with a chosen generator.
We write $\kk C_p=\kk[\sigma]/(\sigma^p-1)$ as~$\kk[\tau]/\tau^p$ for~$\tau=\sigma-1$.
Then the coaugmentation and augmentation maps become:
\[
\eta:\kk\xto{1\mapsto \tau^{p-1}} \kk C_p\qquadtext{and} \eps:\kk C_p\xto{\tau\mapsto 0}\kk.
\]
For $p$ odd, we denote the first terms of the `standard' minimal resolution of $\kk$ by
\[
u_p=(0\to \kk C_p\xto{\tau} \kk C_p\xto{\eps}\kk\to 0).
\]
We view this in $\cK(C_p)$ with~$\kk$ in homological degree zero.
One can verify directly that $u_p$ is $\otimes$-invertible, with $u_p\potimes{-1}=u_p^\vee\cong(0\to \kk \xto{\eta} \kk C_p\xto{\tau}\kk C_p\to 0)$.
Alternatively, one can use the conservative pair of functors~$\bbF^{H}\colon \cK(C_p)\to \Db(\kk)$ for $H\in\{C_p\,,\,1\}$, corresponding to the only closed points $\cM(C_p)$ and~$\cM(1)$ of~$\Spc(\cK(C_p))$. Those functors map~$u_p$ to the $\otimes$-invertibles~$\kk$ and~$\kk[2]$ in~$\Db(\kk)$, respectively.

For $p=2$, we have a similar but shorter $\otimes$-invertible object in~$\cK(C_2)$
\[
u_2=(0\to \kk C_2\xto{\eps}\kk\to 0)
\]
again with~$\kk$ in degree zero.
\end{Exa}
\begin{Not}
\label{Not:2'}%
To avoid constantly distinguishing cases, we abbreviate
\begin{equation*}
\label{eq:2'}%
2':=\left\{
\begin{array}{cl}
  2 & \textrm{if }p>2
  \\
  1 & \textrm{if }p=2.
\end{array}\right.
\end{equation*}
\end{Not}

For any finite group~$G$ and any index-$p$ normal subgroup~$N$, we can inflate the $\otimes$-invertible $u_p$ of \Cref{Exa:u_p} along $\pi\colon G\onto G/N\simeq C_p$ to a $\otimes$-invertible in~$\cK(G)$.
\begin{Def}
\label{Def:u_N}%
Let $N\normal G$ be a normal subgroup of index~$p$. We define
\begin{equation}
\label{eq:u_N}%
\quad u_N:=\left\{
\begin{array}{rl}
  \cdots \to 0\to \kk (G/N)\xto{\tau} \, \kk (G/N)\xto{\eps}\kk\to 0 \to \cdots & \quad \textrm{if $p$ is odd}\!\!
\\[.5em]
  \cdots \to 0 \to \; \quad 0 \ \ \quad\to \  \kk(G/N) \xto{\eps} \kk\to 0 \to \cdots & \quad \textrm{if $p=2$}
\end{array}
\right.\kern-1em
\end{equation}
with $\kk$ in degree zero. We also define two morphisms
\[
a_N\colon \unit \to u_N
\qquadtext{and}
b_N\colon \unit\to u_N[-2']
\]
as follows. The morphism~$a_N$ is the identity in degree zero, independently of~$p$:
\[
\xymatrix@C=1em@R=1em{
\unit \ar[d]_-{a_N} \ar@{}[r]|-{=} &  \cdots \ar[r] & 0 \ar[r] \ar[d] & \kk \ar[r] \ar[d]^-{1} & 0 \ar[r] & \cdots
\\
u_N \ar@{}[r]|-{=} &  \cdots \ar[r] & \kk (G/N) \ar[r]_-{\eps} & \kk \ar[r] & 0 \ar[r] & \cdots
}
\]
The morphism $b_N$ is given by~$\eta\colon \kk\to \kk(G/N)$ in degree zero, as follows:
\[
\vcenter{\xymatrix@C=1em@R=1em{
\unit \ar[d]_-{b_N} \ar@{}[r]|-{=} && \kk \ar[r] \ar[d]^-{\eta} & 0 \ar[d]
\\
u_N[-1] \ar@{}[r]|-{=} && \kk (G/N) \ar[r]_-{\eps} & \kk
}}
\qquadtext{and}
\vcenter{\xymatrix@C=1em@R=1em{
\unit \ar[d]_-{b_N} \ar@{}[r]|-{=} && \kk \ar[r] \ar[d]^-{\eta} & 0 \ar[d] \ar[r] & 0 \ar[d]
\\
u_N[-2] \ar@{}[r]|-{=} & & \kk (G/N) \ar[r]_-{\tau} & \kk (G/N) \ar[r]_-{\eps} & \kk
}}
\]
where the target~$u_N$ is shifted once to the right for $p=2$ (as in the left-hand diagram above) and shifted twice for $p>2$ (as in the right-hand diagram).

When $p$ is odd there is furthermore a third morphism $c_N\colon \unit \to u_N[-1]$, that is defined to be~$\eta\colon \kk\to \kk(G/N)$ in degree zero.
This $c_N$ will play a lesser role.

In statements made for all primes~$p$, simply ignore $c_N$ in the case~$p=2$ (or think $c_N=0$).
Here is an example of such a statement, whose meaning should now be clear: \emph{The morphisms $a_N$ and $b_N$, and $c_N$ (for $p$ odd), are inflated from~$G/N$}.
\end{Def}
\begin{Rem}
\label{Rem:u_N}%
Technically, $u_N$ depends not only on an index-$p$ subgroup~$N\normal G$ but also on the choice of a generator of~$G/N$, to identify $G/N$ with~$C_p$. If one needs to make this distinction, one can write $u_\pi$ for a chosen epimorphism $\pi\colon G\onto C_p$. This does not change the isomorphism type of~$u_N$, namely $\ker(\pi)=\ker(\pi')$ implies $u_\pi\cong u_{\pi'}$.
(We expand on this topic in \Cref{Rem:similar-coordinates}.)
\end{Rem}

\begin{Lem}
\label{Lem:u_N^q}%
Let $N\normal G$ be a normal subgroup of index~$p$ and let~$q\geq 1$. Then there is a canonical isomorphism in~$\cK(G)$
\[
u_N\potimes{q}\cong(\cdots 0\to \kk(G/N) \xto{\tau}\kk(G/N) \xto{\tau^{p-1}\!\!} \cdots \xto{\tau} \kk(G/N) \xto{\eps}\kk\to 0\cdots )
\]
where the first~$\kk(G/N)$ sits in homological degree~$2'\cdot q$ and~$\kk$ sits in degree~$0$.
\end{Lem}
\begin{proof}
Since~$u_N$ is inflated along~$G\onto G/N$ it suffices to prove the lemma for the complex~$u_p$ of \Cref{Exa:u_p} over~$G=C_p$.
We give an argument in the case~$p=2$ and leave the minor modifications required for $p$ odd to the reader.
Observe that $\kk C_2\otimes u_2\cong\Ind^{C_2}_1\Res^{C_2}_1u_2\cong\Ind^{C_2}_1\kk[1]\cong\kk C_2[1]$.
By induction on~$q$, we have
\begin{align*}
  u_2\potimes{q}&\cong\cone(\kk C_2\xto{\eps}\kk)\otimes u_2\potimes{(q-1)}\cong\cone\left(\kk C_2[q-1]\to u_2\potimes{(q-1)}\right)\\
                &\cong\left(0\to\kk C_2\to\kk C_2 \xto{\tau}\kk C_2\to\cdots \xto{\tau} \kk C_2 \xto{\eps}\kk\to 0\right)
\end{align*}
where the first $\kk C_2$ sits in homological degree~$q$ and $\kk$ sits in degree~$0$.
It remains to identify the first map with~$\tau$.
This follows (up to a non-zero scalar multiple, which does not matter) from the fact that $u_2\potimes{q}$ has homology concentrated in degree~$q$, isomorphic to~$\kk$.
\end{proof}

\begin{Rem}
\label{Rem:zeta_N}%
The morphism~$b_N\colon \unit\to u_N[-2']$ of \Cref{Def:u_N} is a quasi-iso\-morphism and the fraction
\[
\zeta_N:=(b_N[2'])\inv\circ a_N\colon \unit \to u_N \lto \unit[2']
\]
is a well-known morphism $\zeta_N\in\Hom_{\Db(\kkG)}(\unit,\unit[2'])=\rmH^{2'}(G,\kk)$ in the derived category~$\Db(\kkG)$. For $G$ elementary abelian, these~$\zeta_N$ generate the cohomology $\kk$-algebra~$\rmH^\sbull(G,\kk)$, on the nose for $p=2$ and modulo nilpotents for $p$ odd.

We sometimes write $\zeta^+_N=\frac{a_N}{b_N}$ for~$\zeta_N$ in order to distinguish it from the inverse fraction $\zeta^-_N:=\frac{b_N}{a_N}$ that exists wherever $a_N$ is inverted.
Of course, when both $a_N$ and $b_N$ are inverted, we have $\zeta^-_N=(\zeta^+_N)\inv=\zeta_N\inv$.
\end{Rem}

\begin{Rem}
\label{Rem:switch=1}%
The switch of factors $(12)\colon u_N\otimes u_N\cong u_N\otimes u_N$ can be computed directly to be the identity (over~$C_p$, then inflate).
Alternatively, it must be multiplication by a square-one element of~$\Aut(\unit)=\kk^\times$, hence~$\pm1$.
One can then apply the tensor-functor $\Psi^G\colon \cK(G)\to \Db(\kk)$, under which $u_N$ goes to~$\unit$, to rule out~$-1$.

It follows that for $p$ odd, $u_N[-1]$ has switch~$-1$, and consequently every morphism $\unit\to u_N[-1]$ must square to zero. In particular $c_N\otimes c_N=0$.
This nilpotence explains why $c_N$ will play no significant role in the topology.
\end{Rem}

We can describe the image under modular fixed-points functors of the $\otimes$-invertible objects~$u_N$ and of the morphisms~$a_N$ and~$b_N$. (We leave $c_N$ as an exercise.)
\begin{Prop}
\label{Prop:Psi^H(a/b)}%
Let $H\normaleq G$ be a normal $p$-subgroup. Then for every index-$p$ normal subgroup $N\normal G$, we have in~$\cK(G/H)$
\[
\Psi^H(u_N)\cong\left\{
\begin{array}{cl}
u_{N/H} & \textrm{if }H\le N
\\
{\unit} & \textrm{if }H\not\le N
\end{array}
\right.
\]
and under this identification
\[
\Psi^H(a_N)=\left\{
\begin{array}{cl}
a_{N/H} & \textrm{if }H\le N
\\
1_{\unit} & \textrm{if }H\not\le N
\end{array}
\right.
\quadtext{and}
\Psi^H(b_N)=\left\{
\begin{array}{cl}
b_{N/H} & \textrm{if }H\le N
\\
0 & \textrm{if }H\not\le N.\!\!
\end{array}
\right.
\]
\end{Prop}
\begin{proof}
Direct from \Cref{Def:u_N} and $\Psi^H(\kk(X))\cong\kk(X^H)$ for $X=G/N$.
\end{proof}

For restriction, there is an analogous pattern but with the cases `swapped'.
\begin{Prop}
\label{Prop:Res_H(a/b)}%
Let $H\le G$ be a subgroup. Then for every index-$p$ normal subgroup $N\normal G$, we have in~$\cK(H)$
\[
\Res^G_H(u_N)\cong\left\{
\begin{array}{cl}
\unit[2'] & \textrm{if }H\le N
\\
u_{N\cap H} & \textrm{if }H\not\le N
\end{array}
\right.
\]
and under this identification
\[
\Res^G_H(a_N)=\left\{
\begin{array}{cl}
0 & \textrm{if }H\le N
\\
a_{N\cap H} & \textrm{if }H\not\le N
\end{array}
\right.
\quadtext{and}
\Res^G_H(b_N)=\left\{
\begin{array}{cl}
1_{\unit} & \textrm{if }H\le N
\\
b_{N\cap H} & \textrm{if }H\not\le N.\!\!
\end{array}
\right.
\]
\end{Prop}
\begin{proof}
Direct from \Cref{Def:u_N} and the Mackey formula for $\Res^G_H(\kk(G/N))$.
\end{proof}

We can combine the above two propositions and handle~$\Psi^{H}$ for non-normal~$H$, since by definition~$\Psi^{H\inn G}=\Psi^{H\inn N_G H}\circ\Res^G_{N_G H}$.
Here is an application of this.
\begin{Cor}
\label{Cor:F^H(a/b)}%
Let $H\le G$ be a $p$-subgroup and~$N\normal G$ of index~$p$. Recall the `residue' tt-functor $\bbF^H=\Res_1\circ\Psi^H\colon \cK(G)\to \Db(\kk)$ at the closed point~$\cM(H)$.
\begin{enumerate}[label=\rm(\alph*), ref=\rm(\alph*)]
\item
If $H\not\le N$ then $\bbF^H(a_N)$ is an isomorphism.
\smallbreak
\item
If $H\le N$ then $\bbF^H(b_N)$ is an isomorphism.
\end{enumerate}
\end{Cor}
\begin{proof}
We apply \Cref{Prop:Res_H(a/b)} for $N_G H\le G$ and~\Cref{Prop:Psi^H(a/b)} for $H\normal N_G H$.
For~(a), $H\not\le N$ forces $N_G H\not\le N$ and $H\not\le N\cap N_G H$. Hence $\Psi^H(a_N)=\Psi^{H\inn N_G H}\Res_{N_G H}(a_N)=\Psi^{H\inn N_G H}(a_{N\cap N_G H})=1_{\unit}$ is an isomorphism.
Similarly for~(b), if $N_G H\le N$ then $\Psi^H(b_N)$ is an isomorphism and if $N_G H\not\le N$ it is the quasi-isomorphism $b_{(N\cap N_G H)/H}$.
Thus $\bbF^H(b_N)$ is an isomorphism in~$\Db(\kk)$.
\end{proof}

Let us now prove that the morphisms~$a_N$ and~$b_N$, and~$c_N$ (for $p$ odd), generate all morphisms from the unit~$\unit$ to tensor products of~$u_N$'s.
This is a critical fact.
\begin{Lem}
\label{Lem:finite-generation}%
Let $N_1,\ldots,N_\ell$ be index-$p$ normal subgroups of~$G$ and abbreviate $u_i:=u_{N_i}$ for $i=1,\ldots,\ell$ and similarly $a_i:=a_{N_i}$ and~$b_i:=b_{N_i}$ and $c_i:=c_{N_i}$ (see \Cref{Def:u_N}).
Let $q_1,\ldots,q_\ell\in\bbN$ be non-negative integers and $s\in \bbZ$. Then every morphism $f\colon \unit\to u_{1}\potimes{q_1}\otimes\cdots\otimes u_{\ell}\potimes{q_\ell}[s]$ in~$\cK(G)$ is a $\kk$-linear combination of tensor products of (\ie a `polynomial' in) the morphisms $a_{i}$ and~$b_{i}$, and~$c_i$ (for $p$ odd).
\end{Lem}
\begin{proof}
We proceed by induction on~$\ell$. The case~$\ell=0$ is just $\End_{\cK(G)}^\sbull(\unit)=\kk$. Suppose $\ell\ge 1$ and the result known for~$\ell-1$.
Up to reducing to~$\ell-1$, we can assume that the $N_1,\ldots,N_\ell$ are all distinct.
Set for readability
\[
v:=u_{1}\potimes{q_1}\otimes\cdots\otimes u_{\ell-1}\potimes{q_{\ell-1}}[s],
\qquad
N:=N_\ell,
\qquad
u:=u_\ell=u_{N}
\qquadtext{and}
q:=q_\ell
\]
so that $f$ is a morphism of the form
\[
f\colon \unit\to v\otimes u\potimes{q}\,.
\]
We then proceed by induction on~$q\ge 0$. We assume the result known for~$q-1$ (the case $q=0$ holds by induction hypothesis on~$\ell$).
The proof will now depend on~$p$.

Suppose first that $p=2$. Consider the exact triangle in~$\cK(G)$
\begin{equation}
\label{eq:aux-f.g-1}%
\vcenter{
\xymatrix@C=.9em@R=1.5em{
u\potimes{(q-1)} \ar@{}[r]|-{=} \ar[d]_-{a_N}
& \cdots 0 \ar[r]
& 0 \ar[r] \ar[d]
& \kk(G/N) \ar[r]^-{\tau} \ar@{=}[d]
& \cdots \ar[r]^-{\tau}
& \kk(G/N) \ar[r]^-{\eps} \ar@{=}[d]
& \kk \ar[r] \ar@{=}[d]
& 0 \cdots
\\
u\potimes{q} \ar@{}[r]|-{=} \ar[d]_-{}
& \cdots 0 \ar[r]
& \kk(G/N) \ar[r]^-{\tau} \ar@{=}[d]
& \kk(G/N) \ar[r]^-{\tau} \ar[d]
& \cdots \ar[r]^-{\tau}
& \kk(G/N) \ar[r]^-{\eps} \ar[d]
& \kk \ar[r] \ar[d]
& 0 \cdots
\\
\kk(G/N)[q] \ar@{}[r]|-{=}
& \cdots 0 \ar[r]
& \kk(G/N) \ar[r]
& 0 \ar[r]
& \cdots \ar[r]
& 0 \ar[r]
& 0 \ar[r]
& 0 \cdots
}}\kern-1em
\end{equation}
where $\kk$ is in degree zero. (See~\Cref{Lem:u_N^q}.)
Tensoring the above triangle with $v$ and applying $\Hom_{G}(\unit,-):=\Hom_{\cK(G)}(\unit,-)$ we get an exact sequence
\begin{equation}
\label{eq:aux-f.g-2}%
\vcenter{
\xymatrix@C=1.5em@R=.4em{
\Hom_G(\unit,v\otimes u\potimes{(q-1)})\ar[r]^-{\cdot a_N}&\Hom_G(\unit,v\otimes u\potimes{q})\ar[r]&\Hom_G(\unit,v\otimes \kk(G/N)[q])\ar@{=}[d]
\\
& f \ar@{}[r]|-{\displaystyle\longmapsto} \ar@{}[u]|-{\rotatebox{90}{$\in$}}
& \kern-1em f'\in \Hom_N(\unit,\Res^G_N(v)[q])
}}\kern-.8em
\end{equation}
Our morphism $f$ belongs to the middle group. By adjunction, the right-hand term is~$\Hom_{N}(\unit,\Res^G_N(v)[q])$. Now since all~$N_1,\ldots,N_\ell=N$ are distinct, we can apply \Cref{Prop:Res_H(a/b)} to compute~$\Res^G_N(v)$ and we know by induction hypothesis (on~$\ell$) that the image~$f'$ of our~$f$ in this group $\Hom_{N}(\unit,\Res^G_N(v)[q])$ is a $\kk$-linear combination of tensor products of $a_{N_i\cap N}$ and~$b_{N_j\cap N}$ for $1\le i,j\le \ell-1$, performed over the group~$N$. We can perform the `same' $\kk$-linear combination of tensor products of $a_i$'s and $b_j$'s over the group~$G$, thus defining a morphism $f''\in \Hom_G(\unit,v[q])$. We can now multiply $f''$ with $b_{N}\potimes{q}\colon \unit\to u_N\potimes{q}[-q]$ to obtain a morphism $f'' b_N^q$ in the same group~$\Hom_G(\unit,v\otimes u\potimes{q})$ that contains~$f$. Direct computation shows that the image of this $f''b_N^q$ in~$\Hom_N(\unit,\Res_N(v)[s])$ is also equal to~$f'$. The key point is that $b_N\potimes{q}$ is simply $\eta\colon \kk\to \kk(G/N)$ in degree~$q$ and this $\eta$ is also the unit of the $\Res^G_N\adj \Ind_N^G$ adjunction.
In other words, the difference $f-f'' b_N^q$ comes from the left-hand group $\Hom_G(\unit,v\otimes u\potimes{(q-1)})$ in the exact sequence~\eqref{eq:aux-f.g-2}, reading
\[
f=f''b_N^q+f''' a_N
\]
for some $f'''\in \Hom_G(\unit,v\otimes u\potimes{(q-1)})$. By induction hypothesis (on~$q$), $f'''$ is a polynomial in $a_i$'s and $b_j$'s. Since $f''$ also was such a polynomial, so is~$f$.

The proof for~$p$ odd follows a similar pattern of induction on~$q$, with one complication. The cone of the canonical map $a_N\colon u_{N}\potimes{(q-1)}\to u_N\potimes{q}$ is not simply~$\kk(G/N)$ in a single degree as in~\eqref{eq:aux-f.g-1} but rather the complex
\[
C:=\big(\cdots \to 0\to \kk(G/N)\xto{\ \tau\ }\kk(G/N)\to 0\to \cdots\big)
\]
with $\kk(G/N)$ in two consecutive degrees~$2q$ and~$2q-1$. So the exact sequence
\begin{equation}
\label{eq:aux-f.g-3}%
\vcenter{
\xymatrix@C=1.5em@R=.4em{
\Hom_G(\unit,v\otimes u\potimes{(q-1)})\ar[r]^-{\cdot a_N}&\Hom_G(\unit,v\otimes u\potimes{q})\ar[r]&\Hom_G(\unit,v\otimes C)
}}
\end{equation}
has a more complicated third term than the one of~\eqref{eq:aux-f.g-2}. That third term $\Hom_G(\unit,v\otimes C)$ itself fits in its own exact sequence associated to the exact triangle $\kk(G/N)[2q-1]\xto {\tau} \kk(G/N)[2q-1]\to C \to \kk(G/N)[2q]$. Each of the terms $\Hom_G(\unit,v\otimes \kk(G/N)[\ast])\cong\Hom_N(\unit,\Res^G_N(v)[\ast])$ can be computed as before, by adjunction.
The image of $f$ in $\Hom_N(\unit,\Res^G_N(v)[2q])$ can again be lifted to a polynomial $f'b^q_N:\unit\to v\otimes u\potimes{q}$ so that the image of the difference $f-f'b^q_N$ in $\Hom_G(\unit,v\otimes C)$ comes from some element in~$\Hom_N(\unit,\Res^G_N(v)[2q-1])$.
That element may be lifted to a polynomial $f''b^{q-1}_Nc_N:\unit\to v\otimes u\potimes{q}$, and we obtain
\[
f=f'b_N^q+f''b_N^{q-1}c_N+f'''a_N
\]
for some $f'''\in \Hom_G(\unit,v\otimes u\potimes{(q-1)})$ similarly as before.
\end{proof}

We can now assemble all the hom groups of \Cref{Lem:finite-generation} into a big graded ring.
\begin{Def}
\label{Def:H**}%
We denote the set of all index-$p$ normal subgroups of~$G$ by
\begin{equation}
\label{eq:N(G)}%
\cN=\cN(G):=\SET{N\normal G\,}{\,[G\!:\!N]=p}.
\end{equation}
Let $\bbN^{\cN}=\bbN^{\cN(G)}=\{q\colon \cN(G)\to \bbN\}$ be the monoid of \emph{twists}, \ie tuples of non-negative integers indexed by this finite set.
Consider the $(\bbZ\times\bbN^{\cN})$-graded ring
\begin{equation}
\label{eq:Rall}%
\Rall(G)=\Rall(G;\kk) := \bigoplus_{s\in\bbZ} \ \bigoplus_{q\in\bbN^{\cN}} \Hom_{\cK(G)}\Big(\unit\,,\,\bigotimes_{N\in\cN}u_N\potimes{q(N)}[s]\Big).
\end{equation}
Its multiplication is induced by the tensor product in~$\cK(G)$.
We call $\Rall(G)$ the \emph{(permutation) twisted cohomology ring of~$G$}.
It is convenient to simply write
\begin{equation}
\label{eq:1(q)}%
\unit(q)=\bigotimes_{N\in\cN}(u_N)\potimes{q(N)}
\end{equation}
for every twist~$q\in \bbN^{\cN(G)}$ and thus abbreviate~$\rmH^{s,q}(G)=\Hom\big(\unit,\unit(q)[s]\big)$.
\end{Def}

\begin{Rem}
\label{Rem:graded-commutative}%
The graded ring~$\Rall(G)$ is graded-commutative by using only the parity of the shift, not the twist; see \Cref{Rem:switch=1}.
In other words, we have
\[
\qquad \qquad h_1\cdot h_2= (-1)^{s_1\cdot s_2}\,h_2\cdot h_1\qquadtext{when}h_i\in\rmH^{s_i,q_i}(G).
\]
For instance, for~$p$ odd, when dealing with the morphisms $a_N$ and $b_N$, which land in even shifts of the object~$u_N$, we do not have to worry too much about the order.
This explains the `unordered' notation $\zeta_N=\frac{a_N}{b_N}$ used in \Cref{Rem:zeta_N}.
\end{Rem}

The critical \Cref{Lem:finite-generation} gives the main property of this construction:
\begin{Thm}
\label{Thm:Rall-noeth}%
The twisted cohomology ring~$\Rall(G)$ of \Cref{Def:H**} is a $\kk$-algebra generated by the finitely many elements $a_N$ and $b_N$, and $c_N$ (for $p$ odd), of \Cref{Def:u_N}, over all~$N\normal G$ of index~$p$. In particular $\Rall(G)$ is noetherian.
\qed
\end{Thm}

\begin{Exa}
\label{Exa:Rall(C_p)}%
The reader can verify by hand that $\Rall(C_2)=\kk[a_N,b_N]$, without relations, and that $\Rall(C_p)=\kk[a_N,b_N,c_N]/\ideal{c_N^2}$ for $p$ odd, where in both cases $N=1$ is the only~$N\in\cN(C_p)$.
This example is deceptive, for the $\{a_N,b_N,c_N\}_{N\in\cN}$ usually satisfy some relations, as the reader can already check for $G=C_2\times C_2$ for instance.
We systematically discuss these relations in \Cref{sec:presentation-H**}.
\end{Exa}

We conclude this section with some commentary.

\begin{Rem}
The name `cohomology' in \Cref{Def:H**} is used in the loose sense of a graded endomorphism ring of the unit in a tensor-triangulated category.
However, since we are using the tt-category~$\cK(G)$ and not~$\Db(\kkG)$, the ring~$\Rall(G)$ is quite different from~$\rmH^\sbull(G,\kk)$ in general.
In fact, $\Rall(G)$ could even be rather dull.
For instance, if $G$ is a non-cyclic simple group then $\cN(G)=\varnothing$ and $\Rall(G)=\kk$.
We will make serious use of~$\Rall(G)$ in~\Cref{sec:abelem} to describe~$\SpcKG$ for $G$ elementary abelian.
In that case, $\rmH^\sbull(G,\kk)$ is a localization of~$\Rall(G)$.
See \Cref{Exa:R(1)=H*}.
\end{Rem}

\begin{Rem}
\label{Rem:twists}%
By \Cref{Prop:Psi^H(a/b)}, there is no `collision' in the twists: If there is an isomorphism $\unit(q)[s]\simeq\unit(q')[s']$ in~$\cK(G)$ then we must have $q=q'$ in~$\bbN^{\cN}$ and~$s=s'$ in~$\bbZ$.
The latter is clear from $\bbF^G(u_N)\cong\unit$ in~$\Db(\kk)$, independently of~$N$. We then conclude from~$\bbF^N(\unit(q))\simeq \unit[2'q(N)]$ in~$\Db(\kk)$, for each~$N\in\cN$.
\end{Rem}

\begin{Rem}
We only use positive twists~$q(N)$ in~\eqref{eq:Rall}.
The reader can verify that already for $G=C_p$ cyclic, the $\bbZ^2$-graded ring~$\oplus_{(s,q)\in\bbZ^2}\Hom(\unit,u_p\potimes{q}[s])$ is not noetherian.
See for instance~\cite{dugger-et-al:C2} for $p=2$.
However, negatively twisted elements tend to be nilpotent.
So the $\bbZ\times\mathbf{\bbZ}^{\cN}$-graded version of~$\Rall(G)$ may yield the same topological information as our $\bbZ\times\bbN^{\cN}$-graded one.
We have not pushed this investigation of negative twists, as it brought no benefit to our analysis.
\end{Rem}

\section{An open cover of the spectrum}
\label{sec:goodopen}%

In this section, we extract some topological information about~$\SpcKG$ from the twisted cohomology ring~$\Rall(G)$ of \Cref{Def:H**} and the maps~$a_N$ and~$b_N$ of \Cref{Def:u_N}, associated to every index-$p$ normal subgroup~$N$ in~$\cN=\cN(G)$.

Recall from \Cref{Cons:kos} that we can use tensor-induction to associate to every subgroup~$H\le G$ a \emph{Koszul object}~$\kos[G]{H}=\tInd_H^G(0\to \kk\xto{1} \kk\to 0)$. It generates in~$\cK(G)$ the tt-ideal $\Ker(\Res^G_H)$, see \Cref{Prop:Ker(Res)}:
\begin{equation}
\label{eq:kos-Res}%
\ideal{\kos[G]{H}}_{\cK(G)}=\Ker\big(\Res^G_H\colon \cK(G)\to\cK(H)\big).
\end{equation}

\begin{Lem}
\label{Lem:cone-a-b}%
Let $N\normal G$ be a normal subgroup of index~$p$. Then we have:
\begin{enumerate}[label=\rm(\alph*), ref=\rm(\alph*)]
\item
\label{it:cone-a}%
In~$\cK(G)$, the object $\cone(a_N)$ generates the same thick subcategory as~$\kk(G/N)$. In particular, $\supp(\cone(a_N))=\supp(\kk(G/N))$.
\smallbreak
\item
\label{it:cone-b}%
In~$\cK(G)$, the object $\cone(b_N)$ generates the same thick tensor-ideal as~$\kos[G]{N}$. In particular, $\supp(\cone(b_N))=\supp(\kos[G]{N})=\supp(\Ker(\Res^G_N))$.
\end{enumerate}
\end{Lem}
\begin{proof}
For $p=2$, we have $\cone(a_N)=\kk(G/N)[1]$ so the first case is clear. For $p$ odd, we have $\cone(a_N)[-1]\simeq(0\to\kk(G/N)\xto{\tau}\kk(G/N)\to 0)=\cone(\tau\restr{\kk(G/N)})$. Hence $\cone(a_N)\in \thick({\kk(G/N)})$. Conversely, since $\tau^p=0$, the octahedron axiom inductively shows that $\kk(G/N)\in\thick({\cone(\tau\restr{\kk(G/N)})})$. This settles~\ref{it:cone-a}.

For~\ref{it:cone-b}, the complex $s:=\cone(b_N)[2']$ becomes split exact when restricted to~$N$ since it is inflated from an exact complex on~$G/N$. In degree one we have $s_1=\kk(G/N)$, whereas~$s_0=\kk$. Hence \Cref{Cor:s-generates} tells us that the complex~$s$ generates the tt-ideal~$\Ker(\Res^G_N\colon \cK(G)\to \cK(N))$. We conclude by~\eqref{eq:kos-Res}.
\end{proof}

\begin{Cor}
\label{Cor:cone-a-b}%
Let $N\normal G$ be of index~$p$. Then $\cone(a_N)\otimes\cone(b_N)=0$.
\end{Cor}
\begin{proof}
By \Cref{Lem:cone-a-b} it suffices to show $\kk(G/N)\otimes\kos[G]{N}=0$. By Frobenius, this follows from~$\Res^G_N(\kos[G]{N})=0$, which holds by~\eqref{eq:kos-Res}.
\end{proof}

We now relate the spectrum of~$\cK(G)$ to the homogeneous spectrum of~$\Rall(G)$, in the spirit of~\cite{balmer:sss}. The comparison map of~\cite{balmer:sss} is denoted by~$\rho^\sbull$ but we prefer a more descriptive notation (and here, the letter $\rho$ is reserved for~$\Spc(\Res)$).
\begin{Prop}
\label{Prop:comp}%
There is a continuous `comparison' map
\[
\comp_G\colon \SpcKG\too\Spech(\Rall(G))
\]
mapping a tt-prime~$\cP$ to the ideal generated by those homogeneous~$f\in\Rall(G)$ whose cone does not belong to~$\cP$.
It is characterized by the fact that for all~$f$
\begin{equation}
\label{eq:comp-closed}%
\comp_G\inv(Z(f))=\supp(\cone(f))=\SET{\cP}{f\textrm{ is not invertible in }\cK(G)/\cP}
\end{equation}
where $Z(f)=\SET{\gp}{f\in \gp}$ is the closed subset of~$\Spech(\Rall(G))$ defined by~$f$.
\end{Prop}
\begin{proof}
Compare \cite[Theorem~3.10]{dellambrogio-stevenson:even-more-spectra}.
The fact that the homogeneous ideal $\comp_G(\cP)$ is prime comes from \cite[Theorem~4.5]{balmer:sss}.
Equation~\eqref{eq:comp-closed} is essentially a reformulation of the definition.
\end{proof}

\begin{Rem}
The usual notation for~$Z(f)$ would be~$V(f)$, and $D(f)$ for its open complement. Here, we already use $V$ for~$\Vee{G}$ and for~$\Vee{G}(H)$, and the letter $D$ is certainly overworked in our trade.
So we stick to $Z(f)$ and~$Z(f)^c$.
\end{Rem}

\begin{Not}
\label{Not:open(a-b)}%
In view of \Cref{Prop:comp}, for any~$f$, the open subset of~$\SpcKG$
\begin{equation}
\label{eq:open(f)}%
\open(f):=\open(\cone(f))=\SET{\cP}{f\textrm{ is invertible in }\cK(G)/\cP}
\end{equation}
is the preimage by~$\comp_G\colon \SpcKG\to \Spech(\Rall(G))$ of the principal open $Z(f)^c=\SET{\gp}{f\notin\gp}$.
It is the open locus of~$\SpcKG$ where $f$ is invertible.
In particular, our distinguished elements~$a_N$ and~$b_N$ (see \Cref{Def:u_N}) give us the following open subsets of~$\SpcKG$, for every~$N\in\cN(G)$:
\begin{align*}
& \open(a_N)=\comp_G\inv(Z(a_N)^c)\textrm{, the open where $a_N$ is invertible, and}
\\
& \open(b_N)=\comp_G\inv(Z(b_N)^c)\textrm{, the open where $b_N$ is invertible}.
\end{align*}
Since $(c_N)^2=0$ by~\Cref{Rem:switch=1}, we do not have much use for~$\open(c_N)=\varnothing$.
\end{Not}
\begin{Cor}
\label{Cor:open(ab)}%
With notation as above, we have for every $N\normal G$ of index~$p$
\[
\open(a_N)\cup \open(b_N)=\SpcKG.
\]
\end{Cor}
\begin{proof}
We compute $\open(a_N)\cup\open(b_N)=\open(\cone(a_N))\cup \open(\cone(b_N))=\open(\cone(a_N)\otimes\cone(b_N))=\open(0_{\cK(G)})=\SpcKG$, using \Cref{Cor:cone-a-b}.
\end{proof}

\begin{Rem}
\label{Rem:loc-triv}%
Every object~$u_N$ is not only $\otimes$-invertible in~$\cK(G)$ but actually locally trivial over~$\SpcKG$, which is a stronger property in general tt-geometry.
Indeed, \Cref{Cor:open(ab)} tells us that around each point of~$\SpcKG$, either $u_N$ becomes isomorphic to~$\unit$ via~$a_N$, or $u_N$ becomes isomorphic to~$\unit[2']$ via~$b_N$.
This holds for \emph{one} invertible~$u_N$.
We now construct a fine enough open cover of~$\SpcKG$ such that \emph{every} $u_N$ is trivialized on each open.
\end{Rem}

\begin{Prop}
\label{Prop:U(H)}%
Let $H\le G$ be a $p$-subgroup. Define an open of~$\SpcKG$ by
\begin{equation}
\label{eq:U(H)-def}%
\goodopen{H}=\goodopen[G]{H}:=\bigcap_{\ouratop{N\in\cN}{H\not\le N}} \open(a_N) \ \cap
\bigcap_{\ouratop{N\in\cN}{H\le N}} \open(b_N).
\end{equation}
Then the closed point $\cM(H)\in\SpcKG$ belongs to this open~$\goodopen{H}$. Consequently~$\{\goodopen{H}\}_{H\in\Sub{p}(G)}$ is an open cover of~$\SpcKG$.
\end{Prop}
\begin{proof}
The point $\cM(H)=\Ker(\bbF^H)$ belongs to~$\goodopen{H}$ by \Cref{Cor:F^H(a/b)}.
It follows by general tt-geometry that $\{\goodopen{H}\}_H$ is a cover:
Let $\cP\in\SpcKG$; there exists a closed point in~$\adhpt{\cP}$, that is, some $\cM(H)$ that admits~$\cP$ as a generalization; but then $\cM(H)\in\goodopen{H}$ forces $\cP\in \goodopen{H}$ since open subsets are generalization-closed.
\end{proof}

For a $p$-group, we now discuss~$\goodopen{H}$ at the two extremes $H=1$ and~$H=G$.

\begin{Rec}
\label{Rec:Frattini}%
Let $G$ be a $p$-group and~$F=F(G)=\cap_{N\in\cN(G)}N$ be its Frattini subgroup. So $F\normal G$ and $G/F$ is the largest elementary abelian quotient of~$G$.
\end{Rec}

\begin{Prop}
\label{Prop:U(1)}%
Let $G$ be a $p$-group with Frattini subgroup~$F$.
The closed complement of the open~$\goodopen[G]{1}$ is the support of $\kos[G]{F}$, \ie the closed support of the tt-ideal~$\Ker(\Res^G_F)$ of~$\cK(G)$.
In particular, if~$G$ is elementary abelian then $\goodopen[G]{1}$ is equal to the cohomological open~$\Vee{G}=\Spc(\Db(\kk G))\cong\Spech(\rmH^\sbull(G,\kk))$.
\end{Prop}
\begin{proof}
By definition, $\goodopen{1}=\cap_{N\in\cN}\open(b_N)$. By \Cref{Lem:cone-a-b}, its closed complement is $\cup_{N\in\cN}\supp(\kos[G]{N})$.
By \Cref{Cor:supp(kos)}, for every $K\le G$
\begin{equation}\label{eq:supp-kos}%
\supp(\kos[G]{K})=\SET{\cP(H,\gp)}{H\not\le_G K}
\end{equation}
(taking all possible~$\gp\in\Vee{\WGH}$). It follows that our closed complement of~$\goodopen{1}$ is
\begin{align*}
& \cup_{N\in\cN(G)} \supp(\kos[G]{N}) \equalby{\eqref{eq:supp-kos}} \SET{\cP(H,\gp)}{\exists\,N\in\cN(G)\textrm{ such that }H\not\le_G N}
\\
& \quad =\SET{\cP(H,\gp)}{H\not\le \cap_{N\in\cN(G)}N}
 =\SET{\cP(H,\gp)}{H\not\le F} \equalby{\eqref{eq:supp-kos}} \supp(\kos[G]{F}).
\end{align*}
The statement with $\Ker(\Res^G_F)$ then follows from~\eqref{eq:kos-Res}. Finally, if~$G$ is elementary abelian then $F=1$ and~$\Ker(\Res^G_1)=\Kac(G)$ is the tt-ideal of acyclic complexes. The complement of its support is~$\Spc(\cK(G)/\Kac(G))=\Spc(\Db(\kkG))=\Vee{G}$.
\end{proof}

In the above proof, we showed that $\cup_{N\in\cN}\supp(\kos{N})=\supp(\kos{F})$ thanks to the fact that~$\cap_{N\in\cN}N=F$. So the very same argument gives us:
\begin{Cor}
\label{Cor:U(1)}%
Let $G$ be a $p$-group and let $N_1,\ldots,N_r\in\cN(G)$ be some index-$p$ subgroups such that~$N_1\cap \cdots \cap N_r$ is the Frattini subgroup~$F$.
(This can be realized with~$r$ equal to the $p$-rank of~$G/F$.)
Then $\goodopen[G]{1}=\cap_{i=1}^r\open(b_{N_i})$ already.
Hence if $\cP\in\open(b_{N_i})$ for all~$i=1,\ldots,r$ then $\cP\in \open(b_N)$ for all~$N\in\cN(G)$.
\qed
\end{Cor}

Let us turn to the open~$\goodopen[G]{H}$ for the $p$-subgroup at the other end: $H=G$.
\begin{Prop}
\label{Prop:U(G)}%
Let $G$ be a $p$-group. Then the complement of the open $\goodopen[G]{G}$ is the union of the images of the spectra~$\Spc(\cK(H))$ under the maps~$\rho_H=\Spc(\Res_H)$, over all the \emph{proper} subgroups~$H\lneqq G$.
\end{Prop}
\begin{proof}
By \Cref{Lem:cone-a-b}, the closed complement of $\goodopen{G}=\cap_{N\in\cN}\open(a_N)$ equals $\cup_{N\in\cN}\supp(\kk(G/N))$.
For every $H\le G$, we have $\supp(\kk(G/H))=\Img(\rho_H)$; see \Cref{Prop:Spc-Res} if necessary.
This gives the result because restriction to any proper subgroup factors via some index-$p$ subgroup, since $G$ is a $p$-group.
\end{proof}

\begin{Rem}
\label{Rem:geom-open}%
Let $G$ be a $p$-group.
This open complement $\goodopen{G}$ of~$\cup_{H\lneqq G}\Img(\rho_H)$ could be called the `geometric open'.
Indeed, the localization functor
\[
\Phi^G\colon \cK(G)\onto \frac{\cK(G)}{\ideal{\kk(G/H)\mid H\lneqq G}}
\]
corresponding to~$\goodopen{G}$ is analogous to the way the geometric fixed-points functor is constructed in topology. For more on this topic, see \Cref{Rem:geom-fixed-pts}.
\end{Rem}
\begin{Rem}
For $G$ not a $p$-group, the open~$\goodopen{G}$ is not defined (we assume $H\in\Sub{p}(G)$ in \Cref{Prop:U(H)}) and the `geometric open' is void anyway as we have $\Img(\rho_P)=\SpcKG$ for any $p$-Sylow~$P\lneqq G$.
The strategy to analyze non-$p$-groups is to first descend to the $p$-Sylow, using that $\Res_P$ is faithful.
\end{Rem}

\begin{Rem}
We saw in \Cref{Prop:U(G)} that the complement of~$\goodopen{G}$ is covered by the images of the closed maps~$\rho_H=\Spc(\Res_H)$ for~$H\lneqq G$. We could wonder whether another closed map into~$\SpcKG$ covers~$\goodopen{G}$ itself. The answer is the closed immersion $\psi^{F}\colon\Spc(\cK(G/F))\hook\Spc(\cK(G))$ induced by the modular fixed-points functor~$\Psi^F$ with respect to the Frattini subgroup~$F\normal G$.
This can be deduced from the results of \Cref{sec:topology-colim} or verified directly, as we now outline.
Indeed, every prime $\cP=\cP_G(K,\gp)$ for $K\le G$ and~$\gp\in\Vee{\WGK}$ comes by Quillen from some elementary abelian subgroup~$E=H/K\le \WGK=(N_G K)/K$.
One verifies that unless~$N_G K=G$ and~$H=G$, the prime~$\cP$ belongs to the image of~$\rho_{G'}$ for a proper subgroup $G'$ of~$G$. Thus if~$\cP$ belongs to~$\goodopen{G}$, we must have~$E=H/K=G/K$ for $K\normaleq G$. Such a~$K$ must contain the Frattini and the result follows.
\end{Rem}

\section{Twisted cohomology under tt-functors}
\label{sec:H**-tt-functors}%

Still for a general finite group~$G$, we gather some properties of the twisted cohomology ring~$\Rall(G)$ introduced in \Cref{Def:H**}.
We describe its behavior under specific tt-functors, namely restriction, modular fixed-points and localization onto the open subsets~$\goodopen[G]{H}$.
Recall that $\cN=\cN(G)=\SET{N\normal G}{[G\!:\!N]=p}$.

\begin{Rem}
Twisted cohomology~$\Rall(G)$ is graded over a monoid of the form \mbox{$\bbZ\times\bbN^{\ell}$}.
The ring homomorphisms induced by the above tt-functors will be homogeneous \emph{with respect to a certain homomorphism~$\gamma$} on the corresponding grading monoids, meaning of course that the image of a homogeneous element of degree~$(s,q)$ is homogeneous of degree~$\gamma(s,q)$.
The `shift' part (in~$\bbZ$) is rather straightforward.
The `twist' part (in~$\bbN^\ell$) will depend on the effect of said tt-functors on the~$u_N$.
\end{Rem}

Let us start with modular fixed-points, as they are relatively easy.

\begin{Cons}
\label{Cons:H**-Psi^H}%
Let~$H\normaleq G$ be a normal subgroup.
By \Cref{Prop:Psi^H(a/b)}, the tt-functor~$\Psi^H\colon \cK(G)\to \cK(G/H)$ maps every~$u_N$ for $N\not\ge H$ to~$\unit$, whereas it maps $u_N$ for $N\ge H$ to~$u_{N/H}$.
This defines a homomorphism of grading monoids
\begin{equation}
\label{eq:gamma-Psi^H}%
\gamma=\gamma_{\Psi^H}\colon \bbZ\times\bbN^{\cN(G)}\to \bbZ\times\bbN^{\cN(G/H)}
\end{equation}
given by $\gamma(s,q)=(s,\bar{q})$ where $\bar{q}(N/H)=q(N)$ for every $N/H\in \cN(G/H)$. In other words, $q\mapsto \bar{q}$ is simply restriction $\bbN^{\cN(G)}\onto\bbN^{\cN(G/H)}$ along the canonical inclusion~$\cN(G/H)\hook\cN(G)$.
By \Cref{Prop:Psi^H(a/b)}, for every twist $q\in \bbN^{\cN(G)}$, we have a canonical isomorphism~$\Psi^H(\unit(q))\cong\unit(\bar{q})$.
Therefore the modular fixed-points functor $\Psi^H$ defines a ring homomorphism also denoted
\begin{equation}
\label{eq:H**-Psi^H}%
\quad\vcenter{\xymatrix@R=.1em{
{\Psi}^H\colon \kern-2em
& \Rall(G) \ar[r]^-{}
& \Rall(G/H)
\\
& \Big(\unit\xto{f} \unit(q)[s]\Big)
\ar@{|->}[r]
& \Big(\unit \xto{\Psi^H(f)} \Psi^H(\unit(q)[s]) \cong \unit(\bar{q})[s]\Big)
}}
\end{equation}
which is homogeneous with respect to~$\gamma_{\Psi^H}$ in~\eqref{eq:gamma-Psi^H}.
\end{Cons}

Restriction is a little more subtle, as some twists pull-back to non-trivial shifts.

\begin{Cons}
\label{Cons:H**-res}%
Let~$\alpha\colon G'\to G$ be a group homomorphism. Restriction along~$\alpha$ defines a tt-functor $\alpha^*=\Infl^{\Img\alpha}_{G'}\circ\Res^{G}_{\Img\alpha}\colon \cK(G)\to \cK(\Img\alpha)\to\cK(G')$.
Combining \Cref{Prop:Res_H(a/b)} for $\Res^{G}_{\Img\alpha}$ with the obvious behavior of the~$u_{N}$ under inflation (by construction), we see that $\alpha^*(u_N)\cong\unit[2']$ if $N\ge\Img\alpha$ and $\alpha^*(u_N)\cong u_{\alpha\inv(N)}$ if $N\not\ge\Img\alpha$ (which is equivalent to $\alpha\inv(N)\in\cN(G')$).
Hence for every $(s,q)\in \bbZ\times\bbN^{\cN(G)}$ we have a canonical isomorphism $\alpha^*(\unit(q)[s])\cong\unit(q')[s']$ where $s'=s+2'\sum_{N\ge\Img\alpha}q(N)$ and~$q'\colon \cN(G')\to \bbN$ is defined for every $N'\in\cN(G')$ as
\[
q'(N')=\sum_{N\in\cN(G)\textrm{ s.t.\ }\alpha\inv(N)=N'}q(N).
\]
(In particular $q'(N')=0$ if~$N'\not\ge\ker(\alpha)$.)
These formulas define a homomomorphism $(s,q)\mapsto (s',q')$ of abelian monoids that we denote
\begin{equation}
\label{eq:gamma-res}%
\gamma=\gamma_{\alpha^*}\colon \bbZ\times\bbN^{\cN(G)}\to \bbZ\times\bbN^{\cN(G')}.
\end{equation}
The restriction functor~$\alpha^*$ defines a ring homomorphism
\begin{equation}
\label{eq:H**-res}%
\quad\vcenter{\xymatrix@R=.1em{
\alpha^*\colon \kern-2em
&\Rall(G) \ar[r]
& \Rall(G')
\\
& \big(\unit\xto{f} \unit(q)[s]\big)
\ar@{|->}[r]
& \big(\unit \xto{\alpha^*(f)} \alpha^*(\unit(q)[s]) \cong \unit(q')[s']\big)
}}
\end{equation}
which is homogeneous with respect to~$\gamma_{\alpha^*}$ in~\eqref{eq:gamma-res}.
\end{Cons}

\begin{Rem}
\label{Rem:Psi^H-onto}%
For instance, $\alpha\colon G\onto G/H$ can be the quotient by a normal subgroup~$H\normaleq G$. In that case~$\alpha^*$ is inflation, which is a section of modular fixed-points~$\Psi^H$. It follows that the homomorphism ${\Psi}^H$ in~\eqref{eq:H**-Psi^H} is split surjective. (This also means that the composed effect on gradings $\gamma_{\Psi^H}\circ\gamma_{\alpha^*}=\id$ is trivial.)
\end{Rem}

Without changing the group~$G$, we can also localize the twisted cohomology ring $\Rall(G)$ by restricting to an open~$\goodopen{H}$ of~$\SpcKG$, as defined in~\Cref{Prop:U(H)}.
Recall the elements~$a_N,b_N\in\Rall(G)$ from \Cref{Def:u_N}.

\begin{Def}
\label{Def:O(H)}%
Let $H\le G$ be a $p$-subgroup. Let $S_H\subset\Rall(G)$ be the multiplicative subset of the graded ring~$\Rall(G)$ generated by all~$a_N$ such that $H\not\le N$ and all~$b_N$ such that~$H\le N$, for all $N\in\cN(G)$. Recall that the $a_N$ and~$b_N$ are central by \Cref{Rem:graded-commutative}.
We define a $\bbZ$-graded ring
\begin{equation}
\label{eq:O(H)}%
\RGH:=\big(\Rall(G)[S_H\inv]\big)\zerotwist
\end{equation}
as the twist-zero part of the localization of~$\Rall(G)$ with respect to~$S_H$.
Explicitly, the homogeneous elements of~$\RGH$ consist of fractions~$\frac{f}{g}$ where $f,g\in\Rall(G)$ are such that $g\colon \unit\to \unit(q)[t]$ is a product of the chosen~$a_N,b_N$ in~$S_H$, meaning that $\unit(q)[t]$ is the $\otimes$-product of the corresponding~$u_N$ for~$a_N$ and~$u_N[-2']$ for~$b_N$, whereas $f\colon \unit\to \unit(q)[s]$ is any morphism in~$\cK(G)$ with the same $\cN$-twist~$q$ as the denominator.
Thus $\RGH$ is $\bbZ$-graded by the shift only: The degree of~$\frac{f}{g}$ is the difference $s-t$ between the shifts of~$f$ and~$g$.
\end{Def}

\begin{Rem}
\label{Rem:cohomology-generators}%
It follows from \Cref{Lem:finite-generation} (and \Cref{Rem:twists}) that the $\bbZ$-graded ring $\RGH$ is generated as a $\kk$-algebra by the elements
\[
\SET{\zeta^+_N,\,\xi^+_N}{H\le N}\cup\SET{\zeta^-_N,\,\xi^-_N}{H\not\le N}
\]
where $\zeta^+_N=a_N/b_N$ is of degree~$+2'$ and $\zeta^-_N=b_N/a_N$ of degree~$-2'$ as in \Cref{Rem:zeta_N}, and where (only for $p$ odd) the additional elements~$\xi^{\pm}_N$ are $\xi^+_N:=c_N/b_N$ of degree~$+1$, and $\xi^-_N:=c_N/a_N$ of degree~$-1$. (For $p=2$, simply ignore the~$\xi^\pm_N$.)
In general, all these elements satisfy some relations; see~\Cref{Thm:presentation}.
Beware that here $\xi^-_N$ is never the inverse of~$\xi^+_N$. In fact, both are nilpotent.
\end{Rem}

In fact, we can perform the central localization of the whole category~$\cK(G)$
\[
\cL(H)=\cL_G(H):=\cK(G)[S_H\inv]
\]
with respect to the central multiplicative subset~$S_H$ of \Cref{Def:O(H)}.
\begin{Cons}
\label{Cons:L(H)}%
The tt-category~$\cL(H)=\cK(G)[S_H\inv]$ has the same objects as~$\cK(G)$ and morphisms $x\to y$ of the form~$\frac{f}{g}$ where $g\colon \unit\to u$ belongs to~$S_H$, for $u$ a tensor-product of shifts of $u_N$'s according to~$g$ (as in \Cref{Def:O(H)}) and where $f\colon x\to u\otimes y$ is any morphism in~$\cK(G)$ with `same' twist~$u$ as the denominator~$g$.
This category~$\cK(G)[S_H\inv]$ is also the Verdier quotient of~$\cK(G)$ by the tt-ideal~$\ideal{\SET{\cone(g)}{g\in S_H}}$ and the above fraction~$\frac{f}{g}$ corresponds to the Verdier fraction
$
\xymatrix{
x \ar[r]^-{f}
& u \otimes y
& y. \ar[l]_-{g\otimes 1}
}
$
See~\cite[\S\,3]{balmer:sss} or \cite[Theorem~3.6]{dellambrogio-stevenson:even-more-spectra}.

The $\bbZ$-graded endomorphism ring $\End^\sbull_{\cL(H)}(\unit)$ of the unit in~$\cL(H)=\cK(G)[S_H\inv]$ is thus the $\bbZ$-graded ring $(S_H\inv\Rall(G))\zerotwist=\RGH$ of \Cref{Def:O(H)}.

There is a general localization $\cK\restr{U}$ of a tt-category~$\cK$ over a quasi-compact open $U\subseteq\SpcK$ with closed complement~$Z$. It is defined as~$\cK\restr{U}=(\cK/\cK_Z)^\natural$.
If we apply this to $U=\goodopen{H}$, we deduce from~\eqref{eq:U(H)-def} that $U=\cap_{g\in S_H}\open(g)$ has closed complement $Z=\cup_{g\in S_H}\supp(\cone(g))$ whose tt-ideal $\cK(G)_Z$ is the above~$\ideal{\SET{\cone(g)}{g\in S_H}}$.
In other words, the idempotent-completion of our $\cL_G(H)=\cK(G)[S_H\inv]$ is exactly $\cK(G)\restr{\goodopen{H}}$.
As with any localization, we know that $\Spc(\cL_G(H))$ is a subspace of~$\SpcKG$, given here by $U=\cap_{g\in S_H}\open(g)=\goodopen{H}$.
\end{Cons}

\begin{Exa}
\label{Exa:R(1)=H*}%
For $G=E$ elementary abelian and the subgroup~$H=1$, the category $\cL_E(1)=\cK(E)\restr{\goodopen{1}}$ in \Cref{Cons:L(H)} is simply the derived category $\cL_E(1)=\Db(E)$, by \Cref{Prop:U(1)}.
In that case, $\Rloc{E}(1)\cong\rmH^\sbull(E;\kk)$ is the actual cohomology ring of~$E$.
Since $H=1\le N$ for all~$N$, we are inverting all the~$b_N$ and no~$a_N$.
As noted in \Cref{Cor:U(1)}, we obtain the same ring (the cohomology of~$E$) as soon as we invert enough $b_{N_1},\ldots,b_{N_r}$, namely, as soon as $N_1\cap \cdots \cap N_r=1$.
\end{Exa}

We again obtain an induced homomorphism of multi-graded rings.

\begin{Cons}
\label{Cons:H**-loc}%
Let $H\le G$ be a $p$-subgroup and consider the above central localization $(-)\restr{\goodopen{H}}\colon \cK(G)\onto \cL_G(H)$.
As explained in \Cref{Rem:loc-triv}, the morphisms~$a_N$ and~$b_N$ give us explicit isomorphisms~$(u_N)\restr{\goodopen{H}}\cong\unit$ if~$N\not\ge H$ and $(u_N)\restr{\goodopen{H}}\cong\unit[2']$ if~$N\ge H$.
This yields a homomorphism on the grading
\begin{equation}
\label{eq:gamma-loc}%
\gamma=\gamma_{\goodopen{H}}\colon \bbZ\times\bbN^{\cN(G)}\to \bbZ
\end{equation}
defined by $\gamma(s,q)=s+2'\sum_{N\ge H}q(N)$ and we obtain a ring homomorphism
\begin{equation}
\label{eq:H**-loc}%
(-)\restr{\goodopen{H}}\colon \Rall(G) \too \End^\sbull_{\cL_G(H)}(\unit)=\RGH
\end{equation}
which is homogeneous with respect to the homomorphism~$\gamma_{\goodopen{H}}$ of~\eqref{eq:gamma-loc}.
\end{Cons}

\begin{Rem}
\label{Rem:comp-nat}%
It is easy to verify that the continuous maps induced on homogeneous spectra by the ring homomorphisms constructed above are compatible with the comparison map of \Cref{Prop:comp}.
In other words, if $F\colon \cK(G)\to \cK(G')$ is a tt-functor and if the induced homomorphism~$F\colon \Rall(G)\to \Rall(G')$ is homogenous with respect to $\gamma=\gamma_{F}\colon \bbZ\times\bbN^{\cN(G)}\to\bbZ\times\bbN^{\cN(G')}$, for instance $F=\Psi^H$ or $F=\alpha^*$ as in \Cref{Cons:H**-Psi^H,Cons:H**-res}, then the following square commutes:
\begin{equation}
\label{eq:comp-nat}%
\vcenter{
\xymatrix@C=4em{
\Spc(\cK(G')) \ \ar[r]^-{\Spc(F)} \ar[d]^-{\comp_{G'}}
& \Spc(\cK(G)) \ar[d]^-{\comp_{G}}
\\
\Spech(\Rall(G')) \ \ar@{->}[r]^-{\Spech({F})}
& \Spech(\Rall(G)).
}}
\end{equation}
This follows from $F(\cone(f))\simeq\cone({F}(f))$ in~$\cK(G')$ for any~$f\in\Rall(G)$.
\end{Rem}

\begin{Rem}
\label{Rem:comp-loc}%
Similarly, for every $H\in\Sub{p}(G)$ the following square commutes
\begin{equation}
\label{eq:comp-loc}%
\vcenter{
\xymatrix{
\kern-2em \goodopen[G]{H}=\Spc(\cL_G(H)) \ \ar@{^(->}[r] \ar[d]_-{\comp_{\cL(H)}}
& \Spc(\cK(G)) \ar[d]^-{\comp_{G}}
\\
\Spech(\RGH) \ \ar@{^(->}[r]
& \Spech(\Rall(G))
}}
\end{equation}
where the left-hand vertical map is the classical comparison map of~\cite{balmer:sss} for the tt-category~$\cL_G(H)$ and the $\otimes$-invertible~$\unit[1]$.
The horizontal inclusions are the ones corresponding to the localizations with respect to~$S_H$, as in~\Cref{Cons:L(H),Cons:H**-loc}.
In fact, it is easy to verify that the square~\eqref{eq:comp-loc} is cartesian, in view of $\goodopen[G]{H}=\bigcap_{g\in S_H}\open(g)=\bigcap_{g\in S_H}\comp_G\inv(Z(g)^c)$ by \Cref{Cons:L(H)} and~\eqref{eq:comp-closed}.
\end{Rem}

We can combine the above functors. Here is a useful example.
\begin{Prop}
\label{Prop:H**-compose}%
Let $H\normaleq G$ be a normal subgroup such that~$G/H$ is elementary abelian.
Then we have a commutative square
\begin{equation}
\label{eq:aux-abelem-2}%
\vcenter{
\xymatrix@C=1.5em{
\kern-2em\Vee{G/H}=\Spc(\Db(\kk(G/H))) \ \ar[r]^-{\check{\psi}^H} \ar[d]_-{\comp_{\Db(\kk(G/H))}}^-{\simeq}
& \Spc(\cK(G)) \ar[d]^-{\comp_{G}}
\\
\Spech(\rmH^\sbull(G/H,k)) \ \ar@{^(->}[r]
& \Spech(\Rall(G))
}}
\end{equation}
and in particular, its diagonal $\comp_{G}\circ\check{\psi}^{H}$ is injective.
\end{Prop}

\begin{proof}
The functor~$\check{\Psi}^H\colon \cK(G)\to \Db(\kk(G/H))$ is the modular fixed-points functor $\Psi^H\colon \cK(G)\to \cK(G/H)$ composed with $\Upsilon_{G/H}\colon\cK(G/H)\onto \Db(\kk(G/H))$, which is the central localization $(-)\restr{\goodopen{1}}$ over the cohomological open, by \Cref{Prop:U(1)}; see~\Cref{Exa:R(1)=H*}.
Thus we obtain two commutative squares~\eqref{eq:comp-loc} and~\eqref{eq:comp-nat}:
\begin{equation*}
\label{eq:comp-nat-G}%
\vcenter{
\xymatrix@C=1.5em{
\Spc(\Db(\kk(G/H))) \ \ar@{^(->}[r]^-{\upsilon_{G/H}} \ar[d]_-{\comp_{\Db(\kk(G/H))}}^-{\simeq}
& \Spc(\cK(G/H)) \ \ar[d]^-{\comp_{G/H}} \ar@{->}[r]^(.5){\psi^H}
& \Spc(\cK(G)) \ar[d]^-{\comp_{G}}
\\
\Spech(\rmH^\sbull(G/H,k)) \ \ar@{^(->}[r]^-{}
& \Spech(\Rall(G/H)) \ \ar@{^(->}[r]^-{}
& \Spech(\Rall(G))
}}
\end{equation*}
the left-hand one for the central localization of~$\cK(G/H)$ over the open~$\goodopen[G/H]{1}=\Vee{G/H}$, and the right-hand one for the tt-functor~$\Psi^H\colon \cK(G)\to \cK(G/H)$.
Note that the bottom-right map is injective because the ring homomorphism in question, ${\Psi}^H\colon \Rall(G)\to \Rall(G/H)$ defined in~\eqref{eq:H**-Psi^H}, is surjective by \Cref{Rem:Psi^H-onto}.
\end{proof}

\section{The elementary abelian case}
\label{sec:abelem}%

In this central section, we apply the general constructions of \Cref{sec:invertibles+H**,sec:goodopen,sec:H**-tt-functors} in the case of~$G=E$ elementary abelian.
We start with a key fact that is obviously wrong in general (\eg for a non-cyclic simple group, the target space is just a point).

\begin{Prop}
\label{Prop:abelem-inj}%
Let $E$ be an elementary abelian group.
The comparison map
\[
\comp_E\colon \Spc(\cK(E))\to \Spech(\Rall(E))
\]
of \Cref{Prop:comp} is injective.
\end{Prop}
\begin{proof}
Let $H,N\le E$ with $[E\!:\!N]=p$. Suppose first that $H\not\le N$.
We use the map~$\check\psi^H=\Spc(\check\Psi^H)\colon \Vee{E/H}\to \SpcKE$ of \Cref{Rec:Part-I}.
Then
\[
\begin{array}{rll}
(\check{\psi}^H)\inv(\open(b_N)) & = (\check{\psi}^H)\inv(\open(\cone(b_N))) & \textrm{by definition, see~\eqref{eq:open(f)}}
\\
& = \open(\cone(\check{\Psi}^H(b_N))) & \textrm{by general tt-geometry}
\\
& = \open (\cone(0\colon \unit\to \unit)) & \textrm{by \Cref{Prop:Psi^H(a/b)}}
\\
& = \open (\unit\oplus \unit[1]) = \varnothing. &
\end{array}
\]
Thus $\Img(\check{\psi}^H)$ does not meet~$\open(b_N)$ when $H\not\le N$.
Suppose now that $H\le N$. A similar computation as above shows that $(\check{\psi}^H)\inv(\open(b_N))=\Spc(\Db(\kk(E/H)))$ since in that case $\check{\Psi}^H(b_N)$ is an isomorphism in~$\Db(\kk(E/H))$. Therefore $\Img(\check{\psi}^H)\subseteq\open(b_N)$ when $H\le N$. Combining both observations, we have
\begin{equation}
\label{eq:aux-abelem-1}%
\Img(\check{\psi}^H)\cap \open(b_N)\neq\varnothing \quad \Longleftrightarrow \quad H\le N.
\end{equation}
Let now $\cP,\cQ\in\Spc(\cK(E))$ be such that $\comp_E(\cP)=\comp_E(\cQ)$ in~$\Spech(\Rall(E))$.
Say $\cP=\cP_E(H,\gp)$ and $\cQ=\cP_E(K,\gq)$ for $H,K\le E$ and $\gp\in \Vee{E/H}$ and~$\gq\in\Vee{E/K}$.
(See \Cref{Rec:Part-I}.)
The assumption $\comp_E(\cP)=\comp_E(\cQ)$ implies  that $\cP\in\open(f)$ if and only if~$\cQ\in\open(f)$, for every $f\in\Rall(E)$. In particular applying this to~$f=b_N$, we see that for every index-$p$ subgroup~$N\normal E$ we have $\cP\in\open(b_N)$ if and only if $\cQ\in\open(b_N)$.
By \eqref{eq:aux-abelem-1}, we have for every $N\in\cN(E)$
\[
H\le N \quad \Longleftrightarrow \quad K\le N.
\]
Since $E$ is elementary abelian, this forces $H=K$. So we have two points~$\gp,\gq\in \Vee{E/H}$ that go to the same image under $\Vee{E/H}\xto{\check\psi^H}\Spc(\cK(E))\xto{\comp_E}\Spech(\Rall(E))$ but we know that this map in injective by \Cref{Prop:H**-compose} for $G=E$.
\end{proof}

In fact, we see that the open~$\goodopen{H}$ of~$\SpcKE$ defined in \Cref{Prop:U(H)} matches perfectly the open~$\Spech(\REH)$ of~$\Spech(\Rall(E))$ in \Cref{Def:O(H)}.
\begin{Thm}
\label{Thm:abelem-top}%
Let $E$ be an elementary abelian $p$-group. Let $H\le E$ be a subgroup. Then the comparison map of \Cref{Prop:comp} restricts to a homeomorphism
\[
\comp_E\colon \goodopen{H}\isoto \Spech(\REH)
\]
where $\REH$ is the $\bbZ$-graded endomorphism ring of the unit~$\unit$ in the localization~$\cL_E(H)$ of~$\cK(E)$ over the open~$\goodopen{H}$.
\end{Thm}
\begin{proof}
Recall the tt-category $\cL(H)=\cL_E(H):=\cK(E)[S_H\inv]$ of \Cref{Cons:L(H)}, where $S_H\subset\Rall(E)$ is the multiplicative subset generated by the homogeneous elements $\SET{a_N}{H\not\le N}\cup\SET{b_N}{H\le N}$ of \Cref{Def:O(H)}.
In view of \Cref{Rem:comp-loc}, it suffices to show that the map $\comp_{\cL(H)}\colon \Spc(\cL(H))\to \Spech(\REH)$ is a homeomorphism.
We have injectivity by \Cref{Prop:abelem-inj}. We also know that $\REH$ is noetherian by \Cref{Thm:Rall-noeth}. It follows from~\cite{balmer:sss} that $\comp_{\cL(H)}$ is surjective. Hence it is a continuous bijection and we only need to prove that it is a closed map.

We claim that $\cL(H)$ is generated by its $\otimes$-unit~$\unit$. Namely, let $\cJ=\thick_{\cL(H)}(\unit)$ be the thick subcategory of~$\cL(H)$ generated by~$\unit$ and let us see that~$\cJ=\cL(H)$.
Observe that $\cJ$ is a sub-tt-category of~$\cL(H)$.
Let $N\in\cN$ be an index-$p$ subgroup. We claim that $\kk(E/N)$ belongs to~$\cJ$.
If $N\not\ge H$, then $a_N$ is inverted in~$\cL(H)$, so $\kk(E/N)=0$ in~$\cL(H)$ by \Cref{Lem:cone-a-b}\,\ref{it:cone-a}.
If $N\ge H$, then $b_N\colon \unit\to u_N[-2']$ is inverted, so $u_N\in\cJ$ and we conclude again by \Cref{Lem:cone-a-b}\,\ref{it:cone-a} since $a_N\colon \unit\to u_N$ is now a morphism in~$\cJ$. For a general proper subgroup~$K<E$, the module $\kk(E/K)$ is a tensor product of~$\kk(E/N)$ for some $N\in\cN$. (Here we use~$E$ elementary abelian again.) Hence $\kk(E/K)$ also belongs to~$\cJ$ as the latter is a sub-tt-category of~$\cL(H)$.
In short $\cJ$ contains all generators~$\kk(E/H)$ for~$H\le E$. Therefore $\cL(H)=\cJ$ is indeed generated by its unit.
It follows from this and from noetherianity of~$\End^\sbull_{\cL(H)}(\unit)=\REH$ that $\Hom^\sbull_{\cL(H)}(x,y)$ is a finitely generated $\REH$-module for every $x,y\in \cL(H)$.
We conclude from a general tt-geometric fact, observed by Lau~\cite[Proposition~2.7]{lau:spc-dm-stacks}, that the map~$\comp$ must then be closed.
\end{proof}

\begin{Cor}
\label{Cor:abelem-Dirac}%
Let $E$ be an elementary abelian $p$-group.
Let $\Rloc{E}$ be the sheaf of $\bbZ$-graded rings on~$\SpcKE$ obtained by sheafifying $U\mapsto \End^\sbull_{\cK(E)\restr{U}}(\unit)$.
Then $(\SpcKE,\Rloc{E})$ is a Dirac scheme in the sense of~\cite{hesselholt-pstragowski:dirac1}.
\end{Cor}
\begin{proof}
We identified an affine cover $\{\goodopen{H}\}_{H\le E}$ in \Cref{Thm:abelem-top}.
\end{proof}

\begin{Rem}
\label{Rem:O(E)}%
This result further justifies the notation for the ring~$\Rloc{E}(H)$ in \Cref{Def:O(H)}.
Indeed, this~$\RE(H)$ is also the ring of sections~$\Rloc{E}(\goodopen{H})$ of the $\bbZ$-graded structure sheaf~$\RE$ over the open~$\goodopen{H}$ of \Cref{Prop:U(H)}.
\end{Rem}

\begin{Cor}
\label{Cor:abelem-top}%
Let $E$ be an elementary abelian $p$-group.
Then the comparison map of \Cref{Prop:comp} is an open immersion. More precisely, it defines a homeomorphism between $\Spc(\cK(E))$ and the following open subspace of~$\Spech(\Rall(E))$:
\begin{equation}
\label{eq:abelem-top}%
\SET{\gp\in\Spech(\Rall(E))}{\textrm{for all $N\normal E$ of index~$p$ either $a_N\notin\gp$ or $b_N\notin\gp$}}.\kern-1em
\end{equation}
\end{Cor}
\begin{proof}
By \Cref{Prop:abelem-inj}, the (continuous) comparison map is injective.
Therefore, it being an open immersion can be checked locally on the domain.
By \Cref{Prop:U(H)}, the open~$\goodopen{H}$ form an open cover of~$\Spc(\cK(E))$. \Cref{Thm:abelem-top} tells us that each $\goodopen{H}$ is homeomorphic to the following open of~$\Spech(\Rall(E))$
\[
U'(H):=\bigcap_{N\not\ge H}Z(a_N)^c\cap\bigcap_{N\ge H}Z(b_N)^c
\]
(recall that $Z(f)^c=\SET{\gp}{f\notin\gp}$ is our notation for a principal open).
So it suffices to verify that the union $\cup_{H\le E}\,U'(H)$, is the open subspace of the statement~\eqref{eq:abelem-top}.
Let $\gp\in U'(H)$ for some~$H\le E$ and let $N\in\cN(E)$; then clearly either $N\not\ge H$ in which case~$a_N\notin\gp$, or $N\ge H$ in which case~$b_N\notin\gp$.
Conversely let~$\gp$ belong to the open~\eqref{eq:abelem-top} and define $H=\cap_{M\in\cN\textrm{\,s.t.\,}b_M\notin\gp}M$.
We claim that $\gp\in U'(H)$. Let $N\in\cN$. If $N\not\ge H$ then $b_{N}\in\gp$ by construction of~$H$ and therefore $a_{N}\notin\gp$.
So the last thing we need to prove is that $N\ge H$ implies $b_{N}\notin\gp$. One should be slightly careful here, as $H$ was defined as the intersection of the $M\in\cN$ such that $b_M\notin\gp$, and certainly such $M$'s will contain~$H$, but we need to see why \emph{every} $N\ge H$ satisfies $b_N\notin\gp$.
This last fact follows from \Cref{Cor:U(1)} applied to~$E/H$.
\end{proof}

\begin{Exa}
\label{Exa:C_p}%
Consider the spectrum of $\cK(C_p)$ for the cyclic group~$C_p$ of order~$p$.
By \Cref{Exa:Rall(C_p)}, the reduced ring $\Rloc{C_p}(1)_{\red}$ is~$k[\zeta^+]$ with $\zeta^+=a/b$ in degree~$2'$ while $\Rloc{C_p}(C_p)_{\red}=k[\zeta^-]$ with $\zeta^-=b/a$.
(The former is also \Cref{Exa:R(1)=H*}.)
Each of these has homogeneous spectrum the Sierpi\'{n}ski space and we easily deduce that
\begin{equation}
\label{eq:C_p}
\Spc(\cK(C_p))=\qquad\vcenter{\xymatrix@R=1em@C=.5em{
& {\color{Brown}\bullet} \ar@{-}@[Brown][rd]_-{{\color{Brown}\goodopen{C_p}}}
&&
{\color{ForestGreen}\bullet}
\\
&& \scalebox{.6}{{\color{Brown}\LEFTCIRCLE}\kern-0.78em{\color{ForestGreen}\RIGHTCIRCLE}} \ar@{-}@[ForestGreen][ru]_-{{\color{ForestGreen}\goodopen{1}}}
}}
\end{equation}
confirming the computation of~$\Spc(\cK(C_{p^n}))$ in \Cref{Prop:cyclic} for $n=1$.

We can also view this as an instance of \Cref{Cor:abelem-top}.
Namely, still by \Cref{Exa:Rall(C_p)}, the reduced ring $\Rall(C_p)_{\red}$ is~$\kk[a,b]$ with~$a$ in degree~$0$ and~$b$ in degree~$-2'$.
Its homogeneous spectrum has one more point at the top:
\[
\Spech(\Rall(C_p))=\qquad\vcenter{
\xymatrix@R=1em@C=.5em@R=.5em{
  &&\bullet \ar@{-}[ld] \ar@{-}[rd]\\
&  {\color{Brown}\bullet} \ar@{-}@[Brown][rd]_-{{\color{Brown}Z(a)^c}}
&&
{\color{ForestGreen}\bullet}
\\
&& \scalebox{.6}{{\color{Brown}\LEFTCIRCLE}\kern-0.78em{\color{ForestGreen}\RIGHTCIRCLE}} \ar@{-}@[ForestGreen][ru]_-{{\color{ForestGreen}Z(b)^c}}
}}
\]
and this superfluous closed point~$\langle a,b\rangle$ lies outside of the open subspace~\eqref{eq:abelem-top}.
\end{Exa}

\begin{Rem}
\label{Rem:Psi-on-O(H)}%
Let $K\le H\le E$.
The functor $\Psi^K\colon\cK(E)\to\cK(E/K)$ passes, by \Cref{Prop:Psi^H(a/b)}, to the localizations over $\goodopen[E]{H}$ and $\goodopen[E/K]{H/K}$, respectively.
On the $\ZZ$-graded endomorphisms rings, we get a homomorphism $\Psi^K\colon\REH\to\Rloc{E/K}(H/K)$ that on generators~$a_N,b_N$ is given by the formulas of \Cref{Prop:Psi^H(a/b)}.
By \Cref{Rem:Psi^H-onto} this homomorphism $\Psi^K\colon\REH\to \Rloc{E/K}(H/K)$ is surjective.
\end{Rem}
\begin{Prop}
\label{Prop:SpcKE-irreducible}%
For every elementary abelian group~$E$, the spectrum~$\SpcKE$ admits a unique generic point~$\eta_E$, namely the one of the cohomological open~$\Vee{E}$.
\end{Prop}
\begin{proof}
We proceed by induction on the $p$-rank.
Let us write $\eta_E=\cP_E(1,\sqrt{0})$ for the generic point of~$\Vee{E}$, corresponding to the ideal $\sqrt{0}$ of nilpotent elements in~$\rmH^\sbull(E;\kk)$.
Similarly, for every~$K\le E$, let us write $\eta_E(K)=\cP_E(K,\eta_{E/K})$ for the generic point of the stratum~$\Vee{E}(K)\simeq\Vee{E/K}$.
We need to prove that every point~$\eta_E(K)$ belongs to the closure of~$\eta_E=\eta_E(1)$ in~$\SpcKE$.
It suffices to show this for every cyclic subgroup $H<E$, by an easy induction argument on the rank, using the fact that $\psi^H\colon\Spc(\cK(E/H))\hook\SpcKE$ is closed.
So let $H\le E$ be cyclic.

Note that inflation $\Infl^{E/H}_E\colon \cK(E/H)\to \cK(E)$ passes to the localization of the former with respect to all~$b_{N/H}$ for all~$N\in\cN(E)$ containing~$H$ (which is just the derived category of~$E/H$) and of the latter with respect to the corresponding~$b_N$:
\begin{equation}
\label{eq:aux-Infl-1}%
\Infl^{E/H}_E \colon \Db(\kk(E/H))
\too \cK(E)\big[\SET{b_{N}}{N\ge H}\inv\big].
\end{equation}
This being a central localization of a fully-faithful functor with respect to a multiplicative subset in the source, it remains fully-faithful. One can further localize both categories with respect to all non-nilpotent~$f\in\rmH^\sbull(E/H;\kk)$ in the source, to obtain a fully-faithful
\begin{equation}
\label{eq:aux-Infl-2}%
\Infl^{E/H}_E \colon \Db(\kk(E/H))[\SET{f}{f\not\in\sqrt{0}}\inv]
\too \cL
\end{equation}
where $\cL$ is obtained from $\cK(E)$ by first inverting all~$b_N$ for~$N\ge H$ as in~\eqref{eq:aux-Infl-1} and then inverting all~$\Infl^{E/H}_E(f)$ for $f\in\rmH^\sbull(E/H;\kk)\sminus\sqrt{0}$.

At the level of spectra, $\Spc(\cL)$ is a subspace of~$\SpcKE$. By construction, it meets the closed subset $\Img(\psi^H)\cong\Spc(\cK(E/H))$ of~$\SpcKE$ only at the image of the generic point~$\eta_E(H)$. Indeed, inverting all~$b_N$ for~$N\ge H$ on~$\Img(\psi^{H})$ corresponds to inverting all~$b_{N/H}$ in~$\cK(E/H)$, hence shows that $\Spc(\cL)\cap \Img(\psi^H)$ is in the image under~$\psi^H$ of the cohomological open~$\Vee{E/H}$. Similarly, inverting all~$f\notin\sqrt{0}$ removes all non-generic points of~$\Vee{E/H}$. In particular, the generic point $\eta_E(H)$ of~$\Vee{E}(H)$ is now a closed point of the subspace~$\Spc(\cL)$ of~$\SpcKE$.

Using that~\eqref{eq:aux-Infl-2} is fully-faithful and that the endomorphism ring of the source is the cohomology of~$E/H$ localized at its generic point (in particular not a product of two rings), we see that $\cL$ is not a product of two tt-categories and therefore $\Spc(\cL)$ is not disconnected.
Also $\eta_E$ belongs to $\Spc(\cL)$ and is distinct from~$\eta_E(H)$.
Hence the closed point~$\eta_E(H)\in\Spc(\cL)$ cannot be isolated.
Thus $\eta_E(H)$ belongs to the closure of some other point in~$\Spc(\cL)$.

Let then~$\cQ\in\SpcKE$ be a point in the subspace~$\Spc(\cL)$, such that $\cQ\neq\eta_E(H)$ and $\eta_E(H)\in\adhpt{\cQ}$, which reads~$\cQ\subsetneq\eta_E(H)$.
We know by \Cref{Cor:P<P'} that this can only occur for~$\cQ=\cP(H',\gp)$ with~$H'\le H$, that is, either $H'=H$ or~$H'=1$ since here~$H$ was taken cyclic.
The case~$H'=H$ is excluded, as in the subspace $\Spc(\cL)$ the only prime of the form~$\cP(H,\gq)$ that remained was $\eta_E(H)$ itself, and~$\cQ$ is different from~$\eta_E(H)$. Thus $H'=1$, which means that $\cQ\in\Vee{E}=\adhpt{\eta_E(1)}$ and we therefore have $\eta_E(H)\in\adhpt{\cQ}\subseteq \adhpt{\eta_E(1)}$ as claimed.
\end{proof}

We can now determine the Krull dimension of the spectrum of~$\cK(E)$.
\begin{Prop}
\label{Prop:SpcKE-dim}
Let $E$ be a elementary abelian $p$-group. Then the Krull dimension of $\SpcKE$ is the $p$-rank of~$E$.
\end{Prop}
\begin{proof}
By \Cref{Prop:U(H)}, the dimension of $\SpcKE$ is the maximum of the dimensions of the open subsets~$\goodopen{H}$, for~$H\le E$.
Each of these spaces has the same generic point~$\eta_E$ (by \Cref{Prop:SpcKE-irreducible}) and a unique closed point~$\cM(H)$ by \Cref{Prop:U(H)}
(and the fact that $\cM(K)\in\goodopen{H}$ forces $K$ and~$H$ to be contained in the same subgroups $N\in\cN(G)$ by \Cref{Prop:Psi^H(a/b)}, which in turn forces $K=H$ because $E$ is elementary abelian).
Using \Cref{Thm:abelem-top} we translate the problem into one about the graded ring~$\REH$.
Let $\eta_E=\gp_0\subsetneq\gp_1\subsetneq\cdots\subsetneq\gp_n=\cM(H)$ be a chain of homogeneous prime ideals in~$\REH$.
Note that $\gp_{n-1}$ belongs to the open $Z(f)^c$ of $\Spech(\REH)$ for some~$f=\zeta^+_N$, $H\le N$, or some $f=\zeta^-_N$, $H\not\le N$.
Each of these has non-zero degree so the graded ring $\REH[f\inv]$ is periodic.
We deduce that $\dim(\Spech(\REH))$ is the maximum of $1+\dim(R)$ where $R$ ranges over the ungraded rings $R=\REH[f\inv]_{(0)}$ for~$f$ as above.
The reduced ring~$R_{\textup{red}}$ is a finitely generated $\kk$-algebra with irreducible spectrum, hence a domain. Therefore $\dim(R)=\dim(R_{\textup{red}})$ is the transcendence degree of the residue field at the unique generic point.
As observed above, this generic point is the same for all $H\le E$, namely the generic point of $\goodopen{1}=\Spc(\Db(\kk E))$.
We conclude that $\dim(\SpcKE)=\dim(\Spc(\Db(\kk E)))$ which is indeed the $p$-rank of~$E$.
\end{proof}

\begin{Rem}
In fact, the proof shows that all closed points~$\cM(H)\in\SpcKE$ have the same codimension (height), namely the $p$-rank of~$E$.
\end{Rem}
\begin{Rem}
Thus for~$E$ elementary abelian, the Krull dimension of~$\SpcKE$ is the same as the Krull dimension of the classical cohomological open~$\Spc(\Db(\kk E))\cong\Spech(\rmH^\sbull(E,\kk))$.
In other words, the spectrum of~$\cK(E)$ is not monstrously different from that of~$\Db(\kk E)$, at least in terms of dimension, or `vertical complexity'.
There is however `horizontal complexity' in~$\SpcKE$: each $\goodopen{H}$ has its own shape and form, and there are as many~$\goodopen{H}$ as there are subgroups~$H\le E$.
We give a finite presentation of the corresponding $\kk$-algebras~$\REH$ in \Cref{sec:presentation-H**}.
\end{Rem}

\goodbreak
\section{Closure in elementary abelian case}
\label{sec:closure}%

In this section,~$E$ is still an elementary abelian $p$-group. Following up on \Cref{Rem:closure}, we can now use \Cref{Thm:abelem-top} to analyze inclusion of tt-primes~$\cP,\cQ$ in~$\cK(E)$, which amounts to asking when~$\cQ$ belongs to~$\adhpt{\cP}$ in~$\SpcKE$.

\begin{Rem}
\label{Rem:closure-E}%
Using again that every $\psi^H\colon \Spc(\cK(E/H))\hook\SpcKE$ is a closed immersion, induction on the $p$-rank easily reduces the above type of questions to the case where the `lower' point~$\cP$ belongs to~$\goodopen[E]{1}=\Vee{E}$.
More generally, given a closed piece $Z$ of the cohomological open~$\Vee{E}$, we consider its closure $\bar{Z}$ in~$\SpcKE=\sqcup_{H\le E}\Vee{E}(H)$ and we want to describe the part $\bar{Z}\cap \Vee{E}(H)$ in each stratum~$\Vee{E}(H)\cong\Vee{E/H}$ for~$H\le E$.
\end{Rem}

\begin{Cons}
\label{Cons:inclusions}%
Let $H\le E$ be a subgroup of our elementary abelian group~$E$. Consider the open subsets~$\goodopen[E]{H}$ of \Cref{Prop:U(H)}, the cohomological open $\goodopen[E]{1}=\Vee{E}$ and their intersection $\goodopen[E]{H}\cap \Vee{E}$. Consider also the stratum~$\Vee{E}(H)=\check\psi^H(\Vee{E/H})$, that is a closed subset of~$\goodopen[E]{H}$ homeomorphic to~$\Vee{E/H}$ via~$\check\psi^H$:
\[
\vcenter{\xymatrix@C=2em{
& \goodopen[E]{H}
&& \Vee{E}
\\
\Vee{E/H} \ar@{^(->}[ru]_-{\check\psi^H}
&& \goodopen[E]{H}\cap \Vee{E} \ar@{_(->}[lu] \ar@{^(->}[ru]
}}
\]
On graded endomorphism rings of the unit (\Cref{Def:O(H)}) this corresponds to
\begin{equation}
\label{eq:R(H)-span}%
\vcenter{\xymatrix@C=1em{
& \REH \ar@{->>}[ld]_-{\Psi^H} \ar[rd]_(.4){Q}
&& \rmH^\sbull(E;\kk) \ar[ld]^(.4){Q'}
\\
\rmH^\sbull(E/H;\kk)
&& \kern-2em \RE(\goodopen[E]{H}\cap \Vee{E})\kern-2em
}}
\end{equation}
where $Q$ is the localization of~$\REH$ with respect to~$\zeta^-_N=\frac{b_N}{a_N}$ for all~$N\not\ge H$, where $Q'$ is the localization of~$\Rloc{E}(1)=\rmH^\sbull(E,\kk)$ with respect to~$\zeta^+_N=\frac{a_N}{b_N}$ for all~$N\not\ge H$ and where $\Psi^H$ is the epimorphism of~\Cref{Rem:Psi-on-O(H)} for~$K=H$.
\end{Cons}

\begin{Lem}
\label{Lem:inclusions}%
With above notation, let~$I\subseteq \rmH^\sbull(E,\kk)$ be a homogeneous ideal of the cohomology of~$E$.
Define the homogeneous ideal~$J=\Psi^H(Q\inv(\ideal{Q'(I)}))$ in the cohomology~$\rmH^\sbull(E/H;\kk)$ of~$E/H$ by `carrying around' the ideal~$I$ along~\eqref{eq:R(H)-span}:
\[
\quad\xymatrix@C=1em{
& \kern-1em Q\inv(\ideal{Q'(I)}) \kern-1em \ar@{|->}[ld]
&& \quad I \quad \ar@{|->}[ld]
\\
\kern-2em J:=\Psi^H(Q\inv(\ideal{Q'(I)})) \kern-4em
&& \ \ideal{Q'(I)} \ \ar@{|->}[lu]
}
\]
Let $Z$ be the closed subset of~$\Vee{E}$ defined by the ideal~$I$.
Then the closed subset of~$\Vee{E/H}$ defined by~$J$ is exactly the intersection $\bar{Z}\cap \Vee{E}(H)$ of the closure $\bar{Z}$ of~$Z$ in~$\Spc(\cK(E))$ with the subspace~$\Vee{E/H}$, embedded via $\check{\psi}^H$.
\end{Lem}
\begin{proof}
Once translated by \Cref{Thm:abelem-top}, it is a general result about the multi-graded ring $A=\Rall(E)$. We have two open subsets, $\goodopen{H}=\cap_{s\in S_H}Z(s)^c$ and~$\Vee{E}=\goodopen{1}=\cap_{s\in S_1}Z(s)^c$ for the multiplicative subsets $S_H$ and~$S_1$ of \Cref{Def:O(H)}.
These open subsets are `Dirac-affine', meaning they correspond to the homogenous spectra of the $\bbZ$-graded localizations~$S_H\inv(A)\zerotwist=\REH$ and~$S_1\inv (A)\zerotwist=\Rloc{E}(1)=\rmH^\sbull(E;\kk)$, where~$(-)\zerotwist$ refers to~`zero-twist', as before. The intersection of those two affine opens corresponds to inverting both~$S_H$ and~$S_1$, that is, inverting~$\SET{\frac{b_N}{a_N}}{N\not\ge H}$ from~$\REH$ and~$\SET{\frac{a_N}{b_N}}{N\not\ge H}$ from~$\rmH^\sbull(E;\kk)$.
This explains the two localizations~$Q$ and~$Q'$ and why their targets coincide.

The intersection $\goodopen{H}\cap\bar{Z}$ coincides with the closure in~$\goodopen{H}$ of $\goodopen{H}\cap Z$.
The latter is a closed subset of~$\goodopen{H}\cap\Vee{E}$ defined by the ideal~$\ideal{Q'(I)}$. The preimage ideal~$Q\inv(\ideal{Q'(I)})$ then defines that closure~$\goodopen{H}\cap \bar{Z}$ in~$\goodopen{H}$.
Finally, to further intersect this closed subset of~$\goodopen{H}$ with the closed subset~$\Vee{E/H}=\Img(\Spech(\Psi^H))$, it suffices to project the defining ideal along the corresponding epimorphism~$\Psi^H\colon \REH\onto \rmH^\sbull(E/H;\kk)$.
\end{proof}

Before illustrating this method, we need a technical detour via polynomials.
\begin{Lem}
\label{Lem:japan}%
Let $I$ be a homogeneous ideal of the cohomology~$\rmH^\sbull(E,\kk)$ and let $1\neq H\lneqq E$ be a fixed non-trivial subgroup.
Suppose that the only homogeneous prime containing~$I$ and all the~$\zeta_N$ for $N\ge H$ (\Cref{Rem:zeta_N}) is the maximal ideal~$\rmH^+(E,\kk)$.
Then there exists in~$I$ a homogeneous\,{\rm(\footnote{\,The grading is the usual $\bbN$-grading in which all the~$\zeta_N$ have the same degree~$2'$. In particular, the first term~$\prod_{M\not\ge H}\zeta_M^d$ in~$f$ has degree~$2'\cdot d\cdot |\SET{M\in \cN}{M\not\ge H}|$.})} polynomial~$f$ of the form
\[
f = \prod_{M\not\ge H}\zeta_M^d + \sum_{m} \lambda_m \cdot \prod_{N\in\cN}\zeta_N^{m(N)}
\]
for some integer~$d\ge 1$ and scalars~$\lambda_m\in \kk$ and finitely many exponents $m\in\bbN^{\cN}$ that satisfy the following properties:
\begin{equation}
\label{eq:japan}%
m(N)\ge 1\textrm{ for at least one }N\ge H
\quadtext{and}
m(N')<d\textrm{ for all }N'\not\ge H.
\end{equation}
\end{Lem}

\begin{proof}
For simplicity, we work in the subring $\rmH^\ast\subseteq \rmH^\sbull(E,\kk)$ generated by the $\zeta_N$. (For $p=2$, this is the whole cohomology anyway and for $p$ odd we only miss nilpotent elements, which are mostly irrelevant for the problem, as we can always raise everything in sight to a large $p$-th power.)
Let us denote the maximal ideal by~$\gm=\ideal{\zeta_N\mid N\in\cN}$.
It is also convenient to work in the quotient $\bbN$-graded ring
\[
A^\ast:=\rmH^\ast(E,\kk)/I
\]
which is generated, as a $\kk$-algebra, by the classes $\bar{\zeta}_N$ of all~$\zeta_N$ modulo~$I$.

The assumption about $Z(I+\SET{\zeta_N}{N\ge H})=\{\gm\}$ implies that $\gm$ has some power contained in~$I+\ideal{\SET{\zeta_N}{N\ge H}}$.
In other words when $N'\not\ge H$ we have
\begin{equation}
\label{eq:aux-japan-1}%
(\bar\zeta_{N'})^d\in\ideal{\bar\zeta_{N}\mid N\ge H}_{A^{\ast}}
\end{equation}
for $d\gg1$ that we take large enough to work for all the (finitely many)~$N'\not\ge H$.

Consider this ideal~$J=\ideal{\,\bar\zeta_{N}\mid N\ge H\,}$ of~$A^{\ast}$ more carefully. It is a $\kk$-linear subspace generated by the classes~$\bar{\theta}_m$ of the following products in~$\rmH^*$
\begin{equation}
\label{eq:aux-japan-2}%
{\theta}_m:=\prod_{N\in\cN}(\zeta_N)^{m(N)}
\end{equation}
with $m\in\bbN^\cN$ such that~$m(N)\ge 1$ for at least one~$N\ge H$.
We claim that $J$ is in fact generated over~$\kk$ by the subset of the $\bar{\theta}_m$ for the special $m\in\bbN^\cN$ satisfying~\eqref{eq:japan}.
Indeed, let $J'\subseteq J$ be the $\kk$-subspace generated by the~$\bar{\theta}_m$ for the special~$m$.
Then we can prove that the class $\bar{\theta}_m$ of each product~\eqref{eq:aux-japan-2} belongs to~$J'$, by using~\eqref{eq:aux-japan-1} and (descending) induction on the number~$\sum_{N\ge H}m(N)$, for a fixed total degree~$\sum_{N}m(N)$.
We conclude that $J=J'$.

By~\eqref{eq:aux-japan-1}, the monomial~$\prod_{M\not\ge H}(\bar\zeta_{M})^d$ belongs to~$J$ and therefore to~$J'$: It is a $\kk$-linear combination of monomials of the form~$\bar{\theta}_m$ for~$m\in\bbN^\cN$ satisfying~\eqref{eq:japan}. Returning from~$A^{\ast}=\rmH^\ast(E,\kk)/I$ to~$\rmH^\ast(E,\kk)$, the difference between $\prod_{M\not\ge H}(\zeta_{M})^d$ and the same $\kk$-linear combination of the lifts~$\theta_m$ in~$\rmH^\ast(E,\kk)$ is an element of~$I$, that we call$~f$ and that fulfills the statement.
\end{proof}

\begin{Prop}
\label{Prop:closure}%
Let $Z\subset\Vee{E}$ be a non-empty closed subset of the cohomological open and let $1\neq H\lneqq E$ be a non-trivial subgroup.
Suppose that in~$\Vee{E}$, the subset~$Z$ intersects the image of the cohomological open of~$H$ in the smallest possible way:
\[
Z\cap \rho_H(\Vee{H}) = \{\cM(1)\}.
\]
Consider the closure $\bar{Z}$ of~$Z$ in the whole spectrum~$\SpcKE$.
Then $\bar{Z}$ does not intersect the stratum $\Vee{E}(H)=\psi^H(\Vee{E/H})$.
Hence $\cM(H)$ does not belong to~$\bar{Z}$.
\end{Prop}
\begin{proof}
Let $I\subset \rmH^\sbull(E,\kk)$ be the homogeneous ideal that defines~$Z$. The closed image~$\rho_H(\Vee{H})$ is given by the (partly redundant) equations~$\zeta_N=0$ for all~$N\ge H$. It follows that the intersection $Z\cap \rho_H(\Vee{H})$ is defined by the ideal~$I+\ideal{\,\zeta_N\mid N\ge H\,}$.
So our hypothesis translates exactly in saying that $I$ satisfies the hypothesis of \Cref{Lem:japan}. Hence there exists a homogeneous element of~$I$
\[
f = \prod_{M\not\ge H}\zeta_M^d +  \sum_{m} \lambda_m \prod_{N\in\cN}\zeta_N^{m(N)}
\]
for scalars~$\lambda_m\in\kk$ and finitely many exponents $m\in\bbN^\cN$ satisfying~\eqref{eq:japan}.
We can now use \Cref{Lem:inclusions} and follow Diagram~\eqref{eq:R(H)-span} with the ideal~$I$ and particularly with its element~$f$.
The element $Q'(f)$ is just~$f$ seen in~$\RE(\goodopen[E]{H}\cap \Vee{E})$. But it does not belong to the image of~$\REH$ under~$Q$ because~$f$ contains some~$b_M$ with~$M\not\ge H$ in denominators in the~$\zeta_M$'s.
Still, we can multiply~$Q'(f)$ by~$\prod_{M\not\ge H}(\frac{b_M}{a_M})^d=\prod_{M\not\ge H}(\zeta_M)^{-d}$ to get a degree-zero homogeneous element
\begin{equation}
\label{eq:aux-tilde-f}%
\tilde{f}=1 +  \sum_{m} \lambda_m \prod_{N\in\cN}\zeta_N^{m'(N)}
\end{equation}
in the ideal~$\ideal{Q'(f)}$, where we set the exponent $m'(N):=m(N)-d$ if~$N\not\ge H$ and $m'(N):=m(N)$ if $N\ge H$.
Note that by~\eqref{eq:japan} the exponent $m'(N)$ is negative if~$N\not\ge H$ and is non-negative if $N\ge H$ and strictly positive for at least one~$N\ge H$.
Both types of exponents of~$\zeta_N$ are allowed in~$\REH$, namely, when $N\not\ge H$, the element~$\zeta^-_N=\frac{b_N}{a_N}$ exists in~$\REH$.
In other words, the element $\tilde{f}\in\ideal{Q'(f)}$ satisfies
\[
\tilde{f} = Q(1 + \tilde{g})
\]
where $\tilde{g}\in \REH$ belongs to the ideal~$\ideal{\zeta_N\mid N\ge H}$ in~$\REH$ and must be of degree zero by homogeneity.
Now, for $N\ge H$, we have $\Psi^H(\zeta_N)=\zeta_{N/H}$ by \Cref{Prop:Psi^H(a/b)}.
It follows that $\Psi^H(\tilde{g})$ belongs to the maximal ideal~$\ideal{\,\zeta_{\bar N}\mid\bar{N}\in\cN(E/H)}\subseteq\rmH^+(E/H,\kk)$ of~$\rmH^\sbull(E/H,\kk)$ and still has degree zero. This forces~$\Psi^H(\tilde{g})=0$ and therefore $\Psi^H(1+\tilde{g})=1$ in~$\rmH^\sbull(E/H,\kk)$.
In the notation of \Cref{Lem:inclusions}, we have shown that $J$ contains~1, which implies $\bar{Z}\cap\Vee{E}(H)=\varnothing$.
\end{proof}

\begin{Cor}
\label{Cor:closure}%
Let $Z\subset\Vee{E}$ be a closed subset of the cohomological open, strictly larger than the unique closed point~$\cM(1)$ of~$\Vee{E}$. Suppose that in~$\Vee{E}$, the subset~$Z$ intersects the images of all proper subgroups trivially, \ie~$Z\cap \big(\bigcup_{H\lneqq E}\rho_H(\Vee{H})\big) = \{\cM(1)\}$.
Then the closure $\bar{Z}$ of~$Z$ in the whole spectrum~$\SpcKE$ has only one more point, namely $\bar{Z}=Z\cup\{\cM(E)\}$.
\end{Cor}
\begin{proof}
By \Cref{Prop:closure}, we see that $\bar{Z}$ does not meet any stratum~$\Vee{E}(H)$ for $H\neq E$. Thus the only point of~$\SpcKE$ outside of~$Z$ itself, hence outside of~$\Vee{E}$, that remains candidate to belong to~$\bar{Z}$ must belong to~$\supp(\Kac(E))\sminus\cup_{H\lneqq E}\Vee{E}(H)=\Vee{E}(E)=\{\cM(E)\}$. We know that $\cM(E)=\ideal{\kk(E/H)\mid H\lneqq E}$ in~$\cK(E)$, by \Cref{Exa:M(G)}. Take~$\cP\in Z$ different from~$\cM(1)$. Since~$\cP$ does not belong to any~$\Img(\rho_H)=\supp(\kk(E/H))$ by assumption, it must contain $\kk(E/H)$. Consequently, $\cM(E)\subseteq \cP$, meaning that $\cM(E)\in \adhpt{\cP}\subseteq\bar{Z}$.
\end{proof}
\begin{Cor}
\label{Cor:dim1-subvar-abelem}%
Let $E$ be an elementary abelian group of rank~$r$. Let $\cP$ be a point of height~$r-1$ in the cohomological open~$\Vee{E}$, that is, a closed point of the classical projective support variety~$\mathcal{V}_E(\kk):=\Vee{E}\sminus\{\cM(1)\}\cong\Proj(\rmH^\sbull(E,\kk))\cong\bbP^{r-1}_{\kk}$.
Suppose that~$\cP$ does not belong to the image~$\rho_H(\Vee{H})$ of the support variety of any proper subgroups~$H\lneqq E$.
Then the closure of~$\cP$ in~$\SpcKE$ is exactly the following
\[
\overline{\{\cP\}}=\{\cM(E),\cP,\cM(1)\}.
\]
\end{Cor}
\begin{proof}
Apply \Cref{Cor:closure} to~$Z=\{\cP,\cM(1)\}$, the closure of~$\cP$ in~$\Vee{E}$.
\end{proof}

\begin{Exa}
We can review the proof of \Cref{Prop:closure} in the special case of \Cref{Cor:dim1-subvar-abelem}, to see how elements like~$f\in\RE(1)$ and $\tilde{f}\in\REH$ come into play.
We do it in the special case where~$\cP$ is a $\kk$-rational point (\eg\ if $\kk$ is algebraically closed).
Let $1\neq H\lneqq E$ be a non-trivial subgroup (the case $r=1$ being trivial).
Choose $N_0,N_1\normal E$ index-$p$ subgroups with~$H\le N_0$ and $H\not\le N_1$.
They define coordinates $\zeta_0,\zeta_1$ in $\bbP^{r-1}$ (where $\zeta_{i}=\zeta_{N_i}$ as in \Cref{Rem:zeta_N}).
There exists a hyperplane of~$\bbP^{r-1}$
\begin{equation}
\label{eq:hyperplane}%
\lambda_0\zeta_0+\lambda_1\zeta_1=0,\qquad [\lambda_0:\lambda_1]\in \bbP^1(\kk),
\end{equation}
going through the rational point~$\cP$.
Note that $\lambda_1\neq 0$ as $\cP\notin Z(\zeta_0)=\rho_{N_0}(\Vee{N_0})$, by assumption.
As in \Cref{Lem:inclusions}, the following two localizations agree
\begin{equation}
\label{eq:common-loc}%
\Rloc{E}(H)\big[(\zeta^-_N)\inv\mid H\not\le N\big]=\Rloc{E}(1)\big[(\zeta^+_N)\inv\mid H\not\le N\big]
\end{equation}
where $N\normal E$ ranges over the index-$p$ subgroups as usual.
We find a lift
\[
\tilde{f}:=\lambda_0 \zeta_0\zeta^-_{N_1}+\lambda_1 \in \RE(H)
\]
of the element~$f=\lambda_0\zeta_0+\lambda_1\zeta_1\in \RE(1)$ of~\eqref{eq:hyperplane} suitably multiplied by~$\zeta_{1}\inv=\zeta^-_{N_1}$ in the localization~\eqref{eq:common-loc}.
Then we have $\Psi^H(\zeta^-_{N_1})=0$ since $H\not\le N_1$, by \Cref{Prop:Psi^H(a/b)}, so $\Psi^H(\tilde{f})=\lambda_1\in\kk^\times$ is an isomorphism.
We deduce that $\cone(\tilde{f})$ belongs to~$\cM(H)\sminus\cP$, which shows that $\cP$ does not specialize to~$\cM(H)$.
\end{Exa}

\begin{Exa}
\label{Exa:Klein-four}%
Let $E=C_2\times C_2$ be the Klein-four group.
Let us justify the description of~$\Spc(\cK(E))$ announced in \Cref{Exa:Klein4} in some detail:
\begin{equation}\label{eq:Klein}%
\vcenter{\xymatrix@C=.0em@R=.4em{
{\color{Black}\overset{\cM(E)}{\bullet}} \ar@{-}@[Gray][rrdd] \ar@{-}@[Gray][rrrrdd] \ar@{-}@[Gray][rrrrrrdd] \ar@{~}@<.1em>@[Red][rrrrrrrrdd] &&& {\color{Black}\overset{\cM(N_0)}{\bullet}} \ar@{-}@[Gray][ldd] \ar@{-}@<.1em>@[Gray][rrrrrrdd]
&& {\color{Black}\overset{\cM(N_1)}{\bullet}} \ar@{-}@[Gray][ldd] \ar@{-}@[Gray][rrrrrrdd]
&& {\color{Black}\overset{\cM(N_\infty)}{\bullet}} \ar@{-}@[Gray][ldd] \ar@{-}@[Gray][rrrrrrdd]
&& {\color{Black}\overset{\cM(1)}{\bullet}} \ar@{~}@[Gray][ldd] \ar@{-}@[Gray][dd] \ar@{-}@[Gray][rrdd] \ar@{-}@[Gray][rrrrdd]
\\ \\
&& {\color{Brown}\underset{\color{Brown}\eta_E(N_0)}{\bullet}} \ar@{-}@<-.4em>@[Gray][rrrrrrrdd]
&& {\color{Brown}\underset{\color{Brown}\eta_E(N_1)}{\bullet}} \ar@{-}@<-.1em>@[Gray][rrrrrdd]
&& {\color{Brown}\underset{\color{Brown}\eta_E(N_\infty)}{\bullet}} \ar@{-}@[Gray][rrrdd]
&\ar@{.}@[RoyalBlue][r]
& {\scriptstyle\color{RoyalBlue}\tinyPone} \ar@{~}@[Gray][rdd] \ar@{.}@[RoyalBlue][rrrrrrr]
& {\color{OliveGreen}\underset{\color{OliveGreen}0}{\bullet}} \ar@{-}@[Gray][dd]
&& {\color{OliveGreen}\underset{\color{OliveGreen}1}{\bullet}} \ar@{-}@[Gray][lldd]
&& {\color{OliveGreen}\underset{\color{OliveGreen}\infty}{\bullet}} \ar@{-}@[Gray][lllldd]
&&
\\ \\
&&&&&&&&& {\color{Black}\bullet_{\eta_E}}\kern-1em
&&&&
}}\kern-1em
\end{equation}
By \Cref{Rec:Part-I}, we have a partition of the spectrum as a set
\[
\SpcKE=\Vee{E}(E)\,\sqcup\, \Vee{E}(N_0)\,\sqcup\, \Vee{E}(N_1)\,\sqcup\, \Vee{E}(N_\infty)\,\sqcup\, \Vee{E},
\]
where we write $N_0,N_1,N_\infty$ for the three cyclic subgroups~$C_2$ and where $\Vee{E}=\Vee{E}(1)$ is the cohomological open as usual. Let us review those five parts $\Vee{E}(H)=\psi^{H}(\Vee{E/H})$ separately, in growing order of complexity, \ie from left to right in~\eqref{eq:Klein}.

For $H=E$, the stratum $\Vee{E}(E)=\psi^E(\Vee{E/E})=\{\cM(E)\}$ is just a closed point.

For each cyclic subgroup~$N_i<E$, the quotient~$E/N_i\simeq C_2$ is cyclic, so~$\Spc(\cK(E/N_i))$ is the space of~\Cref{Exa:C_p}.
Its image under $\psi^{N_i}$ is~$\{\cM_E(E),\eta_E(N_i),\cM_E(N_i)\}$, defining the (brown) point~$\eta_E(N_i):=\psi^{N_i}(\eta_{E/N_i})$, as in the proof of \Cref{Prop:SpcKE-irreducible}.
The stratum~$\Vee{E}(N_i)$ is the image of the cohomological open~$\Vee{E/N_i}$ only, that is, the Sierpi\'{n}ski space~$\{\eta_E(N_i),\cM(N_i)\}$, whose non-closed point~$\eta_E(N_i)$ is the generic point of the irreducible~$\{\cM_E(E),\eta_E(N_i),\cM_E(N_i)\}$ in~$\SpcKE$.

Finally, for $H=1$, the cohomological open $\Vee{E}=\Spc(\Db(\kk E))\cong\Spech(\kk[\zeta_0,\zeta_1])$ is a $\bbP^1_\kk$ with a closed point~$\cM(1)$ on top.
We denote by~$\eta_E$ the generic point of~$\SpcKE$ as in~\Cref{Prop:SpcKE-irreducible} and by~$0,1,\infty$ the three $\bbF_2$-rational points of~$\bbP^1_\kk$ (in green).
The notation~$\Pone$ refers to all remaining points of~$\bbP^1_\kk$.
The undulated lines indicate that \emph{all} points of~$\Pone$ have the same behavior.
Namely, $\eta_E$ specializes to all points of~$\Pone$ and every point of~$\Pone$ specializes to~$\cM(1)$ and the (red) undulated line towards~$\cM(E)$ indicates that all points of~$\Pone$ specialize to~$\cM(E)$, as follows from \Cref{Cor:dim1-subvar-abelem}.
(Note that the latter was rather involved: Its proof occupies most of this section, and relies on technical \Cref{Lem:japan}.)

We have described the closure of every point in~$\SpcKE$, except for the $\bbF_2$-rational points~$0,1,\infty$.
For this, we use the closed immersion~$\rho_{N_i}\colon \Spc(\cK(N_i))\hook \SpcKE$ induced by restriction~$\Res_{N_i}$.
The point $i$ is the image of the generic point~$\eta_{N_i}$ of the V-shaped space~$\Spc(\cK(N_i))$ of~\Cref{Exa:C_p}. Hence its closure is~$\Img(\rho_{N_i})=\{\cM(E_i),i,\cM(1)\}$. So specializations are exactly those of~\eqref{eq:Klein}.

We revisit this picture in more geometric terms in \Cref{Exa:Klein-revisit}.
\end{Exa}

\begin{Rem}
\label{Rem:dim1-subvar}%
It is possible to extend \Cref{Cor:dim1-subvar-abelem} to a general finite group~$G$ by means of the Colimit \Cref{Thm:colim}.
Let $Z\subseteq \SpcKG$ be a one-dimensional irreducible closed subset.
Write its generic point as $\cP=\cP(K,\gp)$ for (unique) $K\in\Sub{p}(G)_{/G}$ and $\gp\in \Vee{\WGK}$.
By Quillen applied to~$\bar{G}=\WGK$, there exists a minimal elementary abelian subgroup $E\le \bar{G}$ such that $\gp\in\Img(\rho_E\colon \Vee{E}\to\Vee{\bar{G}})$, also unique up to $\bar{G}$-conjugation. This~$E\le \bar{G}=(N_GK)/K$ is given by $E=H/K$ for $H\le N_G K$ containing~$K$. Then $\cP=\varphi_{(H,K)}(\cQ)$ where $\cQ\in\SpcKE$ is given by $\cQ=\cP_{E}(1,\gq)$ for some $\gq\in\Vee{E}$.
By \Cref{Lem:preserve-dimension}, the map $\varphi_{(H,K)}\colon \SpcKE\to \SpcKG$ is closed and preserves the dimension of points.
It follows that $\cQ$ is also the generic point of a one-dimensional irreducible in~$\SpcKE$.
By minimality of~$E$, the point~$\cQ\in\Vee{E}$ does not belong to $\Vee{H'}$ for any proper subgroup~$H'<E$.
By \Cref{Cor:dim1-subvar-abelem}, we have $\adhpt{\cQ}=\{\cM_E(E),\cQ,\cM_E(1)\}$ in~$\SpcKE$.
The map $\varphi_{(H,K)}$ sends this subset to~$\{\cM_G(H), \cP, \cM_G(K)\}$.
In summary, every one-dimensional irreducible subset of~$\SpcKG$ is of the form $Z=\{\cM(H),\cP,\cM(K)\}$, where $H$ and~$K$ are uniquely determined by the generic point~$\cP$ via the above method.
\end{Rem}

\section{Presentation of twisted cohomology}
\label{sec:presentation-H**}

We remain in the case of an elementary abelian group~$E$.
In this section we want to better understand the local $\bbZ$-graded rings $\REH$ that played such an important role in \Cref{sec:abelem}.
Thankfully they are reasonable $\kk$-algebras.

\begin{Ter}
\label{Ter:coord}%
Recall that we write $C_p=\ideal{\,\sigma\mid\sigma^p=1\,}$ for the cyclic group of order~$p$ with a chosen generator~$\sigma$.
For brevity we call an $\FFp$-linear surjection $\pi\colon E\onto C_p$ a \emph{coordinate}.
For two coordinates $\pi,\pi'$ we write $\pi\sim\pi'$ if $\ker(\pi)=\ker(\pi')$.
Finally, for a subgroup~$H$, we often abbreviate $H\mid\pi$ to mean $H\le \ker(\pi)$.

Recall from \Cref{Def:u_N} and \Cref{Rem:u_N} that each coordinate~$\pi$ yields an invertible object~$u_\pi=\pi^*u_p$ in~$\cK(E)$.
It comes with maps $a_\pi,b_\pi,c_\pi\colon \kk\to u_\pi[\ast]$.
\end{Ter}

\begin{Rem}
\label{Rem:similar-coordinates}
If $\pi\sim\pi'$ then there exists a unique $\lambda\in\FFp^\times$ such that $\pi'=\pi^{\lambda}$.
Hence, if $p=2$ then necessarily $\pi=\pi'$ and $u_\pi=u_{\pi'}$.
On the other hand, if $p>2$ is odd then we still have $u_\pi\cong u_{\pi'}$ as already mentioned.
Explicitly, consider the automorphism $\lambda\colon C_p\to C_p$ that sends~$\sigma$ to~$\sigma^\lambda$.
The isomorphism $u_{\pi}=\pi^*u_p\isoto \pi^*\lambda^*u_p=(\pi^{\lambda})^*u_p=u_{\pi'}$ will be the pullback $\pi^*\Lambda$ along~$\pi$ of an isomorphism of complexes~$\Lambda\colon u_p\isoto \lambda^* u_p$.
This isomorphism~$\Lambda$ can be given explicitly by the identity in degree~$0$ and by the $\kk C_p$-linear maps $\kk C_p\to \lambda^*\kk C_p$ in degree~$1$ (resp.~$2$) determined by $1\mapsto 1$ (resp.~$1\mapsto 1+\sigma+\cdots\sigma^{\lambda-1}$).
One checks directly that $\Lambda\circ a_p=a_p$ and $\Lambda\circ b_p=\lambda \cdot b_p$. By applying~$\pi^*$ we obtain
\begin{align}
  \label{eq:master-relations-scalar}
  (\pi^*\Lambda)\circ a_\pi=a_{\pi'}\qquadtext{and}
  (\pi^*\Lambda)\circ b_{\pi}=\lambda \cdot b_{\pi'}.
\end{align}
\end{Rem}
\begin{Lem}
\label{Lem:master-relations}
Given coordinates~$\pi_1\not\sim\pi_2$ set $\pi_3=\pi_1\inv\pi_2\inv$.
Write $u_i$, $a_i$ and $b_i$ for $u_{\pi_i}$, $a_{\pi_i}$ and $b_{\pi_i}$ in~$\cK(E)$.
Then we have the relation
\begin{align*}
a_1b_2b_3+b_1a_2b_3+b_1b_2a_3&=0
\end{align*}
as a map from~$\unit$ to~$(u_1\otimes u_2\otimes u_3)[-2'\cdot 2]$ in~$\cK(E)$. (See \Cref{Not:2'} for~$2'$.)
\end{Lem}
\begin{proof}
Let $N_i=\ker(\pi_i)$ for $i=1,2,3$, which are all distinct.
Let $N=N_1\cap N_2\cap N_3$ be the common kernel, which is of index~$p^2$ in~$E$.
By inflation along~$E\onto E/N$, it suffices to prove the lemma for $E=C_p\times C_p$ and $\pi_1$ and~$\pi_2$ the two projections on the factors.
We abbreviate~$u$ for the complex of permutation~$\kk E$-modules $u:=u_1\otimes u_2\otimes u_3$.
Consider the permutation module
$M:=kC_p\otimes kC_p\otimes kC_p\cong k(E/N_1)\otimes k(E/N_2)\otimes k(E/N_3)$ which appears as a summand in various degrees of the complex~$u$.
One element in~$M$ is of particular interest:
\begin{align*}
  m&:=\sum_{i_1,i_2=0}^{p-1}\sigma^{i_1}\otimes \sigma^{i_2}\otimes \sigma^{-i_1-i_2}.
\end{align*}
It is easy to check that $m$ is $E$-invariant, thus defines a $\kk E$-linear map $\tilde{m}\colon k\to M$, that can be used to define the required homotopies. This depends on~$p$.
If $p=2$, the homotopy is given by $\tilde{m}$ when viewed from~$\unit$ to the only $M$-entry of~$u[-2]$ in degree one.
If $p>2$, the homotopy is given by $(\tilde{m},\tilde{m},\tilde{m})$ as a map from $\unit$ to the three $M$-entries of~$u[-4]$ in degree one.
Verifications are left to the reader.
\end{proof}

\begin{Cons}
We construct a commutative $\kk$-algebra $\Rlocp{E}(H)$ by generators and relations. Its generators are indexed by coordinates~$\pi\colon E\onto C_p$ (\Cref{Ter:coord})
\[
\SET{\zeta_\pi^+}{\pi\textrm{ s.t.\ }H\le\ker(\pi)}\ \cup\ \SET{\zeta_{\pi}^-}{\pi\textrm{ s.t.\ }H\not\le\ker(\pi)}.
\]
These generators come equipped with a degree in~$\bbZ$: If \mbox{$H\mid \pi$} the generator~$\zeta^+_\pi$ is set to have degree~$2'$, whereas if \mbox{$H\nmid\pi$} the generator~$\zeta^-_\pi$ is set to have degree~$-2'$.
We impose the following four families of homogeneous relations. First for every coordinate~$\pi$ and every $\lambda\in \FFp^\times$ (for $p$ odd), we have a rescaling relation
\begin{enumerate}[label=\rm(\alph*), ref=\rm(\alph*)]
\item \label{rel-z-scalar}
$\zeta_{\pi^{\lambda}}^+=\lambda\zeta^+_\pi$ if $H\mid\pi$  \quad and \quad $\zeta_{\pi^{\lambda}}^-=\lambda^{-1}\zeta^-_{\pi}$ if~$H\nmid\pi$
\end{enumerate}
and whenever $\pi_3=\pi_1\inv\pi_2\inv$ and $\pi_1\not\sim\pi_2$, writing $\zeta_i^{\pm}:=\zeta_{\pi_i}^{\pm}$, we impose one of the following relations, inspired by~\Cref{Lem:master-relations}:
\begin{enumerate}[label=\rm(\alph*), ref=\rm(\alph*)]
\setcounter{enumi}{1}
\item \label{rel-+++}
$\zeta_{1}^++\zeta_{2}^++\zeta^+_{3}=0$, if $H\mid\pi_1$ and~$H\mid\pi_2$ (and therefore $H\mid\pi_3$)
\smallbreak
\item \label{rel---+}
$\zeta^-_{1}+\zeta^-_{2}+\zeta^-_{1}\zeta^-_{2}\zeta_{3}^+=0$, if $H\nmid\pi_1$ and~$H\nmid\pi_2$ but $H\mid\pi_3$
\smallbreak
\item \label{rel-zx-2--}
$\zeta^-_{1}\zeta^-_{2}+\zeta^-_{2}\zeta^-_{3}+\zeta^-_{3}\zeta^-_{1}=0$ if $H\nmid\pi_i$ for all~$i=1,2,3$.
\end{enumerate}
Since these relations are homogeneous, the ring~$\Rlocp{E}(H)$ is a $\bbZ$-graded ring.
\end{Cons}

\begin{Rem}
\label{Rem:R_E}%
We could also define a multi-graded commutative $\kk$-algebra $\Rallp{E}$ generated by all~$a_\pi,b_\pi$ subject to the relations in~\eqref{eq:master-relations-scalar} and \Cref{Lem:master-relations}.
This algebra $\Rallp{E}$ would be $\ZZ\times\bbN^{\cN}$-graded with $a_\pi$ in degree~$(0,1_{\ker(\pi)})$ and $b_\pi$ in degree $(-2',1_{\ker(\pi)})$.
Then $\Rlocp{E}(H)$ is simply the `zero-twist' part of the localization of~$\Rallp{E}$ with respect to the~$a_\pi,b_\pi$ that become invertible in~$U(H)$, that is, those~$a_\pi$ such that $H\nmid\pi$ and those~$b_\pi$ such that $H\mid\pi$, as in \Cref{Def:O(H)}.
\end{Rem}

\begin{Rem}
By~\eqref{eq:master-relations-scalar} and~\Cref{Lem:master-relations}, there exists a canonical homomorphism
\begin{equation}
\label{eq:R2R}%
\Rlocp{E}(H)\to \REH
\end{equation}
mapping~$\zeta^+_\pi$ to~$\frac{a_{\pi}}{b_{\pi}}$ and~$\zeta^-_\pi$ to~$\frac{b_{\pi}}{a_{\pi}}$.
\end{Rem}

\begin{Exa}
\label{Exa:Rloc(1)}%
Let $H=1$. Recall from~\Cref{Exa:R(1)=H*} that $\Rloc{E}(1)$ is the cohomology ring. Then the homomorphism~\eqref{eq:R2R} is the standard one $\Rlocp{E}(1)\to\rmH^\sbull(E;\kk)$, that maps~$\zeta^+_\pi$ to the usual generator~$\zeta_{\pi}=\pi^*(\zeta_{C_p})$. Note that here $H=1\mid\pi$ for all~$\pi$, so there is no~$\zeta^-_\pi$.
For~$E$ elementary abelian, it is well-known that this homomorphism $\Rlocp{E}(1)\to\rmH^\sbull(E;\kk)$ is an isomorphism modulo nilpotents.
See for instance~\cite{carlson:ETH-lectures}.
\end{Exa}

For two subgroups~$H,K\le E$, the open subsets $\goodopen[E]{H}$ and $\goodopen[E]{K}$ can intersect in~$\SpcKE$. Similarly, we can discuss what happens with the rings~$\Rlocp{E}(H)$.
\begin{Prop}
\label{Prop:Rlocp-intersect}%
Let $H,K\le E$ be two subgroups. Define $S=S(H,K)\subset\Rlocp{E}(H)$ to be the multiplicative subset generated by the finite set
\[
\SET{\zeta^+_{\pi}}{\textrm{ for }\pi\textrm{ with }H\mid \pi\textrm{ and }K\nmid\pi}\cup\SET{\zeta^-_{\pi}}{\textrm{ for }\pi\textrm{ with }H\nmid \pi\textrm{ and }K\mid\pi}
\]
and similarly, swapping~$H$ and~$K$, let $T=S(K,H)\subset\Rlocp{E}(K)$ be the multiplicative subset generated by $\SET{\zeta^+_{\pi}}{H\nmid \pi\textrm{ and }K\mid\pi}\cup\SET{\zeta^-_{\pi}}{H\mid \pi\textrm{ and }K\nmid\pi}$. Then we have a canonical isomorphism of (periodic) $\bbZ$-graded rings
\[
S\inv \Rlocp{E}(H)\cong T\inv \Rlocp{E}(K)
\]
and in particular of their degree-zero parts. Thus the open of~$\Spech(\Rlocp{E}(H))$ defined by~$S$ is canonically homeomorphic to the open of~$\Spech(\Rlocp{E}(K))$ defined by~$T$.
\end{Prop}
\begin{proof}
The left-hand side $S\inv \Rlocp{E}(H)$ is the (`zero-twist' part of the) localization of the multi-graded ring~$\Rallp{E}$ of \Cref{Rem:R_E} with respect to
\begin{multline*}
\SET{a_\pi}{H\nmid\pi}\cup\SET{b_\pi}{H\mid\pi}\cup\SET{a_\pi}{H\mid\pi, K\nmid\pi}\cup\SET{b_\pi}{H\nmid\pi, K\mid\pi}\\
=\SET{a_\pi}{H\nmid\pi\text{ or }K\nmid\pi}\cup\SET{b_\pi}{H\mid\pi\text{ or }K\mid\pi}
\end{multline*}
which is symmetric in~$H$ and~$K$.
This completes the proof.
\end{proof}

\begin{Rem}
\label{Rem:intersection-compatibility}
The above isomorphism is compatible with the homomorphism~\eqref{eq:R2R}, namely the obvious diagram commutes when we perform the corresponding localizations on~$\Rloc{E}(H)$ and~$\Rloc{E}(K)$.
\end{Rem}

\begin{Prop}
\label{Prop:Rlocp-quot}%
Let $K\le H\le E$.
There is a canonical split epimorphism $`\Psi^K\colon \Rlocp{E}(H)\onto \Rlocp{E/K}(H/K)$ whose kernel is $\ideal{\,{\zeta^-_{\pi}}\mid{K\nmid \pi}\,}$. It is compatible with the homomorphism~$\Psi^K$ of \Cref{Rem:Psi-on-O(H)}, in that the following diagram commutes
\[
\xymatrix@C=3em@R=2em{
\Rlocp{E}(H) \ar@{->>}[d]_-{`\Psi^K} \ar[r]^-{\textrm{\eqref{eq:R2R}}} &
\Rloc{E}(H) \ar@{->>}[d]^-{\Psi^K}
\\
\Rlocp{E/K}(H/K) \ar[r]^-{\textrm{\eqref{eq:R2R}}}
& \Rloc{E/K}(H/K).
}
\]
\end{Prop}
\begin{proof}
Set~$\bar{H}:=H/K\le \bar{E}:=E/K$. Similarly, for every coordinate~$\pi\colon E\onto C_p$ such that $K\mid \pi$, let us write $\bar{\pi}\colon E/K\onto C_p$ for the induced coordinate.
The  morphism $`\Psi^K$ will come from a morphism $``\Psi^K\colon\Rallp{E}\to\Rallp{E/K}$, with respect to the homomorphism of gradings~\eqref{eq:gamma-Psi^H}.
As these algebras are constructed by generators and relations (\Cref{Rem:R_E}), we need to give the image of generators.
In view of \Cref{Prop:Psi^H(a/b)} we define $``\Psi^K\colon\Rallp{E}\to \Rallp{E/K}$ on generators by
\begin{align*}
  a_\pi\mapsto
         \begin{cases}
           \ a_{\bar\pi}&\textrm{if }K\mid\pi\\
           \ 1&\textrm{if }K\nmid\pi
         \end{cases}&&
  b_\pi\mapsto
         \begin{cases}
           \ b_{\bar \pi}&\textrm{if }K\mid\pi\\
           \ 0&\textrm{if }K\nmid\pi.\!\!
         \end{cases}
\end{align*}
It is easy to see that the relations in~$\Rallp{E}$ are preserved; thus the map~$``\Psi^K$ is well-defined.
Let $\varpi\colon E\onto E/K$ and for every $\bar{\pi}\colon \bar{E}\onto C_p$ consider the coordinate $\pi=\bar{\pi}\circ\varpi\colon E\onto C_p$. Then $\bar{H}\mid\bar{\pi}$ if and only if $H\mid\pi$.
It follows that the morphism passes to the localizations $`\Psi^K\colon \Rlocp{E}(H)\onto \Rlocp{E/K}(H/K)$ as announced.
The statement about its kernel is easy and commutativity of the square follows from the fact (\Cref{Rem:Psi-on-O(H)}) that $\Psi^K$ treats the $a_\pi$ and~$b_\pi$ according to the same formulas.

The section of~$`\Psi^K$ is inspired by inflation.
Namely, $a_{\bar\pi}\mapsto a_\pi$ and $b_{\bar\pi}\mapsto b_\pi$ defines a map of graded $\kk$-algebras $\Rallp{\bar E}\to\Rallp{E}$ that is already a section to~$``\Psi^K$ and passes to the localizations.
\end{proof}

\begin{Thm}
\label{Thm:presentation}%
The canonical homomorphism~\eqref{eq:R2R} induces an isomorphism
\[
\Rlocp{E}(H)_{\red}\isoto \REH_{\red}
\]
of reduced $\ZZ$-graded $\kk$-algebras.
\end{Thm}

\begin{proof}
It follows from \Cref{Rem:cohomology-generators} that the map is surjective.
We will now show that the closed immersion $\Spech(\REH)\hook\Spech(\Rlocp{E}(H))$ is surjective---this will complete the proof, by the usual commutative algebra argument, which can be found in \cite[Lemma~2.22]{hesselholt-pstragowski:dirac1} for the graded case.
By \Cref{Thm:abelem-top}, this is equivalent to showing the surjectivity of the composite with~$\comp_{E}$, that we baptize~$\beta^H$
\begin{equation}
\label{eq:phi}%
\xymatrix@C=3em{
\beta^H\colon & \kern-2em \goodopen[E]{H} \ar[r]_-{\comp_{E}}^-{\simeq}
& \Spech(\REH) \ \ar@{^(->}[r] & \Spech(\Rlocp{E}(H)).
}
\end{equation}

We proceed by induction on the order of the subgroup~$H$. If $H=1$ the result follows from \Cref{Exa:Rloc(1)}.
So suppose that $H\neq 1$ and pick a homogeneous prime~$\gp\in \Spech(\Rlocp{E}(H))$. We distinguish two cases.

If for every coordinate $\pi\colon E\onto C_p$ such that $H\nmid\pi$ we have $\zeta^-_\pi\in\gp$ then $\gp$ belongs to $V(\SET{\zeta^-_\pi}{H\nmid \pi})$, which we identify with the image of~$\Spech(\Rlocp{E/H}(1))$ by \Cref{Prop:Rlocp-quot} applied to $K=H$.
Namely, we have a commutative square
\[
\xymatrix@C=5em@R=2em{
  \goodopen[E]{H}
  \ar[d]_{\beta^H}
  &
  \goodopen[E/H]{1}
  \ar[d]^{\beta^1}
  \ar[l]_{\psi^H}
  \\
  \Spech(\Rlocp{E}(H))
  &
  \Spech(\Rlocp{E/H}(1))
  \ar[l]_{\Spech(`\Psi^H)}
}
\]
and since the right-hand vertical arrow is surjective by the case already discussed, we conclude that $\gp$ belongs to the image of~$\beta^H$ in~\eqref{eq:phi} as well.

Otherwise, there exists a coordinate~$\pi_1$ such that $H\nmid\pi_1$ and~$\zeta^-_{\pi_1}\notin\gp$. Let $K:=H\cap \ker(\pi_1)$ and let $S=S(H,K)$ be defined as in \Cref{Prop:Rlocp-intersect}:
\[
S=\SET{\zeta^-_\pi}{\text{for }\pi \text{ with }H\nmid\pi\textrm{ and }K\mid\pi}.
\]
We claim that $\gp$ belongs to the open of $\Spech(\Rlocp{E}(H))$ defined by~$S$.
Indeed, let $\zeta_{\pi_2}^-\in S$, that is for $\pi_2$ with $H\nmid\pi_2$ and $K\mid\pi_2$, and let us show that $\zeta_{\pi_2}^-\notin\gp$.
If $\pi_2\sim\pi_1$ this is clear from $\zeta^-_{\pi_1}\notin\gp$ and the relation~\ref{rel-z-scalar} in~$\Rlocp{E}(H)$.
If $\pi_2\not\sim\pi_1$, let $h\in H\sminus K$ (so that $h$ generates the cyclic group $H/K\cong C_p$).
As $\pi_1(h)\neq 1$ and $\pi_2(h)\neq 1$ we may replace $\pi_1$ by an equivalent coordinate~$\tilde{\pi}_1:=\pi_1^\lambda$ such that $\tilde{\pi}_1(h)=\pi_2(h)\inv$ and therefore
$H\mid\pi_3:=\tilde{\pi}_1\inv\pi_2\inv$.
Then relation~\ref{rel---+} exhibits~$\zeta_{\tilde{\pi}_1}^-$ as a multiple of~$\zeta_{\pi_2}^-$.
As the former does not belong to~$\gp$ (by the previous case), neither does~$\zeta_{\pi_2}^-$.
At this point we may apply \Cref{Prop:Rlocp-intersect} for our subgroups~$H$ and~$K$.
By \Cref{Rem:intersection-compatibility}, we have a commutative triangle:
\[
\xymatrix@C=2em@R=2em{
  &
  \goodopen[E]{H}\cap\goodopen[E]{K}
  \ar[dl]_{\beta^H}
  \ar[dr]^{\beta^K}
  \\
  \Spech(\Rlocp{E}(H)[S\inv])
  \ar@{<->}[rr]^{\approx}
  &&
  \Spech(\Rlocp{E}(K)[T\inv])
}
\]
We just proved that $\gp$ belongs to the open subset in the bottom left corner.
As $K$ is a proper subgroup of~$H$, we know that $\beta^K$ is surjective by induction hypothesis and we conclude that $\gp$ belongs to the image of~$\beta^H$ as well.
\end{proof}

\begin{Rem}
In \Cref{Thm:presentation} we have proved something slightly more precise, namely that the map
\begin{equation}
\label{eq:REH-mod-xi}
\Rlocp{E}(H)\to\REH/\langle\xi^{\pm}_\pi\rangle
\end{equation}
(where $\pi$ ranges over all coordinates)
is surjective with nilpotent kernel.
We expect that $\Rlocp{E}(H)$ is already reduced, which would imply that~\eqref{eq:REH-mod-xi} is in fact an isomorphism of graded rings.
In particular, for $p=2$ we expect that $\Rlocp{E}(H)\isoto\Rloc{E}(H)$.
\end{Rem}

\section{Applications and examples}
\label{sec:applications}%

In this final section, we push our techniques further and compute more examples.

\begin{Rem}
\label{Rem:remove-closed-pts}%
For $E$ elementary abelian, \Cref{Cor:abelem-Dirac,Thm:presentation} allow us to think of the \emph{geometry} of $\SpcKE$, beyond its mere topology, by viewing $\SpcKE$ as a Dirac scheme.
Consider further the `periodic' locus of $\SpcKE$, which is the open complement of the closed points~$\SET{\cM(H)}{H\le E}$; see~\Cref{Rec:Part-I}.
This is analogous to considering the projective support variety~$\Proj(\rmH^\sbull(E,\kk))\cong\bbP^{r-1}_{\kk}$ by removing the `irrelevant ideal'~$\cM(1)=\rmH^+(E,\kk)$ from~$\Spech(\rmH^\sbull(E,\kk))$. To avoid confusion with the phrase `closed points', we now refer to the $\cM(H)$ as \emph{very closed points}, allowing us to speak of closed points of~$\bbP^{r-1}_{\kk}$ in the usual sense (as we did in \Cref{Cor:dim1-subvar-abelem}).
Removing those finitely many `irrelevant' points allows us to draw more geometric pictures by depicting the (usual) closed points of the periodic locus, as in classical algebraic geometry.

In fact, for any finite group~$G$, we can speak of the \emph{periodic locus} of~$\SpcKG$ to mean the open~$\Spc'(\cK(G)):=\SpcKG\sminus\SET{\cM(H)}{H\in\Sub{p}(G)}$ obtained by removing the `irrelevant' very closed points.
However, we do not endow these spectra with a scheme-theoretic structure beyond the elementary abelian case, since we do not have \Cref{Cor:abelem-Dirac} in general.
We postpone a systematic treatment of the periodic locus to later work. For now we focus on examples.
\end{Rem}

\begin{Exa}
\label{Exa:Klein-revisit}
Let us revisit Klein-four, with the notation of \Cref{Exa:Klein-four}.
From the picture in~\eqref{eq:Klein} we see that the union of the open subsets $\goodopen[E]{1}$ and $\goodopen[E]{E}$ only misses (three) very closed points hence covers the periodic locus.
We have
\begin{equation}
\label{eq:Klein-Rlocp}
\begin{aligned}
\Rlocp{E}(1)&=\frac{\kk[\zeta^+_{N_0},\zeta^+_{N_1},\zeta^+_{N_{\infty}}]}{\langle{\zeta^+_{N_0}+\zeta^+_{N_1}+\zeta^+_{N_{\infty}}}\rangle}&&\left(=\Hm^*(E;\kk)\right),\\
  \Rlocp{E}(E)&=\frac{\kk[\zeta^-_{N_0},\zeta^-_{N_1},\zeta^-_{N_{\infty}}]}{\langle{\zeta^-_{N_0}\zeta^-_{N_1}+\zeta^-_{N_1}\zeta^-_{N_{\infty}}+\zeta^-_{N_{\infty}}\zeta^-_{N_0}}\rangle}
\end{aligned}
\end{equation}
and their homogeneous spectra are both a projective line with a unique closed point added.
(For~$\Rlocp{E}(E)$, the coordinate transformation for $i=0,1$, $\zeta^-_{N_i}\mapsto \tilde{\zeta}^-_i:=\zeta^-_{N_i}+\zeta^-_{N_\infty}$, identifies the ring with $\kk[\tilde{\zeta}^-_{0},\tilde{\zeta}^-_{1},\zeta^-_{N_\infty}]/\langle\tilde{\zeta}^-_{0}\tilde{\zeta}^-_{1}+(\zeta^-_{N_\infty})^2\rangle$, which corresponds to the image of a degree-two Veronese embedding of~$\mathbb{P}^1$ in~$\mathbb{P}^2$.)
Removing the very closed points (\Cref{Rem:remove-closed-pts}), it is a straightforward exercise to check that the two lines are glued along the open complement of the $\FF_{\!2}$-rational points, according to the rule $(\zeta^+_{N_i})\inv=\zeta_{N_i}^-$.
In other words, we obtain the following picture of~$\Spc'(\cK(E))$:
\begin{figure}[H]
\centering
\includegraphics[scale=0.2]{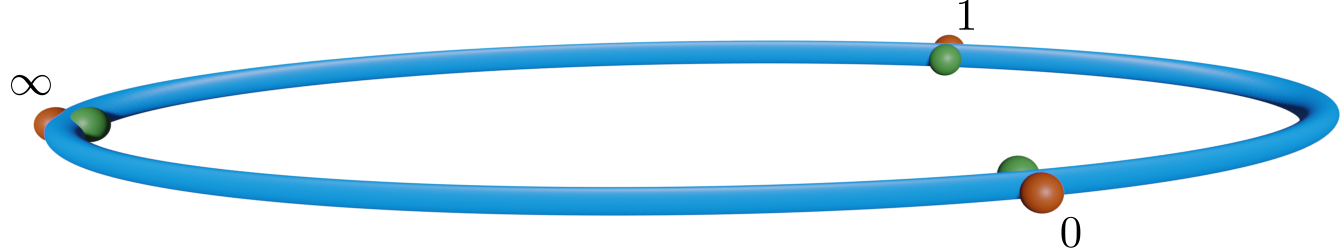}
\caption{A $\mathbb{P}^1_{\!\!\kk}$ with three doubled points.}
\label{fig:P1-doubled}
\end{figure}

To translate between this picture and the one in~\eqref{eq:Klein}, think of the blue part as~$\Pone$, the three green points as the $\FF_{\!2}$-rational points $i=0,1,\infty$ in $\goodopen[E]{1}=\Vee{E}$ and the brown points as the~$\eta_E(N_i)$ in~$\goodopen[E]{E}$.

In the above example as in the next, we encounter non-separated spectra. This is a standard phenomenon for $\bbZ$-graded rings, as explained in~\cite[\S\,3]{brenner-schroer:multihomogeneous} for instance.
We are not aware of a purely tensor-triangular interpretation of this fact.
\end{Exa}

\begin{Exa}
\label{Exa:Klein-conjugation}%
In view of later applications let us consider the action induced on spectra by the involution on~$E=C_2\times C_2$ that interchanges the two $C_2$-factors.
Let us say that the two factors correspond to the subgroups~$N_0$ and~$N_1$.
On the generators~$\zeta^{\pm}_{N_i}$ of~$\Rlocp{E}(1)$ and~$\Rlocp{E}(E)$ in~\eqref{eq:Klein-Rlocp}, the effect of the involution is
\[
\zeta^{\pm}_{N_0}\mapsto\zeta^{\pm}_{N_1}\qquad\zeta^{\pm}_{N_1}\mapsto\zeta^{\pm}_{N_0}\qquad\zeta^{\pm}_{N_\infty}\mapsto\zeta^{\pm}_{N_\infty}.
\]
The subrings of invariants in~$\Rlocp{E}(1)$ and~$\Rlocp{E}(E)$ are, respectively,
\begin{align*}
  \frac{\kk[e_1^+, e_2^+,\zeta^+_{N_\infty}]}{\langle{e_1^++\zeta^+_{N_\infty}}\rangle}\cong\kk[e_2^+,\zeta^+_{N_\infty}]
  &&\textrm{and}
  &&\frac{\kk[e_1^-,e_2^-,\zeta^-_{N_\infty}]}{\langle{e_1^-\zeta^-_{N_\infty}+e_2^-}\rangle}\cong\kk[e_1^-,\zeta^-_{N_\infty}]
\end{align*}
where $e_1^{\pm}=\zeta^\pm_{N_0}+\zeta^\pm_{N_1}$ and $e_2^{\pm}=\zeta^\pm_{N_0}\zeta^\pm_{N_1}$ are the first and second symmetric polynomials in~$\zeta^{\pm}_{N_0}$ and~$\zeta^{\pm}_{N_1}$.
Thus $e_i^{\pm}$ has degree~$\pm i$.
The homogeneous spectra of these rings (with unique very closed point removed) are again two projective lines\,(\footnote{\,More precisely, as already in \Cref{Exa:Klein-revisit}, we are dealing with weighted projective spaces which happen to be isomorphic to projective lines.}) and they are glued together along the complement of two points.
In other words, the quotient of~$\Spc'(\cK(E))$ by the involution is a~$\mathbb{P}^1_{\!\!\kk}$ with two doubled points:
\begin{figure}[H]
\centering
\includegraphics[scale=0.2]{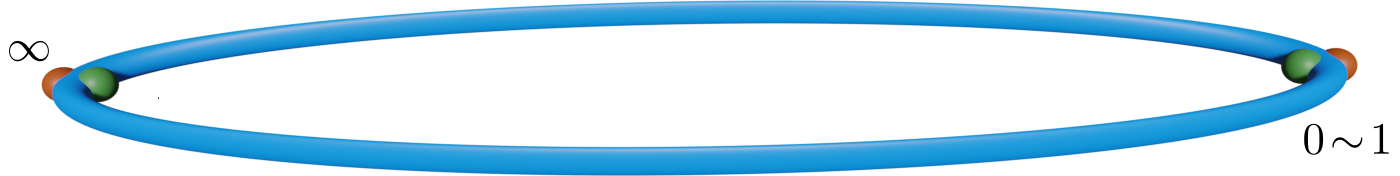}
\caption{A $\mathbb{P}^1_{\!\!\kk}$ with two doubled points.}
\label{fig:P1-2doubled}
\end{figure}

Alternatively, the topological space underlying this quotient may be obtained more directly at the level of \Cref{fig:P1-doubled}.
Indeed, this involution fixes the two colored points corresponding to~$\infty$, fixes no other points, and swaps the points corresponding to~$0$ with the points corresponding to~$1$, respecting the color.
So, again, the quotient can be pictured as a~$\mathbb{P}^1_{\!\!\kk}$ with only two doubled points as in \Cref{fig:P1-2doubled}.
\end{Exa}

\goodbreak
Let us return to general finite groups. We want to optimize the Colimit \Cref{Thm:colim} by revisiting the category of elementary abelian $p$-sections~$\EAppG$.
\begin{Rem}
\label{Rem:reduce-colim}%
In \Cref{Cons:EA}, we gave a `raw' version of the morphisms in the indexing category~$\EAppG$, which could be fine-tuned without changing the colimit~\eqref{eq:colim-elab}.
As with any colimit, we can quotient-out the indexing category $\EAppG\onto \EApG$ by identifying any two morphisms that induce the same map by the functor under consideration, here~$\Spc(\cK(-))$. We then still have
\begin{equation}
\label{eq:colim-elab-quot}%
\colim_{(H,K)\in\EApG}\Spc(\cK(H/K))\isoto\SpcKG.
\end{equation}
The same holds for any intermediate quotient $\EAppG\onto \EAt{p}{G}\onto\EApG$. For instance if $Z(G)$ denotes the center of~$G$, we can consider the category $\EAt{p}{G}$ obtained from~$\EAppG$ by modding out the obvious right action of the group $Z(G)\cdot H'$ on each hom set~$\hom_{\EAppG}((H,K),(H',K'))$.

Let us illustrate how such reductions can be used in practice.
\end{Rem}

\begin{Exa}
\label{Exa:cyclic}%
Let $G=C_{p^n}$ be the cyclic group of order~$p^n$.
As with any abelian group, using $Z(G)=G$, the reduced category $\EApG$ discussed in~\Cref{Rem:reduce-colim} just becomes a poset.
Here, if we denote by $1=H_{n}<H_{n-1}<\cdots<H_{1}<H_{0}=G$ the tower of subgroups of~$G$ then the poset $\EApG$ looks as follows:
\[
\xymatrix@C=.3em@R=.7em{
  &(H_{0},H_{1})
  &&&\cdots&&& (H_{n-1},H_{n})\\
  (H_{0},H_{0})\ar[ru]
  &&(H_{1},H_{1})\ar[lu]\ar[rru]
  &&\cdots&
  &(H_{n-1},H_{n-1})\ar[ru]\ar[llu]
  &&(H_{n},H_{n})\ar[lu]
  }
\]
From \Cref{Thm:colim} we deduce that $\SpcKG$ is the colimit of the diagram
\[
\xymatrix@C=.7em@R=.7em{
  &V
  &&V
  &&&&V
  \\
  \ast
  \ar[ru]
  &&\ast\ar[lu]\ar[ru]
  &&\ar[lu]
  &\cdots
  &\ar[ru]
  &&\ast\ar[lu]
}
\]
with $\ast=\Spc(\cK(1))$ and $V=\Spc(\cK(C_p))$ the V-shaped space in~\eqref{eq:C_p}. In the above diagram, the arrow to the right (resp.\ left) captures the left-most (resp.\ right-most) point of~$V$.
We conclude that the spectrum of $\cK(C_{p^n})$ is equal to
\begin{equation}
\label{eq:C_p^n}%
\vcenter{\xymatrix@R=1em@C=.7em{
{\scriptstyle\gm_0\kern-1em}
& {\bullet} \ar@{-}@[Gray] '[rd] '[rr] '[drrr]
&&
{\bullet}
&{\kern-1em\scriptstyle\gm_1}
&&{\scriptstyle\gm_{n-1}\kern-1em}&\bullet&&\bullet&{\kern-1em\scriptstyle\gm_n}
\\
&{\scriptstyle\gp_1\kern-1em}& {\bullet}&&& {\cdots}& \ar@{-}@[Gray] '[ru] '[rr] '[rrru] &&\bullet&{\kern-1em}\scriptstyle\gp_{n}}}
\end{equation}
This example reproves \Cref{Prop:cyclic}. It will provide the starting point for our upcoming work on the tt-geometry of Artin motives over finite fields.
\end{Exa}

\begin{Rem}
\label{Rem:E_p-EI-cat}%
The category of elementary abelian $p$-sections~$\EAppG$ is a finite \emph{EI-category}, meaning that all endomorphisms are invertible.
The same is true of its reduced versions $\EAt{p}{G}$ and $\EApG$ in \Cref{Rem:reduce-colim}.
\Cref{Thm:colim} then implies formally that $\SpcKG$ is the quotient of the spectra for the maximal elementary abelian $p$-sections by the maximal relations.
Let us spell this out.
\end{Rem}

\begin{Cons}
\label{Cons:EI-maximal}%
Let $I$ be a finite EI-category.
The (isomorphism classes of) objects in~$I$ inherit a poset structure with $x\leq y$ if $\Hom_I(x,y)\neq\emptyset$.
Maximal objects $\Max(I)\subseteq I$ are by definition the maximal ones in this poset.
Now, let $x_1,x_2$ be two objects in~$I$, possibly equal.
The category $\Rel(x_1,x_2)$ of spans $x_1\leftarrow y\to x_2$ (or `relations') between~$x_1$ and~$x_2$, with obvious morphisms (on the~$y$ part, compatible with the spans), is also a finite EI-category and we may consider its maximal objects.
\end{Cons}

\begin{Not}
\label{Not:max-elab}%
We denote by $\maxel(G)$ the set of maximal objects in~$\EAppG$.
A word of warning: In general, there can be more maximal elementary abelian $p$-sections than just the elementary abelian $p$-sections of maximal rank.
\end{Not}

\begin{Cor}
\label{Cor:SpcKG-coeq}%
Let $G$ be a finite group.
The components~$\varphi_{(H,K)}$ of~\eqref{eq:colim-comp} induce a homeomorphism between the following coequalizer in topological spaces
\begin{equation}
\label{eq:coeq}
\operatorname{coeq}
\Bigg(
\vcenter{\xymatrix@C=5em{\displaystyle\coprod_{{E_1\xfrom{g_1}L\xto{g_2}E_2}\atop{\textup{maximal relations}}}\kern-.5cm\Spc(\cK(L))\ar@<2pt>[r]^-{\Spc(\cK(g_1))}\ar@<-5pt>[r]_-{\Spc(\cK(g_2))}&\displaystyle\coprod\limits_{E\in\maxel(G)}\kern-.5cm\SpcKE}}
\Bigg)
\end{equation}
and $\SpcKG$, for `maximal relations' in~$\EAppG$ or any variant of \Cref{Rem:reduce-colim}.
\end{Cor}
\begin{proof}
Applying \Cref{Thm:colim} we obtain
\[
\SpcKG\simeq\operatorname{coeq}\Bigg(
\vcenter{\xymatrix@C=5em{\displaystyle\coprod_{{E_1\xfrom{g_1}L\xto{g_2}E_2}}\kern-.5cm\Spc(\cK(L))\ar@<2pt>[r]^-{\Spc(\cK(g_1))}\ar@<-5pt>[r]_-{\Spc(\cK(g_2))}&\displaystyle\coprod\limits_{E\in\EAppG}\SpcKE}}
\Bigg)
\]
where~$E$ ranges over \emph{all} elementary abelian $p$-sections and $(g_1,g_2)$ over \emph{all} relations.
There is a canonical map from the coequalizer in the statement to this one and it is straightforward to produce an inverse, as with any finite EI-category.
\end{proof}

We can apply \Cref{Cor:SpcKG-coeq} to find the irreducible components of~$\SpcKG$.

\begin{Prop}
\label{Prop:SpcKG-components}%
The set of irreducible components of~$\SpcKG$ is in bijection with the set $\maxel(G)$ of maximal elementary abelian $p$-sections of~$G$ up to conjugation, via the following bijection with generic points:
\begin{align*}
  \maxel(G)_{/G}&\overset{\sim}{\longleftrightarrow}\SpcKG^0\\
  (H,K)&\longmapsto \varphi_{(H,K)}(\eta_{H/K}).
\end{align*}
In particular, $\dim(\SpcKG)=$p$\textrm{\,\rm-rank}_{\textup{sec}}(G)$ is the sectional $p$-rank of~$G$.
\end{Prop}
\begin{proof}
We use coequalizer~\eqref{eq:coeq}.
Recall from \Cref{Prop:SpcKE-irreducible} that $\SpcKE$ for an elementary abelian $p$-group~$E$ is always irreducible.
We get immediately that the map $\maxel(G)_{/G}\onto\SpcKG^0$ is a surjection.
Assume now that $\varphi_{E}(\eta_{E})=\varphi_{E'}(\eta_{E'})$ for $E,E'\in\maxel(G)$ and let us show that~$E$ and~$E'$ are conjugate $p$-sections.
By \Cref{Cor:SpcKG-coeq}, there exists a finite sequence of maximal relations responsible for the identity $\varphi_{E}(\eta_{E})=\varphi_{E'}(\eta_{E'})$ and we will treat one relation at a time.
More precisely, assuming that the generic point in~$\Spc(\cK(E_1))$ is in the image of (the map on spectra induced by) some relation $E_1\xleftarrow{g_1}L\xto{g_2}E_2$, with $E_1,E_2\in\maxel(G)$, we will show below that both~$g_i$ are conjugation isomorphisms (type~\ref{it:mor-a} in \Cref{Exas:EA-morphisms}).
In particular, $E_1,E_2$ are conjugate.
And as conjugation identifies the unique generic points in the spectra for~$E_1$ and~$E_2$ one can apply induction on the number of relations to conclude.

As the map induced by~$g_1$ is a closed immersion (\Cref{Lem:preserve-dimension}) it must be a homeomorphism once its image contains the generic point.
From this, we deduce that $g_1$ itself must be an isomorphism.
(Indeed, the map induced by restriction to a proper subgroup of~$E_1$ is not surjective, already on the cohomological open.
And similarly, the map induced by modular fixed-points with respect to a non-trivial subgroup of~$E_1$ does not even meet the cohomological open.)
Hence $L\simeq E_1$ is maximal too and therefore $g_2$ is also an isomorphism.
The only isomorphisms in~$\EAppG$ are conjugations (\Cref{Rem:EA-generating-morphisms}) and we conclude.

The second statement follows from this together with \Cref{Prop:SpcKE-dim}.
\end{proof}

\begin{Rem}
For $G$ not elementary abelian, we already saw with~$G=Q_8$ in \Cref{Exa:Q_8} that $\SpcKG$ can have larger Krull dimension than~$\Spc(\Db(\kkG))$.
And indeed, $Q_8$ has sectional $p$-rank two and $p$-rank one.
\end{Rem}

\begin{Rem}
For every maximal~$(H,K)\in\maxel(G)$, since $\varphi_{(H,K)}$ is a closed map, it yields a surjection $\varphi_{(H,K)}\colon\SpcKE\onto \adhpt{\varphi_{H,K}(\eta_{H/K})}$ from the spectrum of the elementary abelian $E=H/K$ onto the corresponding irreducible component of~$\SpcKG$. We illustrate this with $G=D_8$ in \Cref{Exa:D8} below, where said surjection coincides with the folding of \Cref{Exa:Klein-conjugation}.
\end{Rem}

\begin{Exa}
\label{Exa:D8}%
Let us compute $\Spc(\cK(D_8))$ for $G=D_8=\langle \,r,s\mid r^4=s^2=1, rs=sr\inv\,\rangle$ the dihedral group of order~$8$.
We label its subgroups as follows\,(\footnote{\,The two Klein-four subgroups are called~$K$ and~$K'$. The names $L_0$ and $L_1$ for the cyclic subgroups of~$K$ (resp.\ $L'_0$ and~$L'_1$ in~$K'$) are chosen to evoke $N_0$ and $N_1$ in \Cref{Exa:Klein-four}. The third cyclic subgroup, $N_{\infty}$, corresponds to~$C_2=Z(D_8)$ and is common to $K$ and~$K'$.}):
\[
\xymatrix@C=1.5em@R=1em{
  &&D_8\\
  &K_{}=\langle r^2,s\rangle\ar@{-}[ru]&C_4=\langle r\rangle\ar@{-}[u]&K'_{}=\langle r^2,r^3s\rangle\ar@{-}[lu]\\
  L_{0}=\langle s\rangle\ar@{-}[ru]&L_{1}=\langle r^2s\rangle\ar@{-}[u]&C_2=\langle r^2\rangle\ar@{-}[lu]\ar@{-}[u]\ar@{-}[ru]&L'_{0}=\langle rs\rangle\ar@{-}[u]&L'_{1}=\langle r^3s\rangle\ar@{-}[lu]\\
  &&1\ar@{-}[llu]\ar@{-}[llu]\ar@{-}[lu]\ar@{-}[u]\ar@{-}[ru]\ar@{-}[rru]
  }
\]
Since $L_0$ and~$L_1$ (resp.\ $L'_0$ and~$L'_1$) are $G$-conjugate, by the element~$r$, we have exactly eight very closed points~$\cM(H)$ for~$H\in \Sub{p}(G)_{/G}$. We shall focus on the open complement of these very closed points, \ie the periodic locus $\Spc'(\cK(D_8))$ of \Cref{Rem:remove-closed-pts}, which is of Krull dimension one.
Since all maps in the coequalizer diagram~\eqref{eq:coeq} preserve the dimension of points (\Cref{Lem:preserve-dimension}) we may first remove these very closed points and \emph{then} compute the coequalizer.

Let us describe $\maxel(D_8)$ and the maximal relations.
In addition to the maximal elementary abelian subgroups~$K$ and~$K'$ there is one maximal elementary abelian subquotient~$D_8/C_2$. So we have three maximal sections: $\maxel(D_8)=\{(K,1),(K',1),(D_8,C_2)\}$.
We compute the relations in the category~$\EAt{2}{D_8}$ which is obtained from~$\EA{2}{D_8}$ by quotienting each hom-set~$\hom((H,M),(H',M'))$ by the action of~$H'$, as in  \Cref{Rem:reduce-colim}.
One then easily finds by inspection five non-degenerate\,(\footnote{\,that is, not of the form $x\xfrom{\id}x\xto{\id}x$ (which would not affect the coequalizer~\eqref{eq:coeq} anyway)}) maximal relations up to isomorphism, pictured as follows:
\begin{equation}
\label{eq:D8-colimit-shape}%
\vcenter{\xymatrix@C=.5em@R=.5em{
  &&&{\color{azure(colorwheel)}(D_8,C_2)}\\
  &(K_{},C_2)\ar@[OliveGreen][rru]\ar@[Brown][ld]&&&&\ar@[OliveGreen][llu](K'_{},C_2)\ar@[Brown][rd]\\
  {\color{azure(colorwheel)}(K_{},1)}\ar@(l,lu)^{r}&&&\ar@[OliveGreen][lll](C_2,1)\ar@[OliveGreen][rrr]&&&{\color{azure(colorwheel)}(K'_{},1)}\ar@(ru,r)^{r}
}}
\end{equation}
Here, the loops labeled~$r$ represent the relations $(K_{},1)\xfrom{1}(K_{},1)\xto{r}(K_{},1)$, and similarly for~$K'$.
All unlabeled arrows are given by $1\in D_8$, as in \Cref{Exas:EA-morphisms}\,\ref{it:mor-b}-\ref{it:mor-c}.
We explain below the brown/green color-coding in the other three relations.

Hence the space $\Spc'(\cK(D_8))$ is a quotient of three copies of the space $\Spc'(\cK(E))$ for $E$ the Klein-four group, equal to $\bbP^1_{\kk}$ with three doubled points as in \Cref{fig:P1-doubled}.

Let us discuss the relations. We start with the self-relation corresponding to the loop~$r$ on~$(K_{},1)$.
As the conjugation by~$r$ on~$K_{}$ simply swaps the subgroups~$L_{0}$ and~$L_{1}$, we deduce from \Cref{Exa:Klein-conjugation} that the quotient of~$\Spc'(\cK(K_{}))$ by this relation is a~$\mathbb{P}^1_{\!\!\kk}$ with two doubled points, as in \Cref{fig:P1-2doubled}.
The same is true for~$K'$.

At this stage we have identified the three irreducible components (see \Cref{fig:D8-components}) and the three remaining relations will tell us how to glue these components.
\begin{figure}[H]
\centering
\includegraphics[scale=0.35]{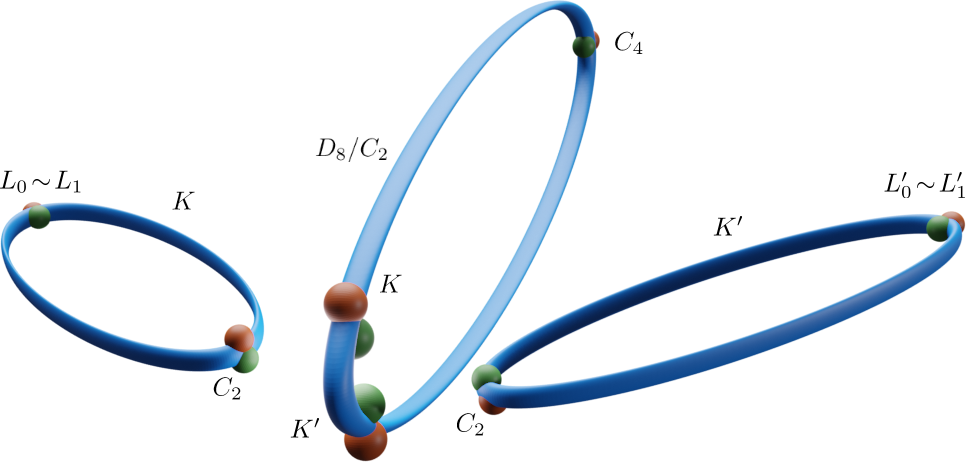}
\medbreak\caption{Three $\mathbb{P}^1_{\!\!\kk}$ with several points doubled.}
\label{fig:D8-components}
\end{figure}

The three sides of the `triangle'~\eqref{eq:D8-colimit-shape} display maximal relations that identify a single point of one irreducible component with a single point of another.
Indeed, each of the middle sections~$K/C_2$, $K'/C_2$ and~$C_2/1$ is a $C_2$, whose periodic locus is a single point~$\eta_{C_2}$ (\Cref{Exa:C_p}).
Each edge in~\eqref{eq:D8-colimit-shape} identifies the image of that single point~$\eta_{C_2}$ in the two corresponding irreducible components in~\Cref{fig:D8-components}.
The color in~\eqref{eq:D8-colimit-shape} records the color of that image: A brown point or a green point in the $\bbP^1_{\kk}$ with doubled points.
Let us do all three.
First, the relation between the two Klein-fours, $K$ and~$K'$, at the bottom of~\eqref{eq:D8-colimit-shape}, identifies the two green points corresponding to~$C_2$, as we are used to with projective support varieties.
Then, the last two relations in~\eqref{eq:D8-colimit-shape}, on the sides, identify a brown point in the $K$- or $K'$-component with the green point in the $D_8/C_2$-component corresponding to~$K_{}/C_2$ and~$K'_{}/C_2$, respectively.
This is a direct verification, for instance using that $(\psi^{C_2})\inv(\Img(\rho^{D_8}_{K}))=(\psi^{C_2})\inv(\supp_{D_8}(\kk(D_8/K)))=\supp_{D_8/C_2}(\Psi^{C_2}(\kk(D_8/K)))=\supp_{D_8/C_2}(\kk(D_8/K))=\Img(\rho^{D_8/C_2}_{K/C_2})$ in~$\Spc(\cK(D_8/C_2))$.

Thus we obtain~$\Spc'(\cK(D_8))$ from these three identifications on the space of \Cref{fig:D8-components}.
The result is the space that appeared in~\Cref{fig:mackey-perm-stab}:
\begin{figure}[H]
\centering
\includegraphics[scale=0.4]{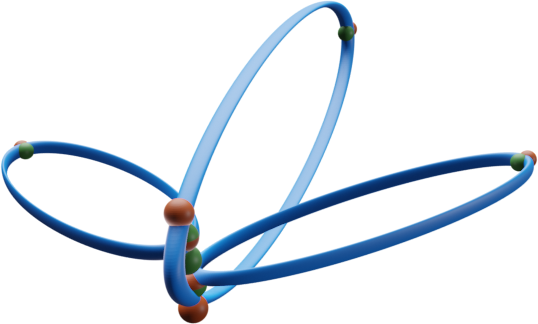}
\medbreak\caption{Three $\mathbb{P}^1_{\!\!\kk}$ with several points doubled and some identified.}
\label{fig:D8}
\end{figure}
\end{Exa}

\addtocontents{toc}{\vspace{1\baselineskip}}

\newcommand{\etalchar}[1]{$^{#1}$}

\end{document}